\documentclass[10pt]{amsart}
\usepackage{color}
\usepackage{accents}
\usepackage{amssymb}
\usepackage{enumerate}
\usepackage{url}
\usepackage[T1]{fontenc}
\usepackage[utf8]{inputenc}
\usepackage{mathtools}
\usepackage[mathscr]{euscript}

\usepackage{hyperref}\hypersetup{colorlinks}

\usepackage{color} 

\definecolor{darkred}{rgb}{1,0,0} 
\definecolor{darkgreen}{rgb}{0,0.8,0}
\definecolor{darkblue}{rgb}{0,0,1}

\hypersetup{colorlinks,
linkcolor=darkblue,
filecolor=darkgreen,
urlcolor=darkred,
citecolor=darkgreen}

\newtheorem{Theorem}{Theorem}[section] 
\newtheorem {Proposition}[Theorem]{Proposition}
\newtheorem {Lemma}[Theorem]    {Lemma}
\newtheorem {Corollary}[Theorem]{Corollary}

\theoremstyle{definition} 
\newtheorem {Conjecture}[Theorem]    {Conjecture}
\newtheorem{Definition}[Theorem]{Definition}

\theoremstyle{remark} 
\newtheorem{Example}[Theorem]{Example}
\newtheorem{Remark}[Theorem]{Remark}

\newcommand{\Aa}{{\mathcal A}}

\newcommand{\Cc}{{\mathcal C}}

\newcommand{\Jj}{{\mathcal J}}

\newcommand{\Hh}{{\mathcal H}}

\newcommand{\Mm}{{\mathcal M}}

\newcommand{\Pp}{{\mathcal P}}
\newcommand{\Rr}{{\mathcal R}}

\newcommand{\Tt}{{\mathcal T}}

\def    \R      {{\mathbb R}}

\def    \Z      {{\mathbb Z}}
\def    \N      {{\mathbb N}}

\def    \CP     {{\mathbb C}{\mathbb P}}
\def    \p      {\partial}
\def    \HF     {\operatorname{HF}}
\def    \CF     {\operatorname{CF}}

\begin{document}


\setlength{\smallskipamount}{6pt}
\setlength{\medskipamount}{10pt}
\setlength{\bigskipamount}{16pt}





\title[Rigid constellations of closed Reeb orbits]{Rigid constellations of closed Reeb orbits}

\author{Ely Kerman}
\email{ekerman@math.uiuc.edu}
\address{Department of Mathematics, University of Illinois at 
Urbana-Champaign, Urbana, IL 61801, USA}

\thanks{This work was partially supported by a grant from the Simons Foundation.}

\begin{abstract} 
We  use Hamiltonian Floer theory to recover and generalize a classic rigidity theorem of Ekelend and Lasry from \cite{el}. That theorem  can be rephrased as an assertion about the existence of multiple closed Reeb orbits for certain tight contact forms on the sphere that  are close, in a suitable sense, to the standard contact form. We first generalize this result to Reeb flows of contact forms on prequantization spaces that are suitably close to Boothby-Wang forms. We then establish, under an additional  nondegeneracy assumption,  the same rigidity phenomenon for Reeb flows on any  closed contact manifold. 

A natural obstruction to obtaining sharp multiplicity results for closed Reeb orbits is the possible existence of fast closed orbits. To complement the existence results established here, we also show that the existence of such fast orbits can not be precluded by any condition which is invariant under contactomorphisms, even for nearby contact forms.

\end{abstract}

\maketitle

\section{Introduction}

The following  theorem of Ekeland and Lasry appeared in 1980.
\begin{Theorem}\label{el}(\cite{el})
Let $\Sigma \subset \R^{2n}$ be a $C^2$-smooth hypersurface which forms the boundary of a compact convex neighborhood of the origin. If there are positive numbers $r \leq R$  such that $R < \sqrt{2}r$ and 
\begin{equation*}
\label{ }
r \leq \|x\| \leq R 
\end{equation*}
for all $x \in \Sigma$, then $\Sigma$ carries at least $n$ geometrically distinct closed characteristics with actions in $[\pi r^2, \pi R^2].$ 
\end{Theorem}
This result and the  ideas developed in its proof have been highly influential. Much of the subsequent progress on the problem of detecting closed characteristics on convex hypersurfaces, such as the remarkable results of Long and Zhu in \cite{lz},  is built on the foundation laid down in \cite{el}.  Indeed Theorem \ref{el} is still one of the most compelling facts supporting the following well known conjecture.

\begin{Conjecture}\label{n}
Every compact convex hypersurface in $\R^n$ carries at least $n$ geometrically distinct closed characteristics.\footnote{Given the results of \cite{crh} it seems reasonable to conjecture the same lower bound for star-shaped hypersurfaces.}
\end{Conjecture}

The basic idea underlying the proof of Theorem \ref{el} is the following: \textit{Closed characterisics on a convex hypersurface $\Sigma$ are critical points of Clarke's dual action principle which is invariant under the natural $S^1$-action on loops. If  $\Sigma$ satisfies the stated pinching condition, then the  $S^{2n-1}$'s worth of closed characteristics (critical points) on the sphere of radius $r$ influences the topology of the negative sublevels of the dual action principle for $\Sigma$. In particular, the  $S^1$-action is free on these sublevels and they must also contain an invariant copy of $S^{2n-1}$. This forces the restriction of the dual action  to these sublevels to have at least  $\mathrm{cuplength}(\CP^{n-1}) +1 =n$ critical points.}  

In the present  work we detect such influences using tools from Hamiltonian Floer theory. With these tools we recover Theorem \ref{el}  and generalize the rigidity phenomenon underlying it to Reeb flows on any closed contact manifold.

\subsection{Recovery} We begin by recovering Theorem \ref{el} in a different but equivalent setting. Let $\lambda$ be a contact form  on $M^{2n-1}$. It defines a unique Reeb vector field $R_{\lambda}$ on $M$ via the equations $$i_{R_{\lambda}}d\lambda=0 \quad \text{    and   }  \quad \lambda(R_{\lambda})=1.$$ 
It also defines a contact structure $\xi= \ker(\lambda)$. Any other contact form defining the same contact structure is of the form $f\lambda$ for some nonvanishing smooth function $f \colon M \to \R$.

Let $\lambda_0$ be the standard contact form on the unit sphere  $\mathbb{S}^{2n-1} \subset \R^{2n}$ obtained by restricting the form $\frac{1}{2}\sum_{j=1}^n(p_j dq_j -q_jdp_j)$. 
The following result is equivalent to Theorem \ref{el} for the case of smooth hypersurfaces.
\begin{Theorem}\label{kss}
Let $\lambda =f \lambda_0$  for some positive function $f$. If \begin{equation}
\label{less}
\frac{\max(f)}{\min(f)}<2 
\end{equation}then there are at least $n$ distinct closed orbits of $R_{\lambda}$ with periods in the interval $[\pi\min(f), \pi \max(f)].$
\end{Theorem}

Recall that a closed orbit $\gamma$  of $R_{\lambda}$ (closed Reeb orbit of $\lambda$) is said to be \textbf{distinct} from another such orbit $\widetilde{\gamma}$, if for all $k\in\N$ and $c\in\R$ there is a $t\in\R$ such that $$\gamma(t) \neq\widetilde{\gamma}(kt+c).$$ Two closed  orbits  are said to be \textbf{geometrically distinct} if they have disjoint images. While, geometrically distinct orbits are distinct the converse does not hold since two distinct orbits can both be relatively prime multiples of a third closed orbit.

To see that Theorem \ref{kss} implies Theorem \ref{el},  consider a smooth hypersurface $\Sigma$ as in the statement of Theorem \ref{el}. It can be described in the form
$$
\Sigma=\{ z\sqrt{f(z)} \mid z \in \mathbb{S}^{2n-1}\},
$$ 
where $f$ is a positive smooth function positive on $\mathbb{S}^{2n-1}$. There is a bijective correspondence between the simple  closed Reeb orbits of $\lambda =f \lambda_0$ with period $T$ (modulo translation), and the closed characteristics on $\Sigma$ with action equal to $T$. To show that Theorem \ref{el} follows from Theorem \ref{kss} it therefore suffices to show that the convexity of $\Sigma$  implies that the closed Reeb orbits of $\lambda$ with period in $[\pi \min(f), \pi \max(f)]$ are simple and hence geometrically distinct. Since $\max(f) < 2\min(f)$, it suffices to show that the convexity of $\Sigma$ imples that $R_{\lambda}$ has no closed orbits of period less than $\pi \min(f)$, or equivalently, that $\Sigma$  has no closed characteristics with action less than $\pi \min(f)$. This last condition follows immediately from the main result of Croke and Weinstein in \cite{cw} which asserts that the closed characteristics of convex hypersurfaces containing the ball of radius $r$ have action at least $\pi r^2$.  

\begin{Remark}
Note that the convexity assumption for hypersurfaces is not invariant under symplectomorphisms and that no assumption like convexity appears in statement of Theorem \ref{kss}.
\end{Remark}

\begin{Remark}  
\label{metric}
Given a contact structure $\xi$ on $M$  there is a natural coarse pseudo-metric on the space of contact forms defining $\xi$. For any two such contact forms, say $\lambda$ and $\tilde{\lambda}$  there is a smooth nonvanishing function $f_{\scriptscriptstyle{\lambda}/{\tilde{\lambda}}}$ on $M$ such that $$\lambda =f_{\scriptscriptstyle{\lambda}/{\tilde{\lambda}}} \tilde{\lambda}.$$ The {\it distance} between them can then be defined as
\begin{equation}
\label{dist}
d(\lambda,\tilde{\lambda}) = \ln \left(\frac{\max\left(f_{\scriptscriptstyle{\lambda}/{\tilde{\lambda}}}\right)}{\min\left(f_{\scriptscriptstyle{\lambda}/{\tilde{\lambda}}}\right)}\right) .
\end{equation}
It is easy to verify that $d$ is a pseudo-metric and that its degeneracy only reflects the fact that it is invariant under scalar multiplication of the contact forms. In these terms, condition \eqref{less} can be rephrased geometrically as $d(\lambda, \lambda_0) < \ln 2.$
\end{Remark}

\subsection{Prequantization spaces} 
Our first generalization of Theorem \ref{kss} establishes the same rigidity phenomenon for all prequantization spaces. Consider a closed symplectic manifold $(Q, \omega)$ such that the de Rham cohomology class 
$-[\omega]/2\pi $ is the image of an integral class $\mathrm{ \mathbf{e}} \in H^2(Q; \mathbb{Z})$.
Let $p \colon M \to Q$ be an $\mathbb{S}^1$-bundle over $Q$ with first Chern class equal to $\mathrm{ \mathbf{e}}$. Denote  the corresponding Boothby-Wang contact form on $M$ by $\lambda_Q$ and the corresponding contact structure by $\xi_Q$. We then have $d \lambda_Q = p^* \omega$ and  the Reeb vector field of $\lambda_Q$ generates the circle action on $M$, with period $2\pi$. 

Let $\alpha_\mathbf{f} \in [\mathbb{S}^1 ,M]$ be the free homotopy class corresponding to the fibres of the bundle $p \colon M \to Q$ and denote its order by $|\alpha_\mathbf{f}|$.

\begin{Theorem}\label{gutt+}
Let  $\lambda =f\lambda_Q$ for some positive function $f$. If  
$$
\frac{\max(f)}{\min(f)}< 2
$$
then there are at least $\frac{1}{2}\dim(Q)+1$ distinct closed Reeb orbits of $\lambda$ which represent the class $\alpha_{\mathbf{f}}$ and have period in the interval  $$[2\pi \min(f), 2\pi \max(f)].$$ 
These orbits are geometrically distinct from one another if the class  $\alpha_{\mathbf{f}}$ is either primitive or of infinite order. Otherwise, they are geometrically distinct if there are no  closed Reeb orbits of $\lambda$  which have period less than or equal to $$\frac{2 \pi}{|\alpha_{\mathbf{f}} |} \left(\max(f)- \min(f)\right)$$ and which represent a class $\beta$ such that $\beta^k =\alpha_{\mathbf{f}}$ for some integer $k>1$.
\end{Theorem}

Note that we can detect geometrically distinct orbits in many cases without a  geometric assumption like convexity, or an equivalent substitute like dynamical convexity.

\begin{Remark}
It is shown in \cite{ggm}, that the class  $\alpha_{\mathbf{f}}$ is of infinite order if $(Q, \omega)$ is symplectically aspherical, and is primitive if $\pi_1(Q)$ is torsion free.
Under the assumption that both these conditions hold  it is also shown in \cite{ggm} that if all the closed  Reeb orbits of $\lambda =f\lambda_Q$  are nondegenerate and $\lambda$ has no contractible closed Reeb orbits with Conley-Zehnder index within one of $2-\frac{1}{2}\dim Q$, then $\lambda$ has infinitely many closed Reeb orbits with contractible projections to $Q$.
\end{Remark}

\subsection{Rigid Constellations} We now present a more extensive generalization of the Ekeland-Lasry rigidity phenomenon. Let $(M, \lambda_0)$ be any closed contact manifold and let  $\alpha \in [\mathbb{S}^1 ,M]$ be any free homotopy class. (The trivial class will be denoted here by $e$.) We first identify collections of closed orbits  of $\lambda_0$ in class $\alpha$ which can meaningfully influence the Reeb flows of contact forms which are nearby in the sense of Remark \ref{metric}. As the original proof of Theorem \ref{el} suggests these collections should consist of simple orbits (so that the natural $\R /\Z$-action is free) and their periods should be isolated in the period spectrum. 

Let $\Tt(\lambda_0)$ be the set of periods of all closed Reeb orbits of $\lambda_0$ and, assuming it is nonempty,  let $T_{\min}(\lambda_0)$ be the smallest such period. Restricting to orbits in class $\alpha$ we get the similarly defined set  $\Tt(\lambda_0, \alpha)$ and minimal $\alpha$-period, $T_{\min}(\lambda_0, \alpha).$ 

Given a $T$ in $\Tt(\lambda_0, \alpha)$ let  $\Cc_{\lambda_0, \alpha}(T)$ be the collection of closed orbits of $R_{\lambda_0}$ which represent the class $\alpha$ and have  period  in  the interval $[T_{\min}(\lambda_0, \alpha), T]$.  
\begin{Definition}\label{rigid} The collection of closed orbits $\Cc_{\lambda_0, \alpha}(T)$ is a {\bf rigid constellation} if 
every orbit  in $\Cc_{\lambda_0, \alpha}(T)$ is simple, no decreasing sequence in $\Tt(\lambda_0, \alpha)$  converges to $T$, and 
\begin{equation}
\label{Tbound}
T <T_{\min}(\lambda_0,\alpha)+ T_{\min}(\lambda_0).
\end{equation}
\end{Definition}
Given a rigid constellation $\Cc_{\lambda_0, \alpha}(T)$ we set  $$T^+ = \min \{T' \in  \Tt(\lambda_0, \alpha)  \colon T' > T\}.$$ This number is strictly greater than $T$ by the second condition of the definition above. 


Since Reeb vector fields are autonomous, the elements of a rigid constellation $\Cc_{\lambda_0, \alpha}(T)$ can be divided into separate $\R /\Z$-families of closed Reeb orbits, the elements of which differ only by simple translation reparameterizations.  It follows from the simplicity condition in Definition \ref{rigid} that closed orbits belonging to different  $\R /\Z$-families of a rigid constellation are geometrically distinct from one another.

\begin{Example}\label{sphere} Let $\lambda_0$ be the standard contact form on the unit sphere $M= \mathbb{S}^{2n-1} \subset \R^{2n}$. The time-$t$ Reeb flow of $\lambda_0$ is $$(z_1, \dots,z_n) \mapsto (e^{i2t}z_1, \dots, e^{i2t}z_n),$$ and so every Reeb orbit is closed with (minimal) period $\pi$. For the choice
 $T =\pi$ the set $\Cc_{\lambda_0, e}(\pi)$ is then a rigid constellation (diffeomorphic to $\mathbb{S}^{2n-1}$) with  $T^+=2\pi$.
\end{Example}

\begin{Example}\label{ellipsoid}
Consider an $n$-tuple $\mathbf{r} = (r_1, \dots, r_n)$ in $\R^n$ such that 
$$0< r_1 < r_2 < \cdots < r_n< \sqrt{2}r_1.$$
Equip the unit sphere $\mathbb{S}^{2n-1}$ with the contact form $$\lambda_{\mathbf{r}}(z) =\left( \frac{|z_1|^2}{r_1^2} + \dots + \frac{|z_n|^2}{r_n^2}\right)^{-1} \lambda_0(z)$$
where $\lambda_0$ is the form from Example \ref{sphere}. The time-$t$ Reeb flow of $\lambda_{\mathbf{r}}$ is $$(z_1, \dots,z_n) \mapsto \left(e^{i2t/r_1^2}z_1, \dots, e^{{i2t}/{r_n^2}}z_n\right).$$ We then have $T_{\min} =\pi r_1^2$ and are free to choose a $T$ of the form $T_k =\pi r_k^2$ for some $k =1, \dots, n.$ For these choices we have 
$$\Cc_{\lambda_{\mathbf{r}}, e}(T_k) =\{\Gamma_1, \dots, \Gamma_k \}$$ where  each $\Gamma_i$ is the $\R /\Z$-family of closed Reeb orbits with image equal to the intersection of $\mathbb{S}^{2n-1}$ with the $z_i$-plane.  Each  $\Cc_{\lambda_{\mathbf{r}},e} (T_k)$ is a rigid constellation with 
$$
T_k^+ = \begin{cases}
  \pi r^2_{k+1}    & \text{if $k<n$ }, \\
   2 \pi r_1^2   & \text{if $k=n$}.
\end{cases}
$$
\end{Example}

\begin{Example}\label{negative} Let $(B,g)$ be a Riemannian manifold with  negative sectional curvature. Let $\Sigma_{g^*} \subset T^*B$ be the unit cosphere bundle and let $\lambda_{g^*}$ be the restriction  to $\Sigma_{g^*} $ of the  tautological  one-form on $T^*B$. Then $\lambda_{g^*} $ is a contact form whose Reeb vector field generates the cogeodesic flow on  $\Sigma_{g^*} $. Let $\tilde{\alpha}$ be a nontrivial primitive element of $[\mathbb{S}^1,B]$ and let $\alpha$ be its  lift to $[\mathbb{S}^1,\Sigma_{g^*} ]$. The  assumption of negative sectional curvature implies that there is a unique ($\R /\Z$-family of) closed Reeb orbits of $\lambda_{g^*} $ in the class $\alpha$. Thus, for  $T =T_{\min}(\lambda_{g^*} ,\alpha)$ the set  $\Cc_{\lambda_{g^*}, \alpha}(T)$ is a rigid constellation with $T^+=+\infty$.
\end{Example}

A collection of closed Reeb orbits of a contact form $\lambda$ is said to be {\bf nondegenerate} if each of its elements $\gamma(t)$ is nondegenerate (and hence isolated) in the usual sense. The rigid constellation in Example \ref{sphere} is  degenerate (Morse-Bott nondegenerate) while those described in Examples \ref{ellipsoid} and \ref{negative} are nondegenerate.

In Section \ref{floer}, we will associate to each nondegenerate rigid constellation $\Cc_{\lambda_0, \alpha}(T)$ a version of $\mathbb{S}^1$-equivariant Floer homology. With this we will define the rank of $\Cc_{\lambda_0, \alpha}(T)$. At this point  we can now state our broadest generalization of Theorem \ref{el}.

\begin{Theorem}\label{persist} Let $(M, \lambda_0)$ be a closed contact manifold and let  $\Cc_{\lambda_0, \alpha}(T)$ be a nondegenerate rigid constellation. Let $\lambda =f\lambda_0$ for some positive function $f$.  If 
\begin{equation}
\label{ass1}
\frac{\max(f)}{\min(f)} < \min\left\{\frac{T^+}{T}, \frac{T_{\min}(\lambda_0)+T_{\min}(\lambda_0 ,\alpha)}{T}\right\}
\end{equation}
and every closed Reeb orbit of $\lambda$ in class $\alpha$ and with period in the interval 
$$\left[\min(f)T_{\min}, \,\max(f)T\right]$$ is nondegenerate, then there are at least $\mathrm{rank}(\Cc_{\lambda_0,\alpha}( T)) $ such orbits which are distinct from one another. These orbits are geometrically distinct from one another if the class  $\alpha$ is either primitive or is of infinite order. Otherwise, these orbits are geometrically distinct if there are no (fast) closed Reeb orbits of $\lambda$ with period less than or equal to $$\frac{1}{|\alpha|}\left(T\max(f)-T_{\min} \min(f)\right)$$ that represent a class $\beta$ such that $\beta^k =\alpha$ for some integer $k>1$.
\end{Theorem}

The primary point of this  result is that it establishes the rigidity of closed Reeb orbits without assumptions on the ambient contact manifold. In particular, it does not assume the existence of strong symplectic fillings. The price of this generality is the presence of the term  $$\frac{T_{\min}(\lambda_0)+T_{\min}(\lambda_0 ,\alpha)}{T}.$$ While this term constrains the range of the rigidity phenomenon, it is needed  to achieve compactness for the relevant moduli spaces of Floer trajectories  used here to detect it. Despite the fact that the term is expressed in the language of dynamics  we do not know whether it represents an actual boundary to the generalized  Ekeland-Lasry rigidity  phenomenon. When $(M, \lambda_0)$ does admit strong symplectic fillings the proof of Theorem \ref{persist} simplifies greatly and yields rigidity results with larger, sometimes infinite, range. These results are described below in Section \ref{seconds}.


In Theorem \ref{persist} we have imposed  nondegeneracy assumptions on certian closed Reeb orbits of $\lambda$ while in Theorem \ref{gutt+}  no such assumptions are made. The lower bounds of Theorem \ref{gutt+}  correspond to cuplength estimates in the spirit of those predicted in the strong form of Arnold's conjecture for Hamiltonian diffeomorphisms. Theorem \ref{persist} is more in the spirit of the nondegenerate form of Arnold's conjecture. Note that rather than assuming the nondegeneracy of all closed Reeb orbits of $\lambda$ (strong nondegeneracy) we have only assumed nondegeneracy for the closed Reeb orbits of $\lambda$ in fixed range of periods. For the Ekeland-Lasry rigidity phenomenon, this seems to be the appropriate  assumption to obtain Morse-type inequalities. The difference between this nondegeneracy assumption and the strong form is analogous, in the setting of Arnold's Conjecture,  to the difference between the assumption  that the fixed points of a  Hamiltonian diffeomorphism are nondegenerate and  the  assumption  (clearly irrelevant in that case) that all its periodic points are nondegenerate.  This point its captured nicely by the contact forms $\lambda_{\mathbf{r}}$ from Example \ref{ellipsoid}. Given $\mathbf{r} =(r_1,r_2)$ with $r_1\leq r_2 < \sqrt{2}r_1$ we know that all contact forms $\lambda_{\mathbf{r}}$ have at least two closed Reeb orbits with periods
in $[\pi r_1^2, \pi r_2^2]$ and of course we would like our theorems to see this too. For $r_1=r_2$ these orbits are detected by Theorem \ref{kss}, and for $r_1<r_2$ they are detected by Theorem \ref{persist}. In contrast, contact forms $\lambda_{\mathbf{r}}$ satisfy the strong nondegeneracy assumption if and only if $r_1$ and $r_2$ are rationally independent.

\begin{Remark} Under the strong nondegeneracy assumption one can, in certain cases, detect more dramatic rigidity phenomena. This is because tools such as the index iteration formulas for Maslov indices become much sharper,  and the rich machinery of symplectic and contact homology may also become available. For example, Conjecture \ref{n} was proven for the case of convex hypersurfaces all of whose closed characteristics are strongly nondegenerate by Long and Zhu in \cite{lz} (starting from from the analytic framework of \cite{el}). More recently this was reproved under weaker \emph{dynamical} convexity assumptions by Gutt and Kang in \cite{gk} using $\mathbb{S}^1$-equivariant symplectic homology, and Abreu and Macarini  in \cite{am} using contact homology. 
\end{Remark}

\subsection{Applications (with no strong symplectic fillings)}
Theorem \ref{persist} can be applied to any Reeb flow on a closed manifold  for which one has a basic understanding of its fastest closed orbits. 
One does not need the contact manifold to admit a strong symplectic filling, nor must  one preclude the existence of certain closed Reeb orbits. These points are illustrated by the following three applications.

 \medskip

\noindent $(\mathrm{\mathbf{ A}}).$
Consider the three dimensional torus  $\mathbb{T}^3 = \R /\Z \times \R /\Z\times \R /\Z$ with angular coordinates $(x,y, \theta)$ and the familiar family of contact forms 
\begin{equation}
\label{tk}
\lambda_k = \cos(k\theta)\, dx + \sin(k\theta)\, dy
\end{equation}
for $k \in \N$.
The underlying contact structures $\xi_k =\ker{\lambda_k}$ are all weakly symplectically fillable (and hence tight). However,  Eliashberg proved in \cite{el} that for $k\geq 2$ the contact structures $\xi_k$ are not strongly symplectically fillable. Hence one can not define symplectic cohomology or linearized contact homology for the contact forms they support. Nevertheless, Theorem \ref{persist} can be applied to  all such forms.  

The Reeb vector field of $\lambda_k$ is 
$$ \cos(k\theta)\, \frac{\p}{\p x} + \sin(k\theta)\, \frac{\p}{\p y}.$$
Hence the Reeb flow is linear on each $xy$-torus. The  closed Reeb orbits  represent nonzero classes of the form $$(m,n,0) \in \Z \times \Z \times \Z =\mathrm{H}_1(\mathbb{T}^3;\Z).$$
For a fixed pair of integers $(m,n)$ these orbits all have  period $\sqrt{m^2 +n^2}$. They foliate  $k$ subtori of the form $$\mathbb{T}^2 \times \{\theta^j_{m,n}\} \subset \mathbb{T }^3$$   where $\theta^1_{m,n}, \dots, \theta^k_{m,n}$ are the solutions of the two equations
\begin{equation*}
\label{periodic1}
\cos\left( k \theta\right)=\frac{m}{\sqrt{m^2 +n^2}}
\end{equation*} 
and 
 \begin{equation*}
\label{periodic2}
\sin\left( k\theta \right)=\frac{n}{\sqrt{m^2 +n^2}}.
\end{equation*}

The orbits on these tori are simple if and only if  $m$ and $n$ are relatively prime. In this case, for $T=\sqrt{m^2 +n^2}$ 
the collection  $\Cc_{\lambda_k, (m,n,0)}(T)$ is a rigid constellation with $T^+ =+\infty$. (Here we have identified $\mathrm{H}_1(\mathbb{T}^3;\Z)$ with $[\mathbb{S}^1, \mathbb{T}^3]$ in the obvious way.) A simple perturbation argument, see for example \cite{bou},  shows that each torus contributes $2$ to the rank of $\Cc_{\lambda_k, (m,n,0)}(T)$. Hence, $\mathrm{rank}(\Cc_{\lambda_k, (m,n,0)}(T))= 2k$. With this, Theorem \ref{persist} implies the following result.

\begin{Theorem}\label{t3}
Let $\lambda =f \lambda_k$ where $f \colon \mathbb{T}^3 \to \R$ is a positive function.  Let $m$ and $n$ be relatively prime integers. If  
$$
\frac{\max(f)}{\min(f)} < \frac{\sqrt{m^2 +n^2}+1}{\sqrt{m^2 +n^2}}
$$
and the closed Reeb orbits of $\lambda$ in class $(m,n,0)$ and with period in the interval 
$$\left[\min(f)\sqrt{m^2 +n^2}, \,\max(f)\sqrt{m^2 +n^2}\right]$$ are nondegenerate, then there are at least $2k$ such orbits which are geometrically distinct from one another. 
\end{Theorem}

 \medskip

\noindent $(\mathrm{\mathbf{ B}}).$ Next we apply Theorem \ref{persist} to an overtwisted contact three manifold. Consider the manifold
$M = \R /\Z \times \R /\Z \times [0,2\pi]$ with coordinates $(x, y, t)$, and the family of contact forms
$$\eta_k =\cos\left( k t\right)\, d x + \sin\left( k t\right)\, dy$$
for $k \in \N \cup \{0\}$.
Each $\eta_k$ is invariant under the free $\R /\Z$-action generated by the vector field $\frac{\p}{\p y}$. Thus we can perform the contact cut operation, as defined by Lerman in \cite{le}, with respect to the restriction of this action to the boundary of $M$ (see Example 2.12 of \cite{le}). We obtain in this way contact forms $\widetilde{\eta}_k$ on $\R /\Z \times \mathbb{S}^2$. Here $x$ can  still be viewed as a coordinate parameterizing the $\R /\Z$-fibres of the product $\R /\Z \times \mathbb{S}^2$ and we can  identify $\widetilde{\eta}_k$ with $\eta_k$ away from the fibres over the poles of $\mathbb{S}^2$ (which correspond to $t=0$ and $t=2\pi$). As shown in \cite{le}, the contact structures $\ker(\widetilde{\eta}_k)$ are overtwisted and contactomorphic to one another for all $k>0$. Hence, the contact manifolds $( \R /\Z \times \mathbb{S}^2,\widetilde{\eta}_k)$ have no strong symplectic fillings for $k>0$.

The Reeb vector field of $\widetilde{\eta}_k$ ($\eta_k$) is
$$\cos\left( k t\right)\,  \frac{\p}{\p x} + \sin\left( k t\right)\, \frac{\p}{\p y}$$
and so the  closed Reeb orbits in class $$m \in  \Z =\mathrm{H}_1(\R /\Z \times \mathbb{S}^2;\Z)$$
occur for values of $t$ such that  
\begin{equation}
\label{periodic3}
\cos\left( k t\right)=\frac{m}{\sqrt{m^2 +n^2}}.
\end{equation} 
and 
 \begin{equation}
\label{periodic4}
\sin\left( k t\right)=\frac{n}{\sqrt{m^2 +n^2}}
\end{equation}for some $n \in \Z.$ 
For a solution $t$ of these equations located in $(0,2\pi)$,  the corresponding orbits  occur in an $\mathbb{S}^1$-family that foliates the two-dimensional $xy$-tori, $\mathbb{T}^2 \times \{t\}$,  as in the previous example. For $t=0$ and $t=2\pi$ the corresponding orbits are isolated and correspond to the fibres of $\R /\Z \times \mathbb{S}^2$ over the poles of $\mathbb{S}^2$. 

The fastest simple closed orbits of $\widetilde{\eta}_k$ all have period $2\pi$ and correspond to the four cases  $(m,n) =(\pm1,0)$ and $(m,n) =(0, \pm1)$.
For the case $(m,n)=(1,0)$ there are $k+1$  solutions  to equations \eqref{periodic3} and \eqref{periodic4},  $$0, \frac{2\pi}{k}, \dots, \frac{2\pi(k-1)}{k}, 2\pi.$$ As described above, the values $t=0$ and $t=1$ correspond to isolated simple closed Reeb orbits, and the other $k-1$ values of $t$ correspond to subtori  foliated by simple closed Reeb orbits (in class $(1,0)$). The collection of all these orbits forms a rigid constellation of rank $2k$ with $T=2\pi$ and $T^+ = 2\sqrt{2} \pi$.

For the case $(m,n)=(-1,0)$, there are $k$ solutions of equations  \eqref{periodic3} and \eqref{periodic4},  $$\frac{\pi}{k}, \frac{3\pi}{k}, \dots,2\pi- \frac{\pi}{k}.$$ These each correspond to unique subtori foliated by simple closed Reeb orbits (in class $(-1,0)$). The collection of these tori again form a rigid constellation of rank $2k$ with $T=2\pi$ and $T^+ = 2\sqrt{2} \pi$.

Finally, for the cases $(m,n)=(0,\pm 1)$, there are again $k$ solutions of equations \eqref{periodic3} and \eqref{periodic4} each of which  corresponds to a subtorus foliated by simple closed Reeb orbits which are contractible in $\R /\Z \times \mathbb{S}^2$. These also form a rigid constellation of rank $2k$, this time with $T=2\pi$ and $T^+ = 4 \pi$.

Applying Theorem \ref{persist} to these three rigid constellations we get the following rigidity theorem.

\begin{Theorem}\label{s2s1}
Consider a contact form on $\R /\Z \times \mathbb{S}^2$ of the form  $\lambda =f \widetilde{\eta}_k$ where $f$ is a positive function on $\R /\Z \times \mathbb{S}^2 $. 
\begin{enumerate}
\item[(i)] 
If $
\max(f) < \sqrt{2}\min(f)
$
and every closed Reeb orbit of $\lambda$ in class $$\pm1 \in \Z =\mathrm{H}_1(\R /\Z \times \mathbb{S}^2;\Z)$$ and with period in the interval 
$\left[\min(f)2\pi, \,\max(f)2\pi\right]$ is nondegenerate, then there are at least $2k$ such orbits which are geometrically distinct from one another. 
\item[(ii)] If $\max(f) < 2\min(f)
$
and every contractible closed Reeb orbit of $\lambda$ with period in the interval 
$\left[\min(f)2\pi, \,\max(f)2\pi\right]$ is nondegenerate, then there are at least $4k$ such orbits which are distinct from one another. They are geometrically distinct if there are no contractible  closed Reeb orbits of $\lambda$ with period less than or equal to $2\pi \left(\max(f)-\min(f)\right).$

\end{enumerate}
\end{Theorem}

 \medskip

\noindent $(\mathrm{\mathbf{C}}).$ Now we apply Theorem \ref{persist} to a family of overtwisted contact forms on $\mathbb{S}^3$. We  start again with  the manifold
$M = \R /\Z \times \R /\Z \times [0,2\pi]$ with coordinates $(x, y, t)$.
For $k \in \N \cup \{0\}$ we now consider the contact forms 
$$\zeta_k=\cos\left(\left(k+\frac{1}{4}\right) t\right)\, d x + \sin\left(\left(k+\frac{1}{4}\right) t\right)\, dy.$$
Each $\zeta_k$ is invariant under the free $\R /\Z$-actions generated by the vector fields $\frac{\p}{\p x}$ and  $\frac{\p}{\p y}$. This time we perform a different contact cut operation by choosing the $\R /\Z$-action generated by $\frac{\p}{\p y}$ at the boundary component  $\R /\Z \times \R /\Z \times \{0\}$, and the $\R /\Z$-action generated by $\frac{\p}{\p x}$ at the boundary component  $\R /\Z \times \R /\Z \times \{2\pi\}$. The resulting manifold is $\mathbb{S}^3$ and we denote the resulting contact form by $\widetilde{\zeta}_k$.  For all $k>0$ the contact structures $\ker(\widetilde{\zeta}_k)$ are overtwisted and contactomorphic to one another, (see the proof of Theorem 3.1, in \cite{le}).

Away from the fibres of the Hopf fibration that lie over the two poles in $\mathbb{S}^2$ we can identify $\widetilde{\zeta}_k$ with $\zeta_k$. Thus we can analyze the closed Reeb orbits  as before. Here 
the set of the fastest simple closed orbits of $\widetilde{\zeta}_k$ is a rigid constellation with $T=2\pi$, $T^+= 2\sqrt{2}\pi$. It is comprised of $2$ isolated orbits and $4k$, $\mathbb{S}^1$-families of closed orbits and so has rank $8k+1$. In this setting Theorem \ref{persist} implies the following result.

\begin{Theorem}\label{s3}
Consider a contact form on $\mathbb{S}^3$ of the form  $\lambda =f \widetilde{\zeta}_k$ where $f$ is a positive function.
 If $\max(f) < \sqrt{2}\min(f)
$
and every contractible closed Reeb orbit of $\lambda$ with period in the interval 
$\left[\min(f)2\pi, \,\max(f)2\pi\right]$ is nondegenerate, then there are at least $8k+2$ such orbits which are distinct from one another. They are geometrically distinct if there are no  closed Reeb orbits of $\lambda$ with period less than or equal to $2\pi \left(\max(f)-\min(f)\right).$
 
\end{Theorem}

\bigskip
\subsection{Applications (with strong symplectic fillings)} \label{seconds}
 
For contact manifolds which  admit a symplectic filling, Theorem \ref{gutt+}
and Theorem \ref{persist} both hold as stated but with improved ranges (if $\alpha$ is interpreted as a free homotopy class of the filling).
For example, if we assume that the prequantization space $(M, \lambda_Q)$ admits an exact symplectic filling, then 
Theorem \ref{gutt+} holds, as stated, with the larger upper bound  
$$
\frac{\max(f)}{\min(f)}< |\alpha_{\mathbf{f}}| +1. 
$$
Similarly, assuming there is an exact symplectic filling of $(M, \lambda_0)$, 
Theorem \ref{persist} holds, as stated, with the upper bound $$
\frac{\max(f)}{\min(f)}< \frac{T^+}{T}. 
$$ As the next set of applications demonstrates, this version of Theorem \ref{persist} allows one to view a variety of previously known rigidity phenomena from a single perspective. 
\bigskip

\noindent $(\mathrm{\mathbf{D}}).$ Applying Theorem \ref{persist} to the contact forms in Example \ref{ellipsoid} one can recover the various pinching theorems of Berestycki, Lasry, Mancini and Ruf established in \cite{blmr} up to some extra nondegeneracy assumptions.

\medskip

\noindent $(\mathrm{\mathbf{E}}).$ Returning to the setting of Example \ref{negative}, consider a  closed Riemannian manifold $(B,g)$ with  negative sectional curvature and the corresponding unit cosphere bundle  $\Sigma_{g^*} \subset T^*B$. Let $\tilde{\alpha}$ be a nontrivial primitive element of $[\R /\Z,B]$ and denote  its lift to $[\R /\Z,\Sigma_{g^*} ]$ by $\alpha$. As described in Example \ref{negative} the rigid constellation $\Cc_{\lambda_{g^*}, \alpha}(T_{\min}(\lambda_{g^*}, \alpha))$ has $T^+ =+ \infty$. It also has rank $1$ and so  Theorem \ref{persist} implies that every hypersurface $\Sigma$ in $T^*B$ which is fibrewise star-shaped about the zero-section must carry at least one  closed characteristic in class $\alpha$. A similar but stronger result in this direction is proved in \cite{bps} as Corollary 3.4.2 (see also \cite{hv}).

\medskip

\noindent $(\mathrm{\mathbf{F}}).$ Let  $g_n$ be the standard flat metric on the $n$-dimensional torus $\mathbb{T}^n$ . Denote its cosphere bundle by $\Sigma_{g_n^*} \subset T^*\mathbb{T}^n$ and the restriction of the tautological one-form to $\Sigma_{g_n^*}$ by  $\lambda_{g_n^*}$. Every nontrivial primitive class 
 $\tilde{\alpha}$ in  $[\R /\Z,\mathbb{T}^n]$  lifts to a class $\alpha$ in  $[\R /\Z,\Sigma_{g_n^*} ]$ which is again nontrivial and primitive. There is a unique $\mathbb{T}^n$-family of closed Reeb orbits in class $\alpha$. Denoting their period by $T(\alpha)$, the constellation $\Cc_{\lambda_{g_n^*}, \alpha }(T(\alpha))$ is  rigid, with  $T^+ =+\infty$ and $\mathrm{rank}(\Cc_{\lambda_{g_n^*}, \alpha }( T))=2^n$. With this, Theorem \ref{persist} implies the following.
\begin{Theorem} 
Let  $\Sigma$ be a smooth hypersurface in $T^*\mathbb{T}^n$ which is fibrewise star-shaped about the zero-section. Let $\lambda_{\Sigma}$ be the restriction of the tautological one-form to $\Sigma$.  If the closed Reeb  orbits of $\lambda_{\Sigma}$  in class $\alpha$ are all nondegenerate then there must be at least $2^n$ such obits which are geometrically distinct. 
\end{Theorem}
This complements a theorem of Cielebak from \cite{ci1}, which establishes the existence of at least $$\left\lfloor  \frac{n}{2}\right\rfloor +1$$ geometrically distinct closed Reeb orbits with no nondegeneracy assumption.

 \medskip

\noindent $(\mathrm{\mathbf{G}}).$
The finite collection of prime closed geodesics in Katok's famous examples of Finsler metrics from \cite{ka}
also yield a rich source of rigid constellations. Consider for example a Katok metric $g_K$ on the sphere $P=\mathbb{S}^{2n}$ or $\mathbb{S}^{2n-1}$. It has precisely $2n$ prime closed geodesics and can be constructed so that for any $\epsilon \in (0,1)$ the  longest prime geodesic  has length $1+\epsilon$ and the shortest has length $1-\epsilon$.  For $\epsilon<\frac{1}{2}$ and $T= 1+\epsilon$ the collection of prime closed geodesics is then a  rigid constellation with $T^+ = 2(1-\epsilon)$ and rank $2n$. If, as above, we let $\Sigma_{g_K^*} \subset T^*P$ be the cosphere bundle of $g_K$, then for any other hypersurface $\Sigma$ of  $T^*\mathbb{S}^{2n}$ which is fibrewise star-shaped about the zero-section there 
 is a a unique smooth function $f_{\scriptscriptstyle{\Sigma}} \colon \Sigma_{g_K^*} \to (0, \infty)$ such that $$\Sigma=\{(q,f_{\scriptscriptstyle{\Sigma}}(q,p) p) \colon (q,p) \in \Sigma_{g_K^*}\}.$$ Let $\lambda_{\Sigma}$ be the restriction of the tautological one-form to $\Sigma$. In this setting Theorem \ref{persist} implies the following result.
 
\begin{Theorem}\label{katoky}
If 
\begin{equation*}
\label{ }
\frac{\max(f_{\scriptscriptstyle{\Sigma}})}{\min(f_{\scriptscriptstyle{\Sigma}})}< \frac{2(1-\epsilon)}{1+\epsilon}
\end{equation*}
and all the closed Reeb orbits of $\lambda_{\Sigma}$ with period in the interval $$\left[\frac{2\pi\min\big(f_{\scriptscriptstyle{\Sigma}}\big)}{1+\epsilon}, \frac{2\pi\max\big(f_{\scriptscriptstyle{\Sigma}}\big)}{1-\epsilon}\right]$$ are nondegenerate, then $\lambda_{\Sigma}$ has at least $2n$ distinct closed Reeb orbits with periods in this interval. These orbits are geometrically distinct if there are no closed Reeb orbits of  $\lambda_{\Sigma}$ with period less than or equal to the length of this interval.
\end{Theorem}

If one restricts Theorem \ref{katoky} to hypersurfaces that are  cosphere bundles of Finsler metrics, it then yields multiplicity results for closed geodesics which complement recent work of Rademacher in \cite{ra} and of Wang in \cite{wa6}. 

\bigskip

\subsection{Fast Orbits}
The possible presence of periodic orbits with small periods sometimes obstructs our ability to conclude that the distinct orbits detected by Theorem \ref{gutt+} and Theorem \ref{persist} are in fact geometrically distinct. This is a fundamental, and somewhat notorious, difficulty common to such multiplicity problems. In certain settings one can preclude the existence of such fast orbits by imposing geometric restrictions,  like the convexity of the hypersurfaces in \cite{el}.
It is natural then to ask if one can find conditions on the contact structure which preclude fast orbits and are invariant under contactomorphisms. We prove here that this is not possible. This is one implication of our final  theorem which can also be viewed as a {\it soft} compliment to Theorem \ref{gutt+} and Theorem \ref{persist}.
 
\begin{Theorem}\label{fast}
Let $(M, \lambda_0)$ be a  contact manifold. For any free homotopy class $\alpha \in [\mathbb{S}^1 ,M]$, and any positive constants $c_1, \,c_2 >0$, there is a contact  form $\lambda=f \lambda_0$ on $M$ such that $\min(f)=1$, 
$
\max(f)<1+c_1
$
and $\lambda$ has a closed Reeb orbit in class $\alpha$ of period less than $c_2.$
\end{Theorem}

There are some intriguing phenomena hidden in the gap between the construction underlying Theorem \ref{fast} and, say, Theorem \ref{persist}. For example, in creating the nearby form $\lambda$ in Theorem \ref{fast} with an arbitrarily fast closed orbit we are forced, by the  Reeb condition, to cede all control over the number and basic properties of any additional closed Reeb orbits we might also create in the process. This follows from work of Rechtman in \cite{re} and is discussed further in Remark \ref{plug}

\subsection{Related Works} 
Our proof of Theorem \ref{gutt+} is motivated by the method to obtain cuplength estimates in Floer theory that was introduced by Albers and Momin in \cite{am} and further developed by Albers and Hein in \cite{ah}.

The proof of Theorem \ref{persist} has two main parts. The first involves the  construction of  a version of Hamiltonian Floer theory for rigid constellations. In this construction we use several results concerning the Hamiltonian Floer theory of autonomous Hamiltonians established by 
Bourgeois and Oancea in \cite{bo}. To obtain the $C^0$-bounds for our Floer trajectories (which allows us to do away with fillings) we also adapt an argument of Albers, Fuchs and Merry from \cite{afm}. In the second part of the proof we use the Floer thoeretic tools developed in the first to adapt an argument of Chekanov from \cite{ch} to detect the desired closed orbits. This is based on previous work of the author from \cite{ke}.

Similar ideas to those underlying Theorem \ref{persist} were developed by Jean Gutt in his thesis  \cite{gu1} and subsequent paper  \cite{gu2}. \footnote{The author is grateful to Peter Albers for notifying him of Gutt's thesis when the author spoke of the results presented here at the Lorentz Centre in July 2014.} In these works, Gutt shows that positive $\mathbb{S}^1$-equivariant symplectic homology can be used as a contact invariant for a certain class of fillable contact manifolds that can be realized as the boundary of Liouville domains. Among the many interesting applications of his theory, Gutt reproves Theorem \ref{el} under the additional strong nondegeneracy assumption, and also proves a result (Theorem 1.6 in \cite{gu2}) very similar in content to Theorem \ref{gutt+} here.  
Happily, besides a shared debt owed to the technical foundations for $\mathbb{S}^1$-equivariant Hamiltonian Floer theory laid down by Bourgeois and Oancea,  this is essentially the extent of the overlap between the two projects. 

The construction of a Reeb semi-plug described in the proof of Theorem \ref{fast} is similar in several details to Cieliebak's construction of a confoliation-type plug from \cite{ci2}. The goals of the two constructions diverge at an early stage, however. In keeping within the class of Reeb flows we sacrifice here the possibility, achieved in \cite{ci2},  of realizing the insertion of our plugs on fillable contact manifolds, as the result  of a symplectomorphism acting on the interior of a filling.  

\section*{Acknowledgements}
The author wishes to thank Miguel Abreu, Peter Albers, Viktor Ginzburg, Richard Hind, Eugene Lerman, Ana Rechtman, EJ Sanchez and Yang Song for helpful  comments, discussions and suggestions.

\section{Hamiltonian Floer theory and rigid constellations}\label{floer}

Let  $\Cc_{\lambda_0,\alpha}( T)$ be a rigid constellation for the contact form $\lambda_0$ on $M$. 
In this section we develop a version of Hamiltonian Floer theory adapted to $\Cc_{\lambda_0,\alpha}( T)$. 


\subsection{Admissible Families of Functions}
 
Denoting the $\R$-coordinate of $\R \times M$ by $\tau$, we consider the symplectization
$$
(\R \times M,  d(e^{\tau} \lambda_0)).
$$

\begin{Definition}
A smooth function $H \colon \R \times M \to \R$ is an  {\bf admissible} Hamiltonian if it is nonnegative, satisfies 
\begin{equation}
\label{back}
H(\tau,p)=0 \text{   for all   } \tau \leq 0 
\end{equation}
and if there is a $\mathbf{T}>0$ such that 
\begin{equation}
\label{front}
dH_{(\tau,p)}=0  \text{   for all   }  \tau \geq \mathbf{T}.
\end{equation}
\end{Definition}

The Hamiltonian vector field, $V_H$ of $H$  is defined by the equation
\begin{equation*}
\label{ }
i_{V_H} \omega = -dH.
\end{equation*}
Conditions  \eqref{back} and \eqref{front} imply that the support of $dH$ is compact, and so the flow of $V_H$ is defined for all times. 

Let $x(t)$ be a $1$-periodic orbit of $H$ (i.e., $V_H$). We define the action of $x(t)$  by 
\begin{eqnarray*}
\Aa_H(x) & = & -\int_{\R /\Z} x^*(e^{\tau} \lambda_0)+ \int_{\R /\Z}H(x(t)) \, dt.
\end{eqnarray*}

Let $\Pp^-(H)$ be the set of all $1$-periodic orbits of $H$ which have negative action. Since $H$ is nonnegative, every $x$ in $\Pp^-(H)$ is nonconstant. We say that $H$ is  {\bf nondegenerate} if every $x$ in $\Pp^-(H)$ is also  transversally nondegenerate. Strict nondegeneracy is impossible since 
 $H$ is  autonomous. In particular every nonconstant $1$-periodic orbit $x(t)$ of $H$ belongs to an $\R /\Z$-family of such orbits which we will denote by $X$. 
The elements of $X$ all have the same action and so the notation $\Aa_H(X)$ can and will be used unambiguously.  The collection of all $\R /\Z$-families of $1$-periodic orbits of $H$ with negative actions will be denoted by $\Pp^-_{\R /\Z}(H)$.

The {\bf action gap} of a nondegenerate admissible Hamiltonian $H$ is defined to be
\begin{equation*}
\label{ }
\Delta(H) = \sup_{X, \,Y \in  \Pp^-_{\R /\Z}(H)} \left\{\Aa_H(X)-\Aa_H(Y)\,  \colon \, \Aa_H(X)\neq \Aa_H(Y) \right\}
\end{equation*}
where we set $\Delta(H)=0$ if there are no families $X$ and $Y$ meeting the required conditions (for example if $dH$ is sufficiently $C^1$-small). Given two nondegerate admissible Hamiltonians  $H^0$ and $H^1$ we set 
\begin{equation*}
\label{ }
\Delta(H^0, H^1) = \sup_{} \left\{\Aa_{H^0}(X^0)-\Aa_{H^1}(X^1) \,  \colon \, \Aa_{H^0}(X^0) \neq \Aa_{H^1}(X^1) \right\} 
\end{equation*}
where here the supremum is over pairs $(X^0, X^1) \in \Pp^-_{\R /\Z}(H^0) \times \Pp^-_{\R /\Z}(H^1)$ and we again set $\Delta(H^0, H^1)=0$ if no relevant pairs exist. 

An {\bf admissible homotopy} from  $H^0$ to $H^1$ is  a smooth family of functions $H^s$  for $s \in \mathbb{R}$, such that  for some $\mathbf{S}>0$ and $\mathbf{T}>0$ we have:
\begin{enumerate}
  \item[($\mathrm{H}^s$1)]  $H^s =H^0$ for all $s \leq -\mathbf{S}$
  \item[($\mathrm{H}^s$2)]  $H^s =H^1$ for all $s \geq \mathbf{S}$
  \item[($\mathrm{H}^s$3)]  for all $s \in \R$ the support of $d(H^s)$ is contained in  $[0,\mathbf{T}] \times M$ 
 \end{enumerate}  
The {\bf cost} of the homotopy $H^s$ is defined to be
\begin{equation*}
\label{ }
\mathrm{cost}(H^s) =\int_{\R} \max_{(\tau,\,p)} \left(\p_s\big(H^s(\tau,p)\big)\right) \, ds.
\end{equation*}

Finally, an {\bf admissible homotopy of homotopies}  between $H^0$ and $H^1$ is  a smooth $\mathbb{R}^2$-family of functions $H^{r,s}$ such that for each $r \in \R$,  $H^{r,s}$ is an admissible homotopy from $H^0$ to $H^1$ and the supports of all the $d(H^{r,s})$ are contained in a single neighborhood of the form $[0,\mathbf{T}] \times M$.
The cost of  $H^{r,s}$ is defined as 
\begin{equation*}
\label{ }
\mathrm{cost}(H^{r,s}) =\int_{\R} \max_{\substack{(\tau,\,p,\,r) }} \left(\p_s\big(H^{r,s}(\tau,p)\big)\right) \, ds.
\end{equation*}

\subsection{Dividing and Tuned Hamiltonians}

 Let $H$ be an admissible Hamiltonian of the form $$H(\tau, p) = h(e^{\tau}).$$ We will refer to $h$ as the {\bf profile} of $H$. 
The Hamiltonian vector field  of $H$ is 
$$V_H(\tau, p) = h'(e^{\tau})R_{\lambda_{0}}(p).$$
Thus,  every nonconstant $1$-periodic orbit $x(t)$  of $V_H$ corresponds to a unique closed Reeb orbit $\gamma_x(t)$ of $\lambda_0$ such that  
\begin{equation}
\label{x0}
x(t) = (\tau_x, \gamma_x(h'(e^{\tau_x})t)).
\end{equation}
In particular,  $h'(e^{\tau_x})$ is the period of the orbit $\gamma_x$ which we will also denote by $T_{\gamma_x}$.
The  action of $x(t)$ is 
\begin{equation}
\label{critical0}
\Aa_H(x) = -e^{\tau_x}  h'(e^{\tau_x})+ h(e^{\tau_x}).
\end{equation}
Note also that the  $\R /\Z$-family of $1$-periodic orbits of $V_H$  containing $x$, $X$, corresponds to the  unique  $\R /\Z$-family of closed Reeb orbits of $\lambda_0$ containing $\gamma_x$.  Denoting this family of Reeb orbits by $\Gamma_X$  we have  $$\Gamma_X =\bigcup_{x \in X}\gamma_x.$$

We now define a class of Hamiltonians which have useful collections of $1$-periodic orbits (related to useful collections of closed Reeb orbits) that  can be identified simply by the fact that their actions are negative.  For positive constants $a,b>0$ and $c>1+a$ let $\mathfrak{h}_{a,b,c}$ be the space of smooth profile functions $h\colon \R \to \R$ with the following properties
\begin{enumerate}
  \item[(h1)]  $h(s)=0 \text{ for } s\leq1$.
  \item[(h2)] $h''(s)>0$ for $s \in (1,1+a)$.
   \item[(h3)] $h(1+a)=a^2$.
  \item[(h4)] $h'(s)=b$ for  $s \in[1+a, c]$.
  \item[(h5)] $h''(s)<0$ for $s \in (c,c+a)$.
  \item[(h6)] $h(s)=b(c-1-a)+a^2$ for  $s\geq c+a$.
\end{enumerate}
Let  $\Rr^b(\lambda_0)$ be the set  of closed Reeb orbits of $\lambda_0$ with period less than $b$.
\begin{Lemma}\label{divide} Suppose that $H(\tau,p)=h(e^{\tau})$ for some profile $h$ in $\mathfrak{h}_{a,b,c}$.    
If $b$  is not the period of a closed Reeb orbit of $\lambda_0$ and $c$ is sufficiently large, then  every  $x \in \Pp^-(H)$  is nonconstant and of the  form $$x = (\tau_x, \gamma_x(h'(e^{\tau_x})t))$$ for some $e^{\tau_x}$ in $(1, 1+a)$ and  some $\gamma_x$ in $ \Rr^b(\lambda_0)$. Moreover, the correspondence $x \rightarrow \gamma_x$ defines a bijection between $\Pp^-(H)$ and $\Rr^b(\lambda_0)$.
 \end{Lemma}

\begin{proof}
Every point $(\tau,p)$ in $\R \times M$ with $e^{\tau} \leq 1$ or   $e^{\tau} \geq c+a$ corresponds to a constant periodic orbit of $H$. The constant orbits in the first region have action ($H$-value) equal to zero by (h1). The orbits in the second region have strictly positive action ($H$-value) by conditions (h3)-(h6). Thus neither set contributes to $\Pp^-(H)$

It follows from our choice of $b$ that the nonconstant periodic orbits of $H$ are of the  form $$x(t) = (\tau_x, \gamma_x(h'(e^{\tau_x})t))$$ for some $e^{\tau_x}$ in either $(1, 1+a)$ or $(c, c+a)$. As mentioned above, the action of such an orbit is
\begin{equation}
\Aa_H(x)= -e^{\tau_x}h'(e^{\tau_x})+ h(e^{\tau_x}).
\end{equation}
These action values correspond to values of the function $$F(s)= -sh'(s) +h(s).$$ 
By (h1), we have $F(1)=0$ and by (h2)  we have $F'(s)=-sh''(s)<0$ in $(1,1+a)$. Thus, all the nonconstant periodic orbits $x$ with  $e^{\tau_x}$ in $(1, 1+a)$ have negative action and so appear in $\Pp^-(H)$. 

On the other hand, given a closed Reeb orbit $\gamma$ of $\lambda_0$ with period $T_{\gamma}<b$ it follows from (h2) and (h4) that there is a unique solution  $\tau_{\gamma}$  of
$$
h'(e^{\tau}) = T_{\gamma}
$$
contained in the interval $(1,1+a)$. Then 
\begin{equation*}
\label{ }
x(t) = (\tau_{\gamma}, \gamma(T_{\gamma}t))
\end{equation*}
belongs to $\Pp^-(H)$.

To complete the proof it just remains to show that a nonconstant periodic orbit $x$ as above with  $e^{\tau_x}$ in $(c, c+a)$ has positive action for all large enough values of $c$.
For such an $x$ we have  
\begin{equation*}
\Aa_H(x) > -(c+a)h'(e^{\tau_x})+ (c-1-a)b.
\end{equation*}
and thus
\begin{equation}
\label{action-up0}
\Aa_H(x) > c(b-T_{\gamma_x})-b(2a+1).
\end{equation}
Since the period spectrum,  $\Tt(\lambda_0)$,  of $\lambda_0$ is closed and $b$ lies outside it, the quantity
\begin{equation*}
\label{ }
b- \max\{t \in \Tt(\lambda_0)\colon t<b\}
\end{equation*}
is positive.  From \eqref{action-up0} we then  get 
\begin{equation*}
\label{}
\Aa_H(x) > c\Big( b- \max\{t \in \Tt(\lambda_0)\colon t<b\}\Big)-b(2a+1).
\end{equation*}
Hence, for all sufficiently large $c>0$ we  have $\Aa_H(x)>0$ for all $x\in \mathrm{Crit}(\Aa_H)$ with $e^{\tau_x} \in (c, c+a)$, as desired. 

\end{proof}

Let $\mathcal{H}_{a,b,c,\kappa}$ be the space of smooth functions of the form $$H(\tau, p)= h(e^{\tau -\kappa})$$ where $\kappa \geq 0$ and $h$ is in $\mathfrak{h}_{a,b,c}$. For an $H$ in $\mathcal{H}_{a,b,c,\kappa}$  we have 
$$V_{H}(\tau,p) = h'(e^{\tau-\kappa})e^{-\kappa}R_{\lambda_0}(p) $$
and for every nonconstant $1$-periodic orbit $x(t)$  of $V_{H}$ there is a unique closed Reeb orbit $\gamma_x(t)$ of $\lambda_0$ such that  
\begin{equation}
\label{x}
x(t) = (\tau_x, \gamma_x(h'(e^{\tau_x-\kappa})e^{-\kappa}t)) 
\end{equation}
and 
\begin{equation}
\label{critical formula}
\Aa_{H}(x) = -e^{\tau_x} T_{\gamma_x}+ h(e^{\tau_x-\kappa}).
\end{equation}

Arguing as above we then get the following.

 \begin{Lemma}\label{k-divide} 
Let $H$ be in $\mathcal{H}_{a,b,c,\kappa}$. If  $be^{-\kappa}$ is not in $\Tt(\lambda_0)$ and $c$ is sufficiently large, then
 every  $x \in \Pp^-(H)$  is nonconstant and of the  form $$x(t) = (\tau_x, \gamma_x(h'(e^{\tau_x-\kappa})e^{-\kappa}t)) $$ for some $e^{\tau_x} \in (e^{\kappa}, (1+a)e^{\kappa})$ and  some $\gamma_x \in \Rr^{be^{-\kappa}}(\lambda_0)$. Moreover, the correspondence $x \rightarrow \gamma_x$ defines a bijection between $\Pp^-(H)$ and $\Rr^{be^{-\kappa}}(\lambda_0)$.
 \end{Lemma}

\begin{Definition} A  function $H$ as in Lemma \ref{k-divide} will be called a {\bf dividing} Hamiltonian.\end{Definition}


%
%

\begin{Remark}\label{general divide} Note that, changing all the $\lambda_0$'s above  to $\lambda$'s, we can use the same definition for the notion of dividing functions which detect collections of closed Reeb orbits of $\lambda$. 
\end{Remark}

\begin{Definition}
\label{tuned} 
A Hamiltonian $H$ is {\bf tuned}  to the rigid constellation $\Cc_{\lambda_0,\alpha}( T)$ if it belongs to some $\mathcal{H}_{a,b,c, \kappa}$ for which  
\begin{enumerate}
  \item[(t1)] $e^{\kappa} < \min\left\{ \frac{T^+}{T}, \, \frac{T_{\min}(\lambda_0)+ T_{\min}(\lambda_0, \alpha)}{T}\right\}$,
  \item[(t2)] $T e^{\kappa}<b<T^+$,
  \item[(t3)] $c> \frac{2b}{b-T e^{\kappa}}$.
 \end{enumerate}
\end{Definition}

Every  $H$ tuned to $\Cc_{\lambda_0,\alpha}( T)$ is dividing. So the sets $\Pp^-(H)$ and $\Rr^{be^{-\kappa}}(\lambda_0)$ are in bijection. The class $\alpha \in [\mathbb{S}^1 ,M]$ determines a unique class in $[\R /\Z, \R \times M]$, which we will also denote by $\alpha$. Let $\Pp_{\alpha}^-(H)$ be the  subset $\Pp^-(H)$
 consisting of $1$-periodic orbits of $H$ which have negative action and which represent the class $\alpha$. The point of Definition \ref{tuned} is that,  if $H$ is tuned to $\Cc_{\lambda_0,\alpha}( T)$, then the elements of $\Pp_{\alpha}^-(H)$ correspond precisely to the elements of  $\Cc_{\lambda_0,\alpha}( T)$, i.e., 
 \begin{equation}
\label{121}
x\in\Pp_{\alpha}^-(H) \longleftrightarrow \gamma_x \in \Cc_{\lambda_0,\alpha}( T). 
\end{equation}

To identify tuned Hamiltonians that are  suitable for defining Floer theory on the symplectization of $(M,\lambda_0)$ we must refine this notion further.
It is clear from equation \eqref{critical formula} that for a tuned $H$ in $\mathcal{H}_{a,b,c,\kappa}$ with a small value of $a$, the actions of the orbits in $\Pp_{\alpha}^-(H)$ are close to $-e^{\kappa}$ times the period of an element of $\Cc_{\lambda_0,\alpha}( T)$. Developing this further we get  the following useful set of inequalities.

\begin{Lemma}\label{spectra} 
For a  rigid constellation $\Cc_{\lambda_0,\alpha}( T)$ there is a positive constant $\bar{a}>0$ such that the following hold:

\begin{enumerate}
  \item If $H$ in $\mathcal{H}_{a,b,c,\kappa}$ is tuned to $\Cc_{\lambda_0,\alpha}( T)$  and $a< \bar{a}$, then 
  \begin{equation}
\label{low}
\Aa_H(x) < -T_{\min}(\lambda_0, \alpha)/2
\end{equation}
for every $x \in \Pp_{\alpha}^-(H)$. 
  \item  If $H$ in $\mathcal{H}_{a,b,c,\kappa}$ is tuned to $\Cc_{\lambda_0,\alpha}( T)$  and $a< \bar{a}$, then 
 \begin{equation}
\label{delta-h}
\Delta(H) < T_{\min}(\lambda_0). 
\end{equation}

  \item If  $H^0 \in \mathcal{H}_{a^0,b^0,c^0,\kappa^0}$ and $H^1\in \mathcal{H}_{a^1,b^1,c^1,\kappa^1}$
are tuned to $\Cc_{\lambda_0,\alpha}( T)$ and   $a^0,\, a^1 < \bar{a}$,  then   
\begin{equation}
\label{delta-h0h1}
\Delta(H^0, H^1) < T_{\min}(\lambda_0).
\end{equation}
\end{enumerate}

\end{Lemma}

\begin{proof}
By equation \eqref{critical formula} and the properties of $h$  we have 
\begin{equation}
\label{eps}
-e^{\kappa}(1+a)T <\Aa_H(x) < -e^{\kappa} T_{\min}(\lambda_0, \alpha) +a^2.
\end{equation}
for every $x \in \Pp_{\alpha}^-(H)$. The fact that inequality \eqref{low} holds for sufficiently small $a$, follows immediately from this
and the fact that $\kappa \geq 0$. The inequalities of  \eqref{eps} also imply that  
\begin{equation}
\label{dh1}
\Delta(H) <  e^{\kappa}(T -T_{\min}(\lambda_0, \alpha))+a e^{\kappa}(T-a).
\end{equation}
Since  $\Cc_{\lambda_0,\alpha}(T)$ is rigid we have $T< T_{\min}(\lambda_0, \alpha) + T_{\min}(\lambda_0).$
Together with condition (t1), this yields 
\begin{equation}
\label{dh2}
e^{\kappa} < \frac{T_{\min}(\lambda_0)}{T -T_{\min}(\lambda_0, \alpha)}
\end{equation}
The fact that inequality \eqref{delta-h} holds for sufficiently small $a$, now follows immediately from \eqref{dh1} and \eqref{dh2}.
Finally,  \eqref{eps} also implies that
\begin{eqnarray}
\Delta(H^0,H^1) & \leq & -e^{\kappa^0}T_{\min}(\lambda_0, \alpha) + e^{\kappa^1}T + (a^0)^2 + e^{\kappa^1} a^1 T\\
{} & \leq  &  e^{\kappa^1}T -T_{\min}(\lambda_0, \alpha) + (a^0)^2 + e^{\kappa^1} a^1 T.
\end{eqnarray}
This, together with (t1), implies that \eqref{delta-h0h1} holds when both  $a^0$ and $a^1$ are sufficiently small.

\end{proof}

\begin{Definition} A  function $H$ in $\mathcal{H}_{a,b,c,\kappa}$ tuned to the rigid constellation $\Cc_{\lambda_0,\alpha}( T)$ is said to  be {\bf finely tuned} to $\Cc_{\lambda_0,\alpha}( T)$ if $a< \bar{a}$ and hence inequalities \eqref{low}, \eqref{delta-h} and \eqref{delta-h0h1} hold. \end{Definition}

\subsection{Almost complex structures}\label{acs} For the next four subsections we will assume that $H$ is an admissible Hamiltonian and that each element of $\Pp^-(H)$, that is each $1$-periodic orbit  of $H$ with negative action,  
is transversally nondegenerate. Given such an $H$ we now define a useful class of almost complex structures on $\R \times M$. Recall first that an almost complex structure $J$ on the symplectization $(\R \times M, d(e^{\tau}\lambda))$ is said to be {\bf cylindrical} if it is invariant under $\tau$-translations and satisfies $J\left(\frac{\p}{\p\tau}\right) =R_{\lambda}$. The related notion of being cylindrical on subsets of the form 
$\{\tau \leq \mathbf{T}\}$ or $\{\tau \geq \mathbf{T}\}$ is defined in the obvious way.

Denote by  $\Jj(H)$  the set of smooth almost complex structures on $\R \times M$ with the following properties:
\begin{enumerate}
  \item[(J1)] $J$ is compatible with $d(e^{\tau} \lambda_0)$.
  \item[(J2)]   $J=J_0$ on $\{\tau \leq 0\}$ where $J_0$ is fixed and cylindrical. 
    \item[(J3)] $J$ is cylindrical on   $\{\tau \geq \mathbf{T}\}$ for some $\mathbf{T}>0$.
  \item[(J4)] For any point $z= (\tau,p)$ on the image of a family $X$ in $\Pp^-_{\R /\Z}(H)$
we have $$[V_H, JV_H](z) \neq 0  \text{   and   }  [V_H, JV_H](z) \notin \mathrm{Span}\{V_H(z), JV_H(z)\}.$$
\end{enumerate}


 \begin{Lemma}(\cite{bo}, see proof of Prop. 3.5 (i) in \S 4)
 The set $\Jj(H)$ is a nonempty open subset of the set of all smooth almost complex structures  with properties (J1)-(J3).
 \end{Lemma}

\subsection{Floer trajectories: Transversality} Consider a pair   $$F=(H,J)$$ consisting of an admissible Hamiltonian $H$ and an almost complex structure  $J$ in $\Jj(H)$. We will refer to $F$ as a Floer data set and will denote the set of all Floer data sets  by $\mathbf{ F}$.

Given two (nonconstant) $\R /\Z$-families  $X$ and $Y$ in $\Pp^-_{\R /\Z}(H)$ and an $F=(H,J)$ in $\mathbf{F}$ we define
$$
\widehat{\Mm}(X, Y;\,F)
$$
to be the space of solutions  $u \colon \R \times \R /\Z \to \R \times M$ of 
\begin{equation*}
\label{ }
\partial_su +J(u)(\partial_tu-V_{H}(u))=0
\end{equation*}
which satisfy the asymptotic conditions
\begin{equation*}
\label{ }
\lim_{s \to -\infty } u(s,t) \in X,   \quad  \lim_{s \to +\infty } u(s,t) \in Y, \quad \text{and} \quad \lim_{s \to \pm\infty } \partial_su(s,t) =0
\end{equation*}
where the convergences are all uniform in $t$. The following  transversality statement for these spaces is established by Bourgeois and Oancea in \cite{bo}. 
 \begin{Proposition}\label{trans}(\cite{bo}, Proposition 3.5 (i)) There is a subset $\Jj_{reg}(H)$ of $\Jj(H)$ of second category such that for any $J \in \Jj_{reg}(H)$ and any pair of $\R /\Z$-families $X,Y \in \Pp^-_{\R /\Z}(H)$, such that the orbits of either $X$ or $Y$ are simple,  each set
 $\widehat{\Mm}(X, Y;\,F)$  for $F=(H, J)$ is a smooth finite dimensional manifold. 
 \end{Proposition}
 
The assumption that  either $X$ or $Y$ is simple implies that the elements of  $\widehat{\Mm}(X, Y;\,F)$ are all somewhere injective. Having thus avoided the fundamental difficulty of dealing with multiply covered maps, transversality can then be established in the  manner of \cite{hs}. The subtle point, observed and overcome in \cite{bo}, is that  condition (J4) can be used to prove that the set of injective points in the domain of an element of $\widehat{\Mm}(X, Y;\,F)$ constitute an open and dense subset of some neighborhood of an end asymptotic to a simple orbit.
Let $\mathbf{F}_{\mathrm{reg}}$ be the subset of $\mathbf{F}$ consisting of Floer data sets  $F=(H,J)$ with $J \in \Jj_{reg}(H)$.

Next we consider spaces of Floer continuation trajectories. Let  $H^s$ be an admissible homotopy between admissible Hamiltonians $H^0$ and $H^1$ whose nonconstant $1$-periodic orbits  with negative action are transversally nondegenerate. Let $J^s$ be a smooth family of $d(e^{\tau} \lambda_0)$-compatible almost complex structures such that  for some $\mathbf{S}>0$ and $\mathbf{T}>0$ we have:
\begin{enumerate}
  \item[($\mathrm{J}^s$1)]  $J^s =J^0 \in \Jj(H^0)$ for all $s \leq -\mathbf{S}$.
  \item[($\mathrm{J}^s$2)]  $J^s =J^1 \in \Jj(H^1)$ for all $s \geq \mathbf{S}$.
   \item[($\mathrm{J}^s$3)] $J^s=J_0$ on $(-\infty,0] \times M$  for all $s\in\R$.
  \item[($\mathrm{J}^s$4)] $J^s$ is cylindrical (for $\lambda_0$) on $[\mathbf{T},+\infty) \times M$ for all $s\in\R$.
 \end{enumerate}
We refer to the pair $F^s=(H^s,J^s)$ as Floer continuation data (connecting $F^0=(H^0, J^0)$ to $F^1=(H^, J^1)$) and will denote the set of all such triples  as $\mathbf{F}^s =\mathbf{F}^s(F^0,F^1)$. 

For an $F^s =(H^s,J^s)$ in $\mathbf{F}^s$ and families $X^0$ in $\Pp^-_{\R /\Z}(H^0)$ and $X^1$ in $\Pp^-_{\R /\Z}(H^1)$,  let 
$$
\widehat{\Mm}_s(X^0, X^1;\,F^s)
$$
be the space of solutions  $u \colon \R \times \R /\Z \to \R \times M$ of 
\begin{equation*}
\label{ }
\partial_su +J^s(u)(\partial_tu-V_{H^s}(u))=0
\end{equation*}
which satisfy the asymptotic conditions
\begin{equation*}
\label{ }
\lim_{s \to -\infty } u(s,t) \in X^0 ,   \quad  \lim_{s \to +\infty } u(s,t) \in X^1, \quad \text{and} \quad \lim_{s \to \pm\infty } \partial_su(s,t) =0.
\end{equation*}

Arguing as above one gets the following basic transversality statement.
\begin{Proposition}\label{s-trans}
Suppose $F^0$ and $F^1$ are in $\mathbf{F}_{\mathrm{reg}}$. Then there is a subset $\mathbf{F}^s_{\mathrm{reg}}$ of $\mathbf{F}^s(F^0, F^1)$ of second category such that for any $F^s \in \mathbf{F}^s_{\mathrm{reg}}$ and any  families $X^0 \in \Pp^-_{\R /\Z}(H^0)$  and $X^1 \in \Pp^-_{\R /\Z}(H^1)$, such that the orbits of either $X^0$ or $X^1$ are simple,  each 
 $\widehat{\Mm}_s(X^0, X^1;\,F^s)$ is a smooth finite dimensional manifold.
 \end{Proposition}

More generally, we have the moduli space of Floer trajectories corresponding to an admissible homotopy of homotopies, $H^{r,s}$, from $H^0$ to $H^1$. Let $J^{r,s}$ be a smooth $\R^2$-family of $d(e^{\tau} \lambda_0)$-compatible almost complex structures such that for some smooth positive function $S(r)$ and for some $\mathbf{T}>0$ the following condition holds 

\medskip
\begin{enumerate}
  \item[($\mathrm{J}^{r,s}$1)]  for each $r$, the $\R$-family $J^{r,s}$ satisfies  ($\mathrm{J}^s1$)-($\mathrm{J}^s4$) for $S=S(r)$.
 \end{enumerate}
\medskip

Following the pattern above, we refer to the triple $F^{r,s}=(H^{r,s}, J^{r,s})$ as Floer homtopy data and will denote set of all such triples  by $\mathbf{F}^{r,s}$. 
For an $F^{r,s} =(H^{r,s},J^{r,s})$ in $\mathbf{F}^{r,s}$ and  two families $X^0$ in $\Pp^-_{\R /\Z}(H)$, $X^1$ in $\Pp^-_{\R /\Z}(H^1)$ let
$$
\widehat{\Mm}_{r,s}(X^0, X^1;\,F^{r,s})=\{(r,u) \colon r\in \R,\, u \in \widehat{\Mm}_s(X^0, X^1;\,F^{r,s})\}.
$$
\begin{Proposition}\label{rs-trans}
There is a subset $\mathbf{F}^{r,s}_{\mathrm{reg}}$ of $\mathbf{F}^{r,s}$ of second category such that for any $F^{r,s} \in \mathbf{F}^{r,s}_{\mathrm{reg}}$ and any  families $X^0 \in \Pp^-_{\R /\Z}(H)$  and $X^1 \in \Pp^-_{\R /\Z}(H^1)$ such that the orbits of either $X^0$ or $X^1$ are simple,   each 
 $\widehat{\Mm}_{r,s}(X^0, X^1;\,F^{r,s})$ is a smooth finite dimensional manifold.
 \end{Proposition}

\subsection{Floer trajectories: $C^0$-bounds} With transversality in hand,  we now turn to compactness. Since the manifold $\R \times M$ is open, we must first establish $C^0$-bounds. The positive end of this manifold  ($\tau \to +\infty$) is never a possible source of noncompactness. Since every (family of) Hamiltonian(s) we consider is constant for $\tau \gg 0$ and every (family of) almost complex structure(s) we consider is cylindrical for $\tau \gg 0$ the maximal principle forbids our curves from entering these regions. It remains for us to deal with the negative end of $\R \times M$.

Before proceeding we recall the relevant notions of energy in this context and some useful equalities and inequalities involving them. 
For $F=(H,J)$ the $L^2$-energy of each $u \in \widehat{\Mm}(X, Y; F)$ is defined to be 
\begin{equation}
\label{energy}
E(u) = \int_{\R \times \R /\Z} d(e^{\tau} \lambda_0)(\p_s u, J(u) \p_s u) \, ds\, dt.
\end{equation}
The following well-known identity then follows from Stokes' Theorem and the definition of $\widehat{\Mm}(X, Y; F)$,
\begin{equation}
\label{energy-action}
E(u) = \Aa_{H}(X) - \Aa_{H}(Y).
\end{equation}
Hence, we have  
\begin{equation}
\label{energy-bound}
E(u) \leq \Delta(H)
\end{equation}
for any $u$ in any $\widehat{\Mm}(X, Y; F)$.

Similarly, for $F^s=(H^s,J^s)$ the $L^2$-energy of $u$ in
$
\widehat{\Mm}_s(X^0, X^1;\,F^s)
$
is defined to be
\begin{equation}
\label{s-energy }
E_s(u) = \int_{\R \times \R /\Z} d(e^{\tau} \lambda_0)(\partial_su, J^s(u) \partial_s u) \, ds\,dt.
\end{equation}
In this case, Stokes' theorem yields
\begin{equation}
\label{sr-energy-action}
E_s(u)= \Aa_{H^0}(X^0) - \Aa_{H^1}(X^1) + \int_{\R \times \R /\Z} \left(\p_sH^s\right)(u(s,t)) \, ds\,dt
\end{equation}
and 
\begin{equation}
\label{s-energy-action}
E_s(u) \leq  \Delta(H^0,H^1)+ \mathrm{cost}(H^s). 
\end{equation}
Finally, for $F^{r,s}=(H^{r,s}, J^{r,s})$ and  $(r,u)$ in
$
\widehat{\Mm}_{r,s}(X^0, X^1;\,F^{r,s}) 
$
 we have
\begin{equation}
\label{rs-energy-action}
E_{r,s}((r,u)) = \int_{\R \times \R /\Z} d(e^{\tau} \lambda_0)(\partial_su, J^{r,s}(u) \partial_s u) \, ds\,dt \leq   \Delta(H^0,H^1)+ \mathrm{cost}(H^{r,s}).
\end{equation}

We now prove that for a fixed $H$ we have uniform $C^0$-bounds for the elements of the spaces $\widehat{\Mm}(X, Y;\,F)$.
Recall that  $T_{\min}(\lambda_0)$ is the  smallest period of any closed Reeb orbit of $\lambda_0$.

\begin{Proposition}\label{C0}
Suppose  that $H$ is an admissible Hamiltonian and that $X$ and $Y$ are transversally nondegenerate families in $\Pp^-_{\R /\Z}(H)$. If $$
\Aa_{H}(X) - \Aa_{H}(Y) < T_{\min}(\lambda_0), $$ then there is a $K>0$ such that for any choice of Floer data of the form $F= (H, J)$ the image of every $u \in \widehat{\Mm}(X, Y;\,F)$ is contained in $[-K,+\infty)\times M \subset \R \times M$.
\end{Proposition}

\begin{proof}

In terms of the product structure of $\R \times M$, any $u \in \widehat{\Mm}(X, Y;\,F)$ can be written in the form  $$u(s,t) = (\rho(s,t), \xi(s,t)).$$
Arguing by contradiction we assume that  there is  a sequence of almost-complex structures $J_k$ in $\Jj(H)$,  and  a sequence of curves  $$u_k(s,t) = (\rho_k(s,t), \xi_k(s,t))$$
in $\widehat{\Mm}(X, Y;\,F_k)$  for $F_k = (H,J_k)$ such that 
\begin{equation}
\label{assumption2}
\lim_{k \to \infty}\left(\min_{(s,t) \in \R\times \R /\Z} \rho_k(s,t) \right) = -\infty.
\end{equation}

To obtain the desired contradiction we will argue as in \cite{afm} (see also \cite{ehs}). 
We begin by isolating purely $J_0$-holomorphic portions of the $u_k$. Fix  a decreasing sequence  $\epsilon_k \searrow0$ such that $u_k$ is transverse to $\{-\epsilon_k \} \times M$ for all $k \in \N$ and  set $$V_k = u_k^{-1}((-\infty, -\epsilon_k]\times M).$$ By \eqref{assumption2} we may assume, by passing to a subsequence if necessary, that each $V_k$ is nonempty. Let  $v_k = u_k|_{V_k}$.  Since $H$ is admissible and $J_k$ belongs to $\Jj(H)$, each  $v_k$ is a $J_0$-holomorphic curve with a possibly disconnected domain and image in $\{\tau \leq 0\}$. For these curves we have 
\begin{equation*}
\label{v-energy}
\int_{V_k}v_k^* d(e^{\tau}\lambda_0)  = \int_{V_k} d(e^{\tau}\lambda_0)(\partial_su_k, J_k(u_k) \partial_s u_k) \, ds\, dt <E(u_k).
\end{equation*}
and so, by equation \eqref{energy-action} we have
\begin{equation}
\label{v-energy}
 \int_{V_k}v_k^*d(e^{\tau}\lambda_0)< \Aa_{H}(X) - \Aa_{H}(Y).
\end{equation}


The Hofer energy of $v_k$ is
$$
E_{\mathrm{Hofer}}(v_k) =  \sup_{\phi } \int_{V_k} v_k^*d(\phi \lambda_0)
$$
where the supremum is over all functions $\phi$ in  $C^{\infty}(\R, [0,1])$ that are nondecreasing.

Applying Stokes' Theorem twice we get 
\begin{eqnarray*}
E_{\mathrm{Hofer}}(v_k) & = &  \sup_{\phi} \int_{V_k} v_k^*d(\phi \lambda_0)\\
{} & = & \sup_{\phi} \int_{\partial V_k} v_k^*(\phi \lambda_0)\\
{} & = & \int_{\partial V_k} v_k^* \lambda_0\\
{} & = & e^{\epsilon_k}\int_{\partial V_k} v_k^*( e^{\tau}\lambda_0)\\
{}& = &   e^{\epsilon_k}\int_{V_k} v_k^*d (e^{\tau}\lambda_0)\\
{}& = &    e^{\epsilon_k}\int_{V_k} v_k^*d(e^{\tau}\lambda_0).
\end{eqnarray*}
Hence, by \eqref{v-energy}, we have 
\begin{equation*}
\label{ }
E_{\mathrm{Hofer}}(v_k ) < e^{\epsilon_k}(\Aa_{H}(X) - \Aa_{H}(Y)).
\end{equation*}
%
So, we have a sequence, $v_k$, of $J_0$-holomorphic curves  in the symplectization $(\R \times M, d(e^{\tau} \lambda_0))$ whose Hofer energies are uniformly bounded from above, which all intersect the hypersurface $\{-\epsilon_1\} \times M$, and whose $\R$-components have minima that converge to negative infinity. This suggests that the curves break at $\tau=-\infty$ along closed Reeb orbits of $\lambda_0$ with period less than the limiting energy bound. 
Indeed this is the case. This is a consequence of Theorem 5.3 of \cite{afm}, which utilizes the compactness argument from \cite{cm} and yields the following precise statement in the present setting.

\begin{Proposition}[Theorem 5.3, \cite{afm}]
There is a subsequence $k_n$ and cylinders $C_n \subset U_{k_n}$ that are biholomorphically equivalent to the standard cylinders $[-L_n, L_n] \times \R /\Z$ such that the lengths $L_n \to \infty$ and the curves $v_{k_n}|_{C_n}$
converge in  $C^{\infty}_{loc}(\R\times \R /\Z, \R \times M)$ to a trivial cylinder over a closed Reeb orbit of $\lambda_0$ of period at most $\Aa_{H}(X) - \Aa_{H}(Y)$. 
\end{Proposition}

\noindent The existence of this closed Reeb orbit of $\lambda_0$ implies that  
$
T_{\min}(\lambda_0) \leq \Aa_{H}(X) - \Aa_{H}(Y)
$ 
and we have arrived at the desired contradiction. 
\end{proof}

Starting from the uniform bounds \eqref{s-energy-action} and \eqref{rs-energy-action}, and arguing as above one also obtains  $C^0$-bounds for moduli spaces of the form $\widehat{\Mm}_s(X^0, X^1;\,F^s)$ and $
\widehat{\Mm}_{r,s}(X^0, X^1;\,F^{r,s})$ whenever $\Aa_{H^0}(X^0) - \Aa_{H^1}(X^1) + \mathrm{cost}(H^s)< T_{\min}(\lambda_0)$ and $\Aa_{H^0}(X^0) - \Aa_{H^1}(X^1) + \mathrm{cost}(H^{r,s})< T_{\min}(\lambda_0)$, respectively.

\subsection{Floer trajectories: Quotients and Compactifications}\label{quotient} Consider a regular Floer data set $F =(H, J)$ in $\mathbf{F}_{\mathrm{reg}}$ and two distinct families $X$ and $Y$ in $\Pp^-_{\R /\Z}(H)$  at least one of which is simple. Since both $H$ and $J$ do not depend on $t$, it follows that  the $\R \times \R /\Z$-action on  $\widehat{\Mm}(X, Y;\,F)$ given by  
\begin{equation*}
\label{ }
(s,t) * u(\cdot,\cdot) = u(\cdot+s,\cdot+t)
\end{equation*}
is free. The quotient
$$
\Mm(X, Y;\,F)=\widehat{\Mm}(X, Y;\,F)/ (\R \times \R /\Z).
$$
is then a smooth manifold. Let  $\Mm^k(X, Y;\,F)$ be the submanifold of $\Mm(X, Y;\,F)$ which consists of all its components which have dimension $k$. Given Proposition \ref{C0}, the follow compactness statements are then standard.

\begin{Proposition}\label{break-glue} Suppose $\Aa_{H}(X) - \Aa_{H}(Y)<T_{\min}(\lambda_0)$.  Then $\Mm^0(X, Y;\,F)$ is a compact manifold of dimension zero. If, in addition, both $X$ and $Y$ are simple,  then $\Mm^1(X, Y;\,F)$ admits a compactification $\overline{\Mm}^1(X, Y;\,F)$ which is a $1$-dimensional  manifold with boundary equal to
\begin{equation*}
\label{ }
\bigcup_{Z \in  \Pp^-_{\R /\Z}(H)} \Mm^0(X, Z;\,F) \times  \Mm^0(Z, Y;\, F).
\end{equation*}
\end{Proposition}

Now consider two admissible and nondegenerate Hamiltonians $H^0$ and $H^1$, regular Floer data  $F^0=(H^0,J^0)$ and $F^1=(H^1,J^1)$  and  regular Floer continuation data $F^s \in\mathbf{F}^s(F^0,F^1)$ in $\mathbf{F}^s_{\mathrm{reg}}$. For families $X^0 \in \Pp^-_{\R /\Z}(H^0)$ and $X^1 \in \Pp^-_{\R /\Z}(H^1)$, at least  one of which is simple,  the manifold  $\widehat{\Mm}_s(X^0, X^1;\,F^s)$ admits a  free $\R /\Z$-action,
\begin{equation*}
\label{ }
t * u(\cdot,\cdot) = u(\cdot,\cdot+t)
\end{equation*}
The quotient 
$$
\Mm_s(X^0, X^1;\,F^s)=\widehat{\Mm}_s(X^0, X^1;\,F^s)/  (\R /\Z).
$$
is then a smooth manifold and we let 
$\Mm_s^k(X^0, X^1;\,F^s)$ be the collection of its $k$-dimensional components. In this case we get the following compactness result.

\begin{Proposition}\label{s-break-glue}
Suppose  $\Aa_{H^0}(X^0) - \Aa_{H^1}(X^1) + \mathrm{cost}(H^s)< T_{\min}(\lambda_0)$. Then $\Mm^0_s(X^0, X^1;\,F^s)$ is compact.
If, in addition, both $X^0$ and $X^1$ are simple and $\mathrm{cost}(H^s) < -\Aa_{H}(X^0)$,  then $\Mm^1_s(X^0, X^1;\,F^s)$ admits a compactification $\overline{\Mm}^1_s(X^0, X^1;\,F^s)$ which is a $1$-dimensional  manifolds whose boundary is  
\begin{eqnarray*}
{} & {} & \bigcup_{Y^0 \in \Pp^-_{\R /\Z}(H^0)} \Mm^0(X^0, Y^0;\,F^0) \times  \Mm^0_s(Y^0, X^1;\,F^s) \\
{} & \cup & \bigcup_{Y^1 \in \Pp^-_{\R /\Z}(H^1)} \Mm^0_s(X^0, Y^1;\,F^s) \times \Mm^0(Y^1, X^1;\,F^1). 
\end{eqnarray*}
\end{Proposition}
\noindent Note that the condition $\mathrm{cost}(H^s) < -\Aa_{H}(X^0)$ is needed to ensure that the orbits $Y_0$ that appear in the expression above have negative action.

Finally, for regular Floer homotopy data $F^{r,s}=(H^{r,s},J^{r,s}) \in\mathbf{F}^{r,s}(F^0,F^1)$ we set 
$$
\Mm_{r,s}(X^0, X^1;\,F^{r,s})=\widehat{\Mm}_{r,s}(X^0, X^1;\,F^s)/  (\R /\Z),  
$$
and define $\Mm_{r,s}^k(X^0, X^1;\,F^{r,s})$ as above. For $k=0$ we get the following.

\begin{Proposition}\label{sr-break-glue}
If  $\Aa_{H^0}(X^0) - \Aa_{H^1}(X^1) + \mathrm{cost}(H^{r,s})< T_{\min}(\lambda_0)$  then $\Mm^0_{r,s}(X^0, X^1;\,F^{r,s})$ is a compact manifold of dimension zero.
\end{Proposition}
\noindent For $k=1$ there are two important versions of the relevant compactness statement to state. 

\medskip
\noindent{\bf Version 1: A closed homotopy of homotopies.} We say that the  homotopy data $F^{r,s}$ is closed if for some $\mathbf{R}>0$
\begin{equation*}
\label{ }
F^{r,s}=\begin{cases}
     F^{0,s} \in \mathbf{F}^s_{\mathrm{reg}}(F^0,F^1)& \text{for all $r \leq -\mathbf{R}$}, \\
     F^{1,s}  \in \mathbf{F}^s_{\mathrm{reg}}(F^0,F^1)& \text{for all $r \geq \mathbf{R}$}.
\end{cases}
\end{equation*}
\begin{Proposition}\label{chain homotopy}
Suppose that $F^{r,s} =(H^{r,s}, J^{r,s})$ is regular and closed and that $\Aa_{H^0}(X^0) - \Aa_{H^1}(X^1) + \mathrm{cost}(H^{r,s})< T_{\min}(\lambda_0)$. If $X^0$ and $X^1$ are both simple and $\mathrm{cost}(H^{r,s}) < |\Aa_{H}(X^0)|$, then $\Mm^1_{r,s}(X^0, X^1;\,F^{r,s})$ admits a compactification $\overline{\Mm}^1_{r,s}(X^0, X^1;\,F^{r,s})$ which is a $1$-dimensionsal manifold with boundary  equal to
\begin{eqnarray*}
{} & {} & \Mm^0_s(X^0, X^1;\,F^{0,s}) \cup  \Mm^0_s(X^0, X^1;\,F^{1,s})\\
{} & \cup & \bigcup_{Y^0 \in \Pp^-_{\R /\Z}(H^0)} \Mm^0(X^0, Y^0;\,F^0) \times  \Mm^{0}_{r,s}(Y^0, X^1;\,F^{r,s}) \\
{} & \cup & \bigcup_{Y^1 \in \Pp^-_{\R /\Z}(H^1)} \Mm^{0}_{r,s}(X^0, Y^1;\,F^{r,s}) \times \Mm^0(Y^1, X^1;\,F^1). 
\end{eqnarray*}
\end{Proposition}
\medskip
\noindent{\bf Version 2: A half-open homotopy of homotopies.} We now consider a more explicit homotopy of homotopies with an open end. Let 
$H^0$, $H^1$ and $G$ be nondegenerate and admissible Hamiltonians. Consider  two admissible homotopies,   $H_0^s$ from
from $H^0$ to $G$, and $H_1^s$  
from $G$ to $H^1$. Now consider a homotopy of homotopies of the form 
\begin{equation*}
\label{ }
H_{0\#1}^{r,s} = \begin{cases}
   H_0^{s+ \xi(r)}   & \text{for $s\leq 0$}, \\
    H_1^{s-\xi(r)}  & \text{for $s>0$}.
\end{cases}
\end{equation*}
where $\xi(r)$ is a smooth positive and nondecreasing function which equals $r$ for $r \gg \xi(0)$ and which equals $\xi(0)$ for $r \leq \xi(0)/2$.
It is easy to check that this is an admissible homotopy of homotopies from $H^0$ to $H^1$ if we choose $\xi(0)$ to be sufficiently large.

Fixing regular Floer data sets $F^0=(H^0,J^0)$, $F^1=(H^1,J^1)$ and $F^G=(G, J^G)$ we extend these to Floer continuation data sets
\begin{equation*}
\label{ }
F_0^s= (H_0^s, J_0^s) \in \mathbf{F}^s(F^0, F^G)
\end{equation*}
and 
\begin{equation*}
\label{ } 
F_1^s= (H_1^s, J_1^s) \in \mathbf{F}^s(F^G, F^1),
\end{equation*}
which we use to form  the Floer homotopy data set
\begin{equation*}
\label{ }
F_{0\#1}^{r,s} = \begin{cases}
  F_0^{s+ \xi(r)}   & \text{for $s\leq 0$}, \\
    F_1^{s-\xi(r)}  & \text{for $s>0$}
\end{cases}
\end{equation*}
in $\mathbf{F}^{r,s}(F^0, F^1)$. 
Perturbing these, if necessary, we may assume the three previous data sets are all regular.

\begin{Proposition}\label{composition}
Suppose that $\Aa_{H^0}(X^0) - \Aa_{H^1}(X^1) + \mathrm{cost}(H^{r,s})< T_{\min}(\lambda_0)$. If $X^0$ and $X^1$ are both simple and $\mathrm{cost}(H_{0\#1}^{r,s}) < -\Aa_{H}(X^0),$  then  the manifold  $$\Mm^1_{r,s}(X^0, X^1;\,F_{0\#1}^{r,s})$$ admits a compactification $$\overline{\Mm}^1_{r,s}(X^0, X^1;\,F_{0\#1}^{r,s})$$ which is a $1$-dimensional  manifold whose boundary is  
\begin{eqnarray*}
{} & {} & \Mm^0_s(X^0, X^1;\,F_{0\#1}^{0,s})\\
{} & \cup & \bigcup_{Z \in \Pp^-_{\R /\Z}(G)} \Mm^0_s(X^0, Z;\,F^s_0) \times  \Mm^0_s(Z, X^1;\,F^s_1)\\
{} & \cup & \bigcup_{Y^0 \in \Pp^-_{\R /\Z}(H^0)} \Mm^0(X^0, Y^0;\,F^0) \times  \Mm^{0}_{r,s}(Y^0, X^1;\,F_{0\#1}^{r,s}) \\
{} & \cup & \bigcup_{Y^1 \in \Pp^-_{\R /\Z}(H^1)} \Mm^{0}_{r,s}(X^0, Y^1;\,F_{0\#1}^{r,s}) \times \Mm^0(Y^1, X^1;\,F^1). 
\end{eqnarray*}
\end{Proposition}

\subsection{Floer theory for finely tuned Hamiltonians} 
Throughout this section we consider  a fixed nondegenerate rigid constellation $\Cc_{\lambda_0,\alpha}( T)$.

\subsubsection{Homology}  To every Hamiltonian  $H$ which is finely tuned to $\Cc_{\lambda_0,\alpha}( T)$ and every regular Floer data set  $F=(H, J) \in \mathbf{F}_{\mathrm{reg}}$ we associate a version of Floer homology. Let $\Pp^-_{\alpha, \,\R /\Z}(H)$ be the set of $\R /\Z$-families of closed $1$-periodic orbits of $H$ which have negative action and which represent the class $\alpha$. Since $H$ is tuned, each family $X$ in $\Pp^-_{\alpha, \,\R /\Z}(H)$ corresponds to a unique  $\R /\Z$-family of closed Reeb orbits $\Gamma_X$ in $\Cc_{\lambda_0,\alpha}( T)$, and vice versa. Since each family $\Gamma_X$ in $\Cc_{\lambda_0,\alpha}( T)$ is simple, so are the families $X$ in $\Pp^-_{\alpha, \,\R /\Z}(H)$.

Define  the chain group by 
$$\CF(H; \alpha) = \mathrm{Span}_{\mathbb{Z}/2\mathbb{Z}}\{X \colon X \in \Pp^-_{\alpha, \,\R /\Z}(H)\}$$
and the corresponding boundary map $\p_F \colon \CF(H; \alpha)\to \CF(H; \alpha)$ to be the linear operator defined on generators by 
$$\p_F(X) = \sum_{Y \in \Pp^-_{\alpha, \,\R /\Z}(H)} \#\Mm^0(X, Y;\,F) \,Y. $$
Here, and in what follows,  for any finite set $\Mm$ the notation $\#\Mm$ will denote the number of elements modulo 2. 
By the definition of {\it finely tuned} we have $\Delta(H) <  T_{\min}(\lambda_0)$ (see Lemma \ref{spectra}). It then follows from Proposition \ref{break-glue} that $\p_F$ is well-defined and satisfies $\p_F\circ \p_F = 0$. Standard arguments imply that the resulting homology is independent of the choice of regular $J \in \Jj(H)$, and so we denote this  homology by 
$
\HF(H;\alpha).
$

\subsubsection{Continuation maps} Let $H^0$ and $H^1$ both be finely tuned to $\Cc_{\lambda_0,\alpha}(T)$. We now construct tools which allow us to compare
$
\HF(H^0;\alpha)
$
and
$
\HF(H^1;\alpha).
$
Let
\begin{equation}
\label{homotopy}
\Delta_s(H^0,H^1) = \min\left\{\frac{T_{\min}(\lambda_0, \alpha)}{2}, \,T_{\min}(\lambda_0) -\Delta(H^0,H^1) \right\}.
\end{equation}
By Lemma \ref{spectra}, $\Delta_s(H^0,H^1) > 0.$ 
Consider an admissible homotopy  $H^s$  from $H^0$ to $H^1$ such that 
\begin{equation}
\label{homotopy}
\mathrm{cost}(H^s) < \Delta_s(H^0,H^1).
\end{equation}
Perturbing $H^s$ if necessary, we choose regular Floer continuation data $F^s=(H^s,J^s)$ between regular Floer data $F^0 =(H^0,J^0)$ and  $F^1=(H^1, J^1)$ and define the linear map 
\begin{equation*}
\label{ }
\theta_{F^s} \colon \CF(H^0; \alpha) \to \CF(H^1; \alpha)
\end{equation*}
on generators by
$$\theta_{F^s}(X^i) = \sum_{Y^j \in \Pp^-_{\alpha,\, \R /\Z}(H^1)} \#\Mm_s^0(X^i, Y^j;\,F^s) \,Y^j.$$
Proposition \ref{s-break-glue} implies that
$\theta_{F^s}$ is a well-defined chain map. The usual  arguments again imply that the resulting map in homology is independent of the
choice of $J^s$ and so we denote this map by 
\begin{equation*}
\label{ }
\Theta_{H^s} \colon \HF(H^0; \alpha) \to \HF(H^1; \alpha).
\end{equation*}


A convex combination of two homotopies that satisfy \eqref{homotopy} also satisfies the same bound. Hence,  the usual homotopy of homotopies argument in Floer theory can be used to show that the map $\Theta_{H^s}$ does not depend on the choice of homotopy   $H^s$ with cost less than $\Delta_s(H^0,H^1)$.

 \begin{Lemma}
 \label{same}
If  $H^s$ and $\widetilde{H}^s$ are two admissible homotopies from $H^0$ to $H^1$ with cost less than $\Delta_s(H^0, H^1)$, then the maps $\Theta_{H^s}$ and  $\Theta_{\widetilde{H}^s}$,  
  are equal.
 \end{Lemma}

In the present setting, the usual composition rule for continuation maps has the following form.

\begin{Lemma}
\label{split}
Suppose that $H^0$, $H^1$ and $H^2$ are  Hamiltonians  that are finely tuned to $\Cc_{\lambda_0,\alpha}(T)$ and that  $H^s_{10}$ is an admissible homotopy from $H^0$ to $H^1$  and $H^s_{21}$ is an admissible homotopy from $H^1$ to $H^2$. If 
\begin{equation}
\label{ }
\mathrm{cost}(H^s_{10})<\Delta_s(H^0, H^1)
\end{equation}
\begin{equation}
\label{ }
\mathrm{cost}(H^s_{21})<\Delta_s(H^1, H^2)
\end{equation}
and 
\begin{equation}
\label{ }
\mathrm{cost}(H^s_{10})+ \mathrm{cost}(H^s_{21})<\Delta_s(H^0, H^2).
\end{equation}
Then there is an admissible homotopy $H_{20}^s$ from $H^0$ to $H^2$ with cost at most $\mathrm{cost}(H^s_{10})+ \mathrm{cost}(H^s_{21})$ such that 
\begin{equation*}
\label{ }
\Theta_{H^s_{20}} = \Theta_{H^s_{21}} \circ \Theta_{H^s_{10}}.
\end{equation*}
\end{Lemma}

With these tools in place we can now begin the process of identifying the Floer homology groups associated to different finely tuned Hamiltonians. As above, the arguments we use are standard, but are complicated by the need to manage the cost of homotopies at each step.

Given a function $G \colon \R \times M \to \R$  set 
\begin{equation}
\label{ }
\|G\| = \max_{\R \times M} G - \min_{\R \times M} G.
\end{equation}


\begin{Corollary}
\label{closer}
Suppose that the Hamiltonians $H^0$ and $H^1$  are finely tuned to  $\Cc_{\lambda_0,\alpha}(T)$. If 
\begin{equation}
\label{close1}
\max_{\R \times M}(H^1-H^0) < \Delta_s(H^0, H^1),
\end{equation}
\begin{equation}
\label{close2}
\max_{\R \times M}(H^0-H^1) < \Delta_s(H^1, H^0),
\end{equation}
and 
\begin{equation}
\label{close3}
\|H^1-H^0\| < \Delta_s(H^0, H^0),
\end{equation}
then  $\HF(H^0; \alpha)$ and $\HF(H^1; \alpha)$ are isomorphic.
\end{Corollary}

\begin{proof}
Fix a smooth nondecreasing step function $\mathrm{\mathbf{step}} \colon \R \to \R$ such that 
$$
\mathrm{\mathbf{step}}(\tau)=\begin{cases}
   0   & \text{ for $\tau \leq -1$ }, \\
    1  & \text{for  $\tau \geq 0$}.
\end{cases}
$$
Let
\begin{equation*}
\label{ }
H^s = (1-\mathrm{\mathbf{step}}(s))H^0 + \mathrm{\mathbf{step}}(s)H^1
\end{equation*}
We then have
\begin{equation*}
\label{}
\mathrm{cost}(H^s) =\max_{\R \times M}(H^1-H^0)  
\end{equation*}
and
\begin{equation*}
\label{}
 \mathrm{cost}(H^{-s}) = \max_{\R \times M} (H^0-H^1).
\end{equation*}
By Proposition  \ref{split},  there is an admissible homotopy $\widetilde{H}^s$ from $H^0$ to $H^0$ with 
\begin{equation}
\label{close}
\mathrm{cost}(\widetilde{H}^s) \leq \max_{\R \times M} (H^1-H^0) - \min_{\R \times M} (H^1-H^0) = \|H^1-H^0\| < \Delta_s(H^0, H^0),
\end{equation}
such that 
\begin{equation*}
\label{ }
\Theta_{\widetilde{H}^s} = \Theta_{H^{-s}} \circ  \Theta_{H^s} \colon \HF(H^0; \alpha) \to \HF(H^0; \alpha).
\end{equation*}
By Lemma \ref{same}, the map $\Theta_{\widetilde{H}^s}$ is the same as that corresponding to the constant homotopy from $H^0$ to itself, and so $\Theta_{\widetilde{H}^s}$ is an isomorphism.  Thus,  $\Theta_{H^s}$ is injective. Applying the same argument to the composition $\Theta_{H^{s}} \circ  \Theta_{H^-s}$
we see that $\Theta_{H^s}$ is also surjective.
\end{proof}

\begin{Corollary}
\label{c1close}
If $H^0$ and $H^1$ are finely tuned $\Cc_{\lambda_0,\alpha}(T)$ and $H^1$ is sufficiently $C^1$-close to $H^0$ then $\HF(H^0; \alpha)$ is isomorphic to $\HF(H^1; \alpha)$.  
\end{Corollary}

\begin{proof}
This follows easily from Corollary \ref{closer} since the hypotheses of the theorem are met for all $H^1$ sufficiently $C^1$-close to $H^0$.
For example, if $H^k$ is a sequence of finely tuned Hamiltonians converging to $H^0$ in the $C^1$-topology then $\max_{\R \times M}(H^k-H^0) \to 0$ whereas 
$\Delta_s(H^0, H^k) \to \Delta_s(H^0, H^0)>0$.
\end{proof}

\begin{Corollary}
\label{bigger}
Suppose that the Hamiltonians $H^0$ and $H^1$  are finely tuned to  $\Cc_{\lambda_0,\alpha}(T)$ and that $H^s$  is an admissible homotopy from $H^0$ to $H^1$ such that for each $s$ the Hamiltonian $H^s$  is also finely tuned to $\Cc_{\lambda_0,\alpha}(T)$. Then 
the following statements hold.
\begin{enumerate}
  \item The groups $\HF(H^s; \alpha)$ are isomorphic to one another for all $s$. 
  \item If, in addition,  $\mathrm{cost}(H^s) =0$, then  the map
$$\Theta_{H^s}  \colon \HF(H^0; \alpha) \to \HF(H^1; \alpha)$$ is an isomorphism.
\end{enumerate}
\end{Corollary}

\begin{proof}

Reparameterizing if necessary we may assume that  $H^s= H^0$ for all $s\leq 0$ and $H^s= H^1$ for all $\geq 1$.
It follows from Corollary \ref{closer}, and continuity, that for each $s' \in [0,1]$ there is a $\delta_{s'}>0$ such that 
$\HF(H^{\varsigma}; \alpha)$ is isomorphic to $\HF(H^{s'}; \alpha)$ for all $\varsigma \in (s'- \delta_{s'}, s'+\delta_{s'})$.
Covering $[0,1]$ by finitely many such intervals it follows that for all $s\in [0,1]$ the groups $\HF(H^{s}; \alpha)$  are isomorphic to one another. 


To  prove the second assertion of the Corollary, it suffices (by Lemma \ref{same}), to find an admissible homotopy $\widetilde{H}^s$ from $H^0$ to $H^1$ with $\mathrm{cost}(\widetilde{H}^s) < \Delta_s(H^0, H^1)$ such that  $\Theta_{\widetilde{H}^s}\colon \HF(H^0; \alpha) \to \HF(H^1; \alpha)
$ is an isomorphism. We will use $H^s$ and Lemma \ref{split} to construct  this $\widetilde{H}^s$.

Arguing as above, and invoking the proof of Corollary \ref{closer},  we can find numbers $s_0=1<s_1<\dots s_N=1$ such that for $k=0, \dots, N-1$  each linear homotopy $$G^s_k =H^{s_k}(1-\mathbf{ step}(s)) + H^{s_{k+1}}\mathbf{ step}(s)$$ induces an isomorphism $\Theta_{G_k^s}\colon \HF(H^{s_k}; \alpha) \to \HF(H^{s_{k+1}}; \alpha)$. From $H^s$ we can also construct a homotopy $$H^s_k = H^{(s_k + (s_{k+1}-s_k)\mathbf{ step}(s))}$$ from $H^{s_k}$ to $H^{s_{k+1}}$ which is admissible and cost free. It follows from Lemma \ref{same} that $\Theta_{H_k^s} = \Theta_{G_k^s}$ so that each $\Theta_{H_k^s}$ is an isomorphism. On the other hand, since each $H^s_k$ is costfree and each $\Delta_s(H^{s_k}, H^{s_{k+1}})$ is positive we can invoke Lemma \ref{split} $N+1$ times to obtain  a cost free admissible homotopy $\widetilde{H}^s$ from $H^0$ to $H^1$ such that
\begin{equation*}
\label{ }
\Theta_{\widetilde{H}^s} = \Theta_{H_{N-1}^s} \circ  \cdots \circ \Theta_{H_0^s}.
\end{equation*}
This completes the proof.
\end{proof}

Now we get to our main invariance result.

\begin{Proposition}
\label{hf}
The rank of $\HF(H; \alpha)$ is the same for every Hamiltonian $H$ that is finely tuned to the rigid constellation $\Cc_{\lambda_0,\alpha}(T)$.
\end{Proposition}

\begin{proof}


We first observe that for all sufficiently small  $\epsilon>0$ every Hamiltonian 
 in the space  $$\Hh_{\epsilon}(T ) =\mathcal{H}_{\epsilon, T+\epsilon,3(T+\epsilon)/\epsilon,0}$$ is finely tuned to $\Cc_{\lambda_0,\alpha}(T)$. Moreover, there is an $\epsilon_0>0$ such that for any $\epsilon, \epsilon' < \epsilon_0$ and any Hamiltonians
 $H_{\epsilon} \in \Hh_{\epsilon}(T)$ and $H_{\epsilon'} \in \Hh_{\epsilon'}(T)$ we have 
\begin{equation*}
\label{}
\Delta_s(H_{\epsilon}, H_{\epsilon '}) > (T_{\min}(\lambda_0)+T_{\min}(\lambda_0, \alpha) -T)/2>0.
\end{equation*}
It then follows from Corollary \ref{closer} that for all $\epsilon< \epsilon_0$ and every Hamiltonian $H_{\epsilon}$ in $\Hh_{\epsilon}(T)$  the rank of  $\HF(H_{\epsilon}; \alpha)$ is the same.

By Corollary \ref{bigger}, it now suffices to show that given any finely tuned Hamiltonian $H$ there is an admissible homotopy $H^s$ which consists of finely tuned Hamiltonians and connects $H$ to some $H_{\epsilon} \in \Hh_{\epsilon}(T)$ with $\epsilon< \epsilon_0$. The starting point of $H^s$, $H$,  belongs to $\Hh_{a,b,c,\kappa}$ for some $h$ in   $\mathfrak{h}_{a,b,c}$. For the endpoint $H_{\epsilon}$ we choose $\epsilon< \epsilon_0$ small enough so that  the following inequalities  hold  $$\text{ $\epsilon<a$, \quad $T+\epsilon <b$\quad  and \quad $ 3(T+\epsilon)/\epsilon>c$}.$$
We now  view $h$ as belonging to a smooth family of functions $h(A,B,C)$ such that $h(A,B,C)$ belongs to $\mathfrak{h}_{A,B,C}$. The segments of the path $H^s$ will then be defined by varying the parameters $A$, $B$, $C$ and $\kappa$ one at a time.

We begin with $\kappa$. The distinction between a tuned and finely tuned Hamiltonian involves only the relationship between $a$ and $\kappa$, and if $a$ {\it works} for $\kappa$ then it also works for all smaller values of $\kappa$. So, the first segment of the path $H^s$ will be $$s\in [0,1] \mapsto  h(a,b,c)(e^{\tau -(1-s)\kappa}).$$ 
The next segment increases $C$ from $c$ to $3(T+\epsilon)/\epsilon$. Again it follows easily from the definitions that the intermediate Hamiltonians remain finely tuned. Continuing in this way, we decrease $B$ from b to $T+\epsilon$  and finally decrease $A$ from $a$ to $\epsilon$. By joining these four segments in order and reparameterizing to smoothen the transitions between them we obtain the desired homotopy $H^s$. 
\end{proof}

\section{The Proof of Theorem \ref{persist}}\label{proof1}
In the set-up of the Theorem \ref{persist} we are given a nondegenerate rigid constellation $\Cc_{\lambda_0,\alpha}( T)$  and  a pinched  contact form $\lambda = f\lambda_0$ such that the function $f$ is positive.  By rescaling we may assume that $\min(f) =1$ and so the pinching condition becomes 
\begin{equation*}
\label{pinch}
\max(f)<  \min\left\{\frac{T^+}{T}, \frac{T_{\min}(\lambda_0)+T_{\min}(\lambda_0 ,\alpha)}{T}\right\}.
\end{equation*}
We also have the condition that every closed Reeb orbit of $\lambda$ in the class $\alpha$ and with period in $[T_{\min}(\lambda_0, \alpha),T\max(f)]$ is nondegenerate.

Define $\widehat{T}$ by 
\begin{equation*}
\label{pinch}
\widehat{T} = \min\left\{T^+, T_{\min}(\lambda_0)+T_{\min}(\lambda_0 ,\alpha)\right\}.
\end{equation*}
Every  $b$ in the open  interval  $(T \max(f), \widehat{T})$  is not the period of a closed Reeb orbit of  $\lambda_0$. Choose such a $b$ which also lies in the complement of  $\mathcal{T}(\lambda,\alpha)$.
For a profile $h$ in $\mathfrak{h}_{a,b,c}$ we set 
\begin{equation*}
\label{G}
G(\tau,p)=h(e^{\tau}/f(p)).
\end{equation*} While $G$ is admissible, it is clearly not radial. However, for the diffeomorphism $\Psi_{\lambda} \colon \R \times M \to \R \times M$  defined  by 
$$
(\tau,p) \mapsto (\tau +f(p), p)
$$
we have 
\begin{equation}
\label{lambda}
\Psi_{\lambda}^*(e^{\tau}\lambda_0) = e^{\tau}\lambda
\end{equation} 
and 
\begin{equation}
\label{G}
\Psi_{\lambda}^*G(\tau,p) = h(e^{\tau}).
\end{equation} 
Thus $\Psi_{\lambda}^*G$ is radial. By our choice of $b$ above, we may therefore assume that for all sufficiently small $a$ and sufficiently large $c$, the Hamiltonian $\Psi_{\lambda}^*G$ is dividing (see Lemma \ref{divide}). More precisely, by \eqref{lambda}, the Hamiltonian  $\Psi_{\lambda}^*G$ is dividing for $\lambda$ and not for $\lambda_0$ (see Remark \ref{general divide}). As a consequence, there is a bijection between $\Pp^-_{\alpha}(\Psi_{\lambda}^*G)$ and $\Rr^b_{\alpha}(\lambda)$ and so, between $\Pp^-_{\alpha,\, \R /\Z}(\Psi_{\lambda}^*G)$ and $\Rr^b_{\alpha,\, \R /\Z}(\lambda)$, as well.

By equation \eqref{lambda} and the fact that $\Psi_{\lambda}$ is isotopic to the identity (and so preserves $[\R /\Z, \R \times M]$), we also know that $\Psi_{\lambda}$ maps $\Pp^-_{\alpha,\, \R /\Z}(\Psi_{\lambda}^*G)$ bijectively onto $\Pp^-_{\alpha,\, \R /\Z}(G)$ and preserves actions. Thus, every family in $\Pp^-_{\alpha,\, \R /\Z}(G)$ also corresponds to a unique family in $\Rr^b_{\alpha,\, \R /\Z}(\lambda)$.

\medskip
\noindent \textbf{Milepost 1.  To prove the first assertion of Theorem \ref{persist} it suffices to find at least $\mathrm{rank}(\Cc_{\lambda_0,\alpha}( T))$  elements of $\Pp^-_{\alpha,\, \R /\Z}(G)$ which correspond to distinct closed Reeb orbits of $\lambda$ with periods in  $\left[T_{\min}(\lambda_0, \alpha),\, T\max(f)\right]$.}
\medskip

We now refine this task. For the constant $a$ from the definition of the profile function $h$ consider the following subset of $\Pp^-_{\alpha,\, \R /\Z}(G)$
$$\Pp^a_{\alpha, \,\R /\Z}(G)= \left\{Y \in \Pp^-_{\alpha, \,\R /\Z}(G)\colon \Aa_G(Y) \in \left(-(1+a)T\max(f),\, -T_{\min}(\lambda_0, \alpha)+a^2 \right)   \right\}.$$ 

\begin{Lemma}\label{nondegenerate}
If $a>0$ is sufficiently small, then every family  $Y$ in $\Pp^a_{\alpha, \,\R /\Z}(G)$ is nondegenerate and the corresponding family $\Gamma_Y$ in $\Rr^b_{\alpha, \,\R /\Z}(\lambda)$ has period in the interval  $$\left[T_{\min}(\lambda_0, \alpha),\, T\max(f)\right].$$ 
\end{Lemma}

\begin{proof}

Consider a family $Y \in \Pp^a_{\alpha, \,\R /\Z}(G)$ and the corresponding family  $\Gamma_Y$ in $\Rr^b_{\alpha, \,\R /\Z}(\lambda)$.  As described above, the preimage $\Psi_{\lambda}^{-1}(Y)$ is an $\R /\Z$-family in $\Pp^-_{\R /\Z}(\Psi_{\lambda}^*G; \alpha)$  with the same action. 
So,
\begin{equation*}
\label{bounds}
\Aa_{\Psi_{\lambda}^*G}(\Psi_{\lambda}^{-1}(Y))=\Aa_G(Y) \in \left(-(1+a)T\max(f),\, -T_{\min}(\lambda_0, \alpha)+a^2 \right).\end{equation*}
Since the Hamiltonian $\Psi_{\lambda}^*G$ is dividing we also have
\begin{equation*}
\label{ }
\Aa_G(Y) = \Aa_{\Psi_{\lambda}^*G}(\Psi_{\lambda}^{-1}(Y))=-e^{\tau_Y} T_{\Gamma_Y} +h(e^{\tau_Y})
\end{equation*}
where $T_{\Gamma_Y}$ is the common period of the family $\Gamma_Y$ and $e^{\tau_Y}$ is in $(1, 1+a)$. 
Thus 
\begin{equation}
\label{open}
T_{\Gamma_Y} \in \left(\frac{T_{\min}(\lambda_0, \alpha)-a^2}{1+a},\, (1+a)T\max(f) + a^2\right).
\end{equation}

So, if the assertion of the lemma doesn't hold then for every small $a>0$ there is a  $Y$ in  $\Pp^a_{\alpha, \,\R /\Z}(G)$ such that  \eqref{open} holds but $T_{\Gamma_Y}$ lies outside the closed subinterval  
\begin{equation*}
\label{ }
\left[T_{\min}(\lambda_0, \alpha),\, T\max(f)\right] \subset \left(\frac{T_{\min}(\lambda_0, \alpha)-a^2}{1+a},\, (1+a)T\max(f) + a^2\right).
\end{equation*}
Thus, there is sequence of closed Reeb orbits of $\lambda$ in class $\alpha$ whose periods are monotonically converging to one of  the endpoints, $T_{\min}(\lambda_0, \alpha)$ or $T\max(f)$ as $a \to 0$. By Arzela-Ascoli, a subsequence of these orbits must converge to a closed Reeb orbit of $\lambda$ in class $\alpha$ with period equal to either  $T_{\min}(\lambda_0, \alpha)$ or $T\max(f)$. This contradicts our assumption that every closed Reeb orbit of $\lambda$ in class $\alpha$ and with period in $[T_{\min}(\lambda_0, \alpha),T\max(f)]$ is nondegenerate (and hence isolated).

\end{proof}

Henceforth we will assume that  $a>0$ is sufficiently small in the sense of Lemma \ref{nondegenerate}.

\medskip
\noindent \textbf{Milepost 2.  To prove the first assertion of Theorem \ref{persist} it suffices to find at least $\mathrm{rank}(\Cc_{\lambda_0,\alpha}( T))$ distinct elements of $\Pp^a_{\alpha,\, \R /\Z}(G)$.}
\medskip

To achieve this, we now use the Floer theoretic machinery developed in the previous section. 
We will argue as in \cite{ke} by adapting a technique introduced by Chekanov in \cite{ch}. 

Starting with the profile $h$ used to define $G$  we set 
\begin{equation*}
\label{H0}
H^0 =h(e^{\tau})
\end{equation*} 
and 
\begin{equation*}
\label{H1}
H^1 =h(e^{\tau}/\max(f)).
\end{equation*}
By our choice of $b \in (T \max(f), \widehat{T})$ both $H^0$ and $H^1$ are finely tuned to $\Cc_{\lambda_0,\alpha}( T)$ for all sufficiently small $a$ and sufficiently large $c$  . Since $h$ is nondecreasing and $\min(f)=  1$ we also have 
\begin{equation*}
\label{ }
h(e^{\tau}) \geq h(e^{\tau}/f(p)) \geq h(e^{\tau}/\max(f))
\end{equation*}
for all $(\tau, p) \in \R \times M$, and thus
\begin{equation}
\label{0G1}
H^0 \geq G \geq H^1.
\end{equation}

Using, again, the simple function $\mathrm{\mathbf{step}}$ we define two admissible homotopies;
\begin{equation*}
\label{ }
{H}_0^s = (1-\mathrm{\mathbf{step}}(s))H^0 + \mathrm{\mathbf{step}}(s)G
\end{equation*}
from $H^0$ to $G$, and
\begin{equation*}
\label{ }
{H}_1^s = (1-\mathrm{\mathbf{step}}(s))G + \mathrm{\mathbf{step}}(s)H^1
\end{equation*}
from $G$ to $H^1$. Inequality, \eqref{0G1} implies that  
\begin{equation}
\label{down}
\p_s({H}_0^s),\, \p_s({H}_1^s) \leq 0
\end{equation}
and so 
\begin{equation*}
\label{}
\mathrm{cost}({H}_0^s)=  \mathrm{cost}({H}_1^s) =0.
\end{equation*}

Consider the ${\mathbb{Z}/2}$-vector space spanned by the elements of $\Pp^a_{\alpha,\, \R /\Z}(G)$,
$$V^a(G;\alpha) = \mathrm{Span}_{\mathbb{Z}/2}\{Y \in \Pp^a_{\alpha,\, \R /\Z}(G)\}.$$
For regular Floer continuation data sets ${F}_0^s=({H}_0^s,{J}_0^s)$  and  ${F}_1^s=({H}_1^s,{J}_1^s)$ we define two linear maps. The first,
\begin{equation*}
\label{ }
\chi_{{F}_0^s} \colon \CF(H^0; \alpha) \to V^a(G; \alpha),
\end{equation*}
is defined on generators by
$$\chi_{{F}_0^s}(X^0) = \sum_{Y \in \Pp^a_{\alpha,\, \R /\Z}(G)} \#\Mm^0_s(X^0, Y;\,{F}_0^s) \,Y,$$
and the second  map
\begin{equation*}
\label{ }
\chi_{{F}_1^s} \colon V^a(G; \alpha) \to \CF(H^1; \alpha)
\end{equation*}
is defined on generators by
$$\chi_{{F}_1^s}(Y) = \sum_{X^1 \in \Pp^-_{\alpha, \,\R /\Z}(H^1)} \#\Mm^0_s(Y, X^1;\,{F}_1^s) \, X^1.$$
Since the elements of $\Pp^-_{\alpha, \,\R /\Z}(H^0)$ and $\Pp^-_{\alpha, \,\R /\Z}(H^1)$ are all simple, it follows  from Proposition \ref{s-break-glue} that  the maps  $\chi_{{F}_0^s}$ and $\chi_{{F}_1^s}$ are well defined. 

Note that no claim is being made that  $\chi_{{F}_0^s}$ and $\chi_{{F}_1^s}$ are chain maps. Indeed the relevant version of Floer homology can only be defined for $G$ by imposing prohibitively restrictive assumptions on it and hence on $\lambda$.  This reflects one of the important observations of Chekanov in \cite{ch}. We now prove the following result. 

\begin{Proposition}\label{prop:ident}
The composition $\chi_{{F}_1^s} \circ \chi_{{F}_0^s}\colon \CF(H^0; \alpha) \to \CF(H^1; \alpha)$ is a chain map which induces an isomorphism in homology. 
\end{Proposition}

\begin{proof} To prove this we return to the setting of Proposition \ref{composition}. For the homotopies  ${H}_0^s$ and ${H}_1^s$ above, consider the half-open homotopy of homotopies
\begin{equation*}
\label{ }
{H}_{0\#1}^{r,s} = \begin{cases}
   {H}_0^{s+ \xi(r)}   & \text{for $s\leq 0$}, \\
    {H}_1^{s-\xi(r)}  & \text{for $s>0$}.
\end{cases}
\end{equation*}
where $\xi(r)$ is a smooth, positive and nondecreasing function which equals $r$ for $r \gg 2$ and which equals $2$ for $r \leq 0$. 
It follows from \eqref{0G1} and the choices above that \begin{equation}
\label{final cost}
 \mathrm{cost}({H}_{0\#1}^{r,s})=0.
\end{equation} 

Fixing regular Floer data sets $F^0=(H^0,J^0)$, $F^1=(H^1,J^1)$ and $F^G=(G,  J^G)$ we extend these to Floer continuation data sets
\begin{equation*}
\label{ }
{F}_0^s= ({H}_0^s,  {J}_0^s) \in \mathbf{F}^s(F^0, F^G),
\end{equation*}
\begin{equation*}
\label{ }
{F}_1^s= ({H}_1^s,  {J}_1^s) \in \mathbf{F}^s(F^G, F^1),
\end{equation*}
which we use to form  the Floer homotopy data set
\begin{equation*}
\label{ }
{F}_{0\#1}^{r,s} = \begin{cases}
 {F}_0^{s+ \xi(r)}   & \text{for $s\leq 0$}, \\
    {F}_1^{s-\xi(r)}  & \text{for $s>0$}
\end{cases}
\end{equation*}
in $\mathbf{F}^{r,s}(F^0, F^1)$. 
Perturbing again, if necessary, we assume that these data sets are all regular.

Given $X^0 \in \Pp^-_{\alpha, \,\R /\Z}(H^0)$ and $X^1\in \Pp^-_{\alpha, \,\R /\Z}(H^1)$, Proposition \ref{composition} implies that  the boundary of the compactification  $\overline{\Mm}^1_{r,s}(X^0, X^1;\,{F}_{0\#1}^{r,s})$  can be identified with the elements of the following four sets:
\begin{enumerate}
  \item[(I)]  $\displaystyle\Mm^0_s(X^0, X^1;\,{F}_{0\#1}^{0,s})$,
  
  \medskip
\item[(II)]   $\displaystyle \bigcup_{Z \in \Pp^-_{\alpha, \,\R /\Z}(G)} \Mm^0_s(X^0, Z;\,{F}^s_0) \times  \Mm^0_s(Z, X^1;\,{F}^s_1)$,
\item[(III)] $\displaystyle\bigcup_{Y^0 \in \Pp^-_{\alpha, \,\R /\Z}(H)} \Mm^0(X^0, Y^0;\,F^0) \times  \Mm^{0}_{r,s}(Y^0, X^1;\,{F}_{0\#1}^{r,s})$,

  \item[(IV)] $\displaystyle\bigcup_{Y^1 \in \Pp^-_{\alpha,\, \R /\Z}(H^1)} \Mm^{0}_{r,s}(X^0, Y^1;\,{F}_{0\#1}^{r,s}) \times \Mm^0(Y^1, X^1;\,F^1).$
\end{enumerate}

By definition, the number of elements in set (I), modulo 2, is the coefficient of $X^1$ in the image of $X^0$ under the
map $$\theta_{{F}_{0\#1}^{0,s}} \colon \CF(H^0; \alpha) \to \CF(H^1; \alpha),$$ which is well-defined by \eqref{final cost}. 
If we can show that every $Z$ which contributes a term to the set (II) must belong to the subset  $\Pp^a_{\alpha,\, \R /\Z}(G)$ of $\Pp^-_{\R /\Z}(G; \alpha)$, then 
 the number of elements in set (II), modulo 2, will be the coefficient of $X^1$ in the image of $X^0$ under the
map  $$\chi_{{F}_1^s} \circ \chi_{{F}_0^s} \colon \CF(H^0; \alpha) \to \CF(H^1; \alpha).$$
With this,  the fact that $\theta_{{F}_{0\#1}^{0,s}}$ and  $\chi_{{F}_1^s} \circ \chi_{{F}_0^s}$ are chain homotopic will follow from the usual arguments.

Suppose then  that  $Z \in \Pp^-_{\R /\Z}(G; \alpha)$ contributes a nontrivial term to the set (II). In this case both $\Mm^0_s(X^0, Z;\,{F}^s_0)$ and $\Mm^0_s(Z, X^1;\,{F}^s_1)$ must be nonempty. By \eqref{sr-energy-action} and \eqref{down} we then have
\begin{equation*}
\label{ }
 \Aa_{H^0}(X^0)> \Aa_G(Z) > \Aa_{H^1}(X^1)
\end{equation*}
and so, by \eqref{eps}, the following
\begin{equation*}
\label{ }
 -(1+a)T \max(f)< \Aa_G(Z)< -T_{\min}(\lambda_0, \alpha) +a^2. \end{equation*}
Thus,  $Z$ belongs to $\Pp^a_{\alpha,\, \R /\Z}(G)$, as desired. 

Since  $\chi_{{F}_1^s} \circ \chi_{{F}_0^s}$ is chain homotopic to $\theta_{{F}_{0\#1}^{0,s}}$ it only remains to show that $\Theta_{{H}_{0\#1}^{0,s}}$ (the map that $\theta_{{F}_{0\#1}^{0,s}}$ induces in homology) is an isomorphism. 
Consider the admissible homotopy
\begin{equation*}
\label{ }
H^s(\tau,p)=h\left(\frac{e^{\tau}}{1-\mathrm{\mathbf{step}}(s) + \mathrm{\mathbf{step}}(s) \max(f)}\right)
\end{equation*}
from $H^0$ to $H^1$. 
For all small enough $a$ and large enough $c$, each function $H^s$ is finely tuned to the rigid constellation $\Cc_{\lambda_0,\alpha}( T)$. Moreover, $\p_s(H^s) \leq 0$ and so Corollary \ref{bigger} implies  that the map 
$$\Theta_{H^s}  \colon \HF(H^0; \alpha) \to \HF(H^1; \alpha)$$ is an isomorphism.
Lemma \ref{same} then implies that 
\begin{equation*}
\label{ }
\Theta_{{H}_{0\#1}^{0,s}} = \Theta_{H^s}
\end{equation*}
which concludes the proof.

\end{proof}

At this point we can complete the task described in Milepost 2, and thus the proof of the first assertion of Theorem \ref{persist}.

\begin{Lemma}\label{number}
There are at least $\mathrm{rank}(\Cc_{\lambda_0,\alpha}( T))$ 
distinct elements of $\Pp^a_{\alpha,\, \R /\Z}(G)$. 
\end{Lemma}

\begin{proof}

Let $V_0$ be a subspace of $\CF(H; \alpha)$ that represents the homology $\HF(F^0; \alpha)$.
By Proposition \ref{prop:ident}, the restriction of $\chi_{{F}_R^s} \circ \chi_{{F}_L^s}$ to $V_0$ has no kernel.
Thus the restriction of $\chi_{{F}_L^s}$ to $V_0$ also has no kernel.  We therefore have
\begin{equation*}
\label{ }
\dim(V^a(G;\alpha)) \geq \dim(\chi_{{F}_L^s}(V_0)) = \dim(V_0)= \mathrm{rank}(\HF(F^0; \alpha)).
\end{equation*}
With this we are done.
\end{proof}

Finally we prove the second assertion of Theorem \ref{persist}

\begin{Lemma}\label{distinct} 
If the class $\alpha$  is either primitive or of infinite order, then the closed Reeb orbits of $\lambda$ corresponding to the distinct elements of $\Pp^a_{\alpha, \,\R /\Z}(G)$ are geometrically distinct.
Otherwise, they are geometrically distinct if there are no closed Reeb orbits of  $\lambda$ with period at most $$\frac{1}{|\alpha|}\left( T\max(f)-T_{\min}(\lambda_0, \alpha)\right)$$
that represent a class $\beta$ in $[\mathbb{S}^1 ,M]$ such that $\beta^k =\alpha$ for some integer $k>1$.
\end{Lemma}

\begin{proof}

Let  $Y$ and $Y'$ be distinct elements of $\Pp^a_{\alpha, \,\R /\Z}(G)$ and let $\Gamma_Y$ and $\Gamma_{Y'}$ by the corresponding (distinct) elements of $\Rr^b_{\alpha,\, \R /\Z}(\lambda)$. Assume that  $\Gamma_Y$ and $\Gamma_{Y'}$ are not geometrically distinct. Then there must be an $\R /\Z$-family $\Gamma_R$ of closed Reeb orbits of $\lambda$ and integers $k\geq1$ and $l \geq 1$ such that 
\begin{equation}
\label{not distinct}
\Gamma_Y = (\Gamma_R)^k  \quad \text{ and } \quad \Gamma_{Y'} = (\Gamma_R)^{k+l}.
\end{equation}
Let $\beta$ be the class in $[\R /\Z, M]$ represented by $\Gamma_R$. By  \eqref{not distinct} we then have
\begin{equation}
\label{prim}
\beta^k =\beta^{k+l} =\alpha.
\end{equation}
This implies that $\beta^l$ is equal to the trivial element $e \in [\R /\Z, M]$ and so
\begin{equation}
\label{ord}
\alpha^l = (\beta^k)^l = (\beta^l)^k =e.
\end{equation}
In this case the class $\alpha$ can be neither primitive, by \eqref{prim}, or of infinite order, by \eqref{ord}. This implies the first statement of the lemma.

Suppose then that  $\Gamma_Y$ and $\Gamma_{Y'}$ are not geometrically distinct and that the $\alpha$ is not primitive and  is of finite order.  By Lemma \ref{nondegenerate}, we have
\begin{equation}
\label{both}
T_{\Gamma_Y}, \,  T_{\Gamma_{Y'}}\in \left[T_{\min}(\lambda_0, \alpha),\, T\max(f)\right].
\end{equation} 
Together with \eqref{not distinct},  this implies that 
\begin{equation*}
\label{}
kT_{\Gamma_R}, \,  (k+l)T_{\Gamma_R}\in \left[T_{\min}(\lambda_0, \alpha),\, T\max(f)\right],
\end{equation*}
and so 
\begin{equation}
\label{lbound}
lT_{\Gamma_R} \leq T\max(f)-T_{\min}(\lambda_0, \alpha).
\end{equation}
Equation \eqref{ord} implies that $l \geq |\alpha|$ an so we conclude from \eqref{lbound} that  
\begin{equation*}
\label{ }
T_{\Gamma_R}\leq \frac{1}{|\alpha|}\left(T\max(f)-T_{\min}(\lambda_0, \alpha)\right).
\end{equation*}
This implies the second statement of the lemma and concludes the proof.
\end{proof}

\section{The Proof of Theorem \ref{gutt+}}

To begin we recall the setting and the statement of the theorem.
Let  $(Q, \omega)$ be a symplectic manifold of dimension $2n$ such that the class 
$-[\omega]/2\pi \in H^2(Q;\mathbb{R})$ is the image of an integral class $\mathrm{ \mathbf{e}} \in H^2(Q; \mathbb{Z})$.
Let $$p_Q \colon M \to Q$$ be an $\mathbb{S}^1$-bundle over $Q$ with first Chern class equal to $\mathrm{ \mathbf{e}}$ and let $\lambda_Q$ be the corresponding Boothby-Wang contact form on $M$. 
Denote by $\alpha_\mathbf{f} \in [\mathbb{S}^1 ,M]$ the free homotopy class corresponding to the fibres of the bundle $p_Q$. 

\begin{Theorem}
Let  $\lambda =f\lambda_Q$ for some positive function $f$. If  
$$
\frac{\max(f)}{\min(f)}< 2
$$
then there are at least $n+1$ distinct closed Reeb orbits of $\lambda$ which represent the class $\alpha_{\mathbf{f}}$ and have period in the interval  $$[2\pi \min(f), 2\pi \max(f)].$$ 

These orbits are geometrically distinct from one another if the class  $\alpha_{\mathbf{f}}$ is either primitive or is of infinite order. Otherwise, they are geometrically distinct if there are no  closed Reeb orbits of $\lambda$  which have period less than or equal to $$\frac{2 \pi}{|\alpha_{\mathbf{f}} |} \left(\max(f)- \min(f)\right)$$ and which represent a class $\beta$ such that $\beta^k =\alpha_{\mathbf{f}}$ for some integer $k>1$.
\end{Theorem}

The  assertions concerning the conditions under which the detected orbits are  geometrically distinct follow as in Lemma \ref{distinct}, and so their proof is left to the reader.  It remains to for us to detect  $n+1$ closed Reeb orbits of $\lambda$ in class $\alpha_\mathbf{f}$ which are distinct and whose periods lie in the interval  $[2\pi \min(f), 2\pi \max(f)].$ 

We may assume that $\lambda$ has finitely many, say $N$, distinct $\R /\Z$-families of such closed orbits. We denote them by $$\Xi_1, \dots, \Xi_N.$$ We may also  assume, by a simple rescaling,  that $\min(f) =1$. It remains to prove that $N \geq n+1$

\medskip

\noindent{\bf{Step 1.}} First we derive a Morse theoretic implication of the fact that $\omega_Q$ is a symplectic form (Lemma \ref{pointss} below).
Let $F\colon Q \to \R$ be a Morse function, let $q$ be a critical point of $F$ with Morse index $\mathrm{index}(q)$ and let $g$ be a Riemannian metric on $Q$. We denote the 
stable submanifold of $q$, for the negative gradient flow of $F$ with respect to $g$, by $W^s(q,(F,g))$, or just $W^s(q)$ assuming the gradient data is clear from the context. For a generic choice of the metric $g$ each $W^s(q)$ is an embedded submanifold diffeomorphic to $\R^{2n-\mathrm{index}(q)}$ and admits a compactification as a manifold with corners whose boundary faces are comprised of stable submanifolds of critical points of $F$ with Morse index greater than that of $q$. 

The following result is implied by standard transversality arguments. 
\begin{Lemma}\label{generic}
For a generic collection of Morse functions $F_1, \dots, F_{n}$ and metrics $g_1, \dots, g_{n}$, the stable and unstable submanifolds of the critical points of the $F_j$ all intersect transversally as do all of their repeated intersections. Moreover for any closed Reeb orbit of $\xi_i(t)$ of $\lambda$ belonging to one of the families  $\Xi_i$, and any critical point $q_j$ of any $F_j$  we have 
\begin{equation}
\label{trans}
W^s(q_j) \pitchfork (p_Q (\xi_i(t)).
\end{equation}
\end{Lemma}

From now on we fix Morse functions $F_1, \dots, F_{n}$ and metrics $g_1, \dots, g_{n}$ as in Lemma \ref{generic}.
We note, for later purposes, the following immediate consequence of condition \eqref{trans} and the description of the closure of stable submanifolds above.

\begin{Corollary}
\label{away}
If $q$ is a critical point of one of the $F_j$ and $\mathrm{index}(q) \geq 2$, then for all $i=1, \dots, N$ we have 
\begin{equation}
\overline{W^s(q)} \cap p_Q (\Xi_i) = \emptyset.
\end{equation}
\end{Corollary}

Since $\omega_Q$ is a symplectic form, its $n$-fold wedge product $$\omega_Q \wedge \cdots \wedge \omega_Q$$ is a volume form on $Q$.  When expressed in the Morse theoretic version of the cup product from \cite{bc},  for example, the existence if this nontrivial wedge product  has the following implication.
\begin{Lemma}\label{pointss}
Let $F_1, \dots, F_n$  and $g_1, \dots, g_n$ be a collection of Morse functions and metrics as in Lemma \ref{generic}. There are critical points $q_j$ of the $F_j$ such that 
$\mathrm{index}(q_j) = 2$ and $$W^s(q_1) \cap  \dots \cap W^s(q_n)$$
is a compact manifold of dimension zero with an odd number of elements.
\end{Lemma}

We will denote the set $W^s(q_1) \cap  \dots \cap W^s(q_n)$ by $\Mm_0$ and its elements by 
$$
\{\widehat{q}_1, \dots, \widehat{q}_{2K+1}\}.
$$

\medskip

\noindent{\bf{Step 2.}} We now define a useful  lift of the set $\Mm_0 \subset Q$ to $M$.
Since each stable manifold $W^s(q_i)$ is contractible, the restriction of the bundle $p_Q \colon M \to Q$ to  $W^s(q_i)$ is trivial. Fix  such a trivialization for each $W^s(q_i)$. Then given any point $q \in W^s(q_i)$ and any point $m \in p_Q^{-1}(q)$ there is a unique lift of $W^s(q_i)$ to $M$ which intersects  $p_Q^{-1}(q)$ at $m$. We denote this lift by $$[ W^s(q_i)]_m$$ and note that 
\begin{equation*}
\label{ }
[ W^s(q_i)]_m \cap[ W^s(q_i)]_{m'} =\emptyset \iff m\neq m' \in p_Q^{-1}(q). 
\end{equation*}

Choose an $m_1 \in p_Q^{-1}(\widehat{q}_1)$ and consider the set $$[W^s(q_1)]_{m_1} \cap  \dots \cap [W^s(q_n)]_{m_1}.$$
Since it projects to $\Mm_0$ we have  
\begin{equation*}
\label{ }
[W^s(q_1)]_{m_1} \cap  \dots \cap [W^s(q_n)]_{m_1}  = \bigcup_{j \in [1, 2K+1]} \left(\bigcap_{i\in [1,n]}[W^s(q_i)]_{m_1} \cap p_Q^{-1}(\widehat{q}_j)\right).
\end{equation*}
In particular $[W^s(q_1)]_{m_1} \cap  \dots \cap [W^s(q_n)]_{m_1}$ is a finite set of points  each of which is a point on a fibre $p_Q^{-1}(\widehat{q}_j)$ at which all the $[W^s(q_i)]_{m_1}$ meet. By construction  $m_1$ is one of these points. Relabelling the $\widehat{q}_j$, if necessary,   we may assume that 
\begin{equation*}
\label{ }
[W^s(q_1)]_{m_1} \cap  \dots \cap [W^s(q_n)]_{m_1} =\{m_1, \dots, m_{k_1-1} \}
\end{equation*}
where $2\leq k_1\leq 2K+2$ and 
\begin{equation*}
\label{points}
m_j \in  p_Q^{-1}(\widehat{q}_j).
\end{equation*}

To proceed we now choose a point  $m_{k_1} \in p_Q^{-1}(\widehat{q}_{k_1})$ and consider the lifts $[W^s(q_i)]_{m_{k_1}}$.
For a generic such point we may assume that 
\begin{equation*}
\label{ }
[W^s(q_i)]_{m_1} \cap [W^s(q_l)]_{m_{k_1}} \cap p_Q^{-1}(\widehat{q}_j) =\emptyset, 
\end{equation*}
for all $i$, $l$, and $j$. The intersection
\begin{equation*}
\label{second set}
[W^s(q_1)]_{m_{k_1}} \cap  \dots \cap [W^s(q_n)]_{m_{k_1}} 
\end{equation*}
is again a finite set consisting of points on the  fibres $p_Q^{-1}(\widehat{q}_j)$ at which all the $[W^s(q_i)]_{m_{k_1}}$ meet.
Since 
\begin{equation*}
\label{ }
\bigcap_{1\in [1,n]}[W^s(q_i)]_{m_1} \cap p_Q^{-1}(\widehat{q}_{k_1}) =\emptyset
\end{equation*}
it follows that none of the points in $[W^s(q_1)]_{m_{k_1}} \cup  \dots \cup [W^s(q_n)]_{m_{k_1}} $ lie in the fibres $p_Q^{-1}(\widehat{q}_j)$ for $j=1, \dots, k_1-1$. Thus, relabelling again if needed,  we may assume that 
\begin{equation*}
\label{ }
[W^s(q_1)]_{m_{k_1}} \cap  \dots \cap [W^s(q_n)]_{m_{k_1}} =\{m_{k_1}, \dots, m_{k_2-1} \}
\end{equation*}
where $k_1+1 \leq k_2\leq 2K+2$ and again 
\begin{equation*}
\label{points}
m_j \in  p_Q^{-1}(\widehat{q}_j).
\end{equation*}  

Continuing in this way, we obtain a set of points 
$$
\{m_1, \dots, m_{k_1}, \dots, m_{k_2}, \dots, m_{k_L}, \dots , m_{2K+1}\}
$$
such that 
\begin{equation}
\label{project}
m_j \in  p_Q^{-1}(\widehat{q}_j)
\end{equation}
for all $j=1, \dots, 2K+1$.
Setting $k_0=1$ and $k_{L +1}=2K+1$ we  have
\begin{equation*}
\label{ }
[W^s(q_1)]_{m_{k_j}} \cup  \dots \cup [W^s(q_n)]_{m_{k_j}} =\{ m_{k_j},  \dots , m_{k_{j+1}}\}
\end{equation*} 
for $j=0, \dots, k_L$. We may also assume that for $d \neq d'$ 
\begin{equation}
\label{disjoint}
[W^s(q_i)]_{m_{k_d}} \cap [W^s(q_l)]_{m_{k_{d'}}} \cap p_Q^{-1}(\widehat{q}_j) =\emptyset, 
\end{equation}
for all $i$, $l$, and $j$.
 
We set  $$[\Mm_0] = \{m_1, \dots, m_{2K+1}\}$$ and note, for future reference, that 
 \begin{equation}
\label{union}
[\Mm_0] = \bigcup_{j=0}^L \Big( [W^s(q_1)]_{m_{k_j}} \cap [W^s(q_2)]_{m_{k_j}} \cap  \cdots \cap [W^s(q_n)]_{m_{k_j}}\Big).
\end{equation}

\medskip

\noindent{\bf{Step 3.}} Here we identify the set  $[\Mm_0]$ with a space of solutions to Floer's equation. Recall that the Reeb flow of $\lambda_Q$ generates the natural $\mathbb{S}^1$-action on the bundle $M$  with (minimal) period $2\pi$.  In particular, every $q \in Q$  can be identified with the $\R /\Z$-family of closed Reeb orbits of $\lambda_Q$ of period $2\pi$ whose image is $p_Q^{-1}(q)$. We will denote this family by $\Gamma_q$ and will denote an element of $\Gamma_q$ by $\gamma_q(t)$. 

To proceed  we now utilize some of the machinery developed in the proof of Theorem \ref{persist}. The collection $\Cc_{\lambda_Q, \alpha_{\mathbf{f}}} (2\pi)$ is  a rigid constellation with 
$$\min\left\{ \frac{T^+}{T},\,\frac{T_{\min}(\lambda_Q)+ T_{\min}(\lambda_Q, \alpha_{\mathbf{f}})}{T}\right\} =2.$$
Choose a constant  $b$ which lies in the open  interval  $(2\pi \max(f), 4\pi)$ and which is not the period of a closed Reeb orbit of $\lambda$.
Since $\max(f)<2$, we can choose $a$ sufficiently small and $c$ sufficiently large so that for any profile $h$ in $\mathfrak{h}_{a,b,c}$ 
the functions $H^0(\tau,p)=h(e^{\tau})$ and $H^1(\tau,p)=h(e^{\tau}/\max(f))$
are finely tuned to $\Cc_{\lambda_Q, \alpha_{\mathbf{f}}} (2\pi)$ and the function  $\Psi^*_{\lambda}G(\tau,p)=h(e^{\tau})$ is dividing with respect to $\lambda$. 

The nonconstant $1$-periodic orbits of $H^0$  with negative action are of the form 
\begin{equation}
\label{type+}
x(t) = (\tau_0, \gamma(2\pi t)),
\end{equation}
where $\tau_0$ is the unique solution of $h'(e^{\tau}) =2\pi$ in the interval $(0, \ln(1+a))$
and $\gamma(t)$ belongs to one of the families $\Gamma_q$ for $q \in Q$.
We denote the collection of all $1$-periodic orbits of the form \eqref{type+} by $X(H^0)$. 

Choosing  a $J^0$ in $\Jj(H^0)$, we define $\Mm_1$ to be the set of smooth maps
$u \colon \R \times \R/\Z \to \R \times M $
such that
\begin{equation*}
\label{eq}
\p_su+J^0(u)\left(\p_tu -V_{H^0}(u)\right) = 0
\end{equation*}
\begin{equation}\label{lim}
\lim_{s \to \pm \infty} u(s,t) = x^{\pm}(t) \in X(H^0),
\end{equation}
and
\begin{equation}
\label{initial}
p_{M}(u(0,0)) \in [\Mm_0]
\end{equation}
where $p_M \colon \R \times M \to M$ is the obvious projection. 

The set of Floer trajectories $\Mm_1$ is in bijection with $[\Mm_0]$. To see this let $u$ belong  to $\Mm_1$ and suppose that $p_M(u(0,0)) =m_j \in [\Mm_0]$. Since the action $\Aa_{H^0}$ is constant on $X(H_0)$, the energy identity
\begin{equation*}
\label{ }
\int_{\R \times \R /\Z}d(e^{\tau}\lambda_0)(\p_s u, J^0(u) \p_s u)\,ds\,dt = \Aa_{H^0}(x^-) - \Aa_{H^0}(x^+)
\end{equation*}
together with the limiting conditions \eqref{lim} imply that $\p_su(s,t)=0$ for all $(s,t)$. 
Thus, $u(s,t) =(\tau)0, \gamma(2\pi t))$
where  $\gamma(t)$ belongs to one of the families $\Gamma_q$ for $q \in Q$.
It then follows from condition \eqref{initial} and property \eqref{project} that  
 \begin{equation*}
\label{ }
u(s,t) = (\tau_0, \gamma_{\widehat{q}_j}(2\pi t))
\end{equation*}
where $\gamma_{\widehat{q}_j}$ in the unique element of the family  $\Gamma_{\widehat{q}_j}$ satisfying  $\gamma_{\widehat{q}_j}(0)=m_j$. Conversely, every such map belongs to $\Mm_1$ and so it is in bijection with $\Mm_0$. In particular, $\Mm_1$ has an odd number of elements.

\medskip

\noindent{\bf{Step 4.}} We now begin to deform the space $\Mm_1$ as a set of Floer trajectories. The first deformation involves moving the orbits that define the right asymptotic limits of the curves of $\Mm_1$. Here we use the  family of Hamiltonians 
\begin{equation*}
\label{H1}
H^{\varrho}(\tau,p) =h\left(\frac{e^{\tau}}{1+ \varrho(\max(f)-1)}\right)
\end{equation*}
which decreases from $H^0$ to $H^1
$ as $\varrho$ goes from zero to one.

\begin{Lemma}
For all $\rho \in [0,1]$, the nonconstant $1$-periodic orbits of $H^{\varrho}$  with negative action are of the form 
\begin{equation}
\label{rho}
t \mapsto (\tau_{\varrho}, \gamma(2\pi t))
\end{equation}
where $\tau_{\varrho}$ is the unique solution of $$h'\left(\frac{e^{\tau}}{1+ \varrho(\max(f)-1)}\right) =2\pi( 1+ \varrho(\max(f)-1))$$ in the interval $\left(1+ \varrho(\max(f)-1), 1+ \varrho(\max(f)-1)+ \ln(1+a) \right)$, 
and $\gamma(t)$ belongs to one of the families $\Gamma_q$ for $q \in Q$.
\end{Lemma}

Let $X(H^{\varrho} )$ be  the collection of all $1$-periodic orbits of the form \eqref{rho}. For each $\varrho$, let $H^{\varrho, s}$ be the homotopy from $H^0$ to $H^{\varrho}$ of the form 
\begin{equation*}
\label{}
H^{\varrho,s} = \left[(1-\mathrm{\mathbf{step}}(s))H^0 + \mathrm{\mathbf{step}}(s)\frac{1}{2}(H^0+H^{\varrho})\right](1-\mathrm{\mathbf{step}}(s)) + \mathrm{\mathbf{step}}(s)H^{\varrho}.
\end{equation*}
Each function appearing in this family is finely tuned to $\Cc_{\lambda_Q, \alpha_{\mathbf{f}}} (2\pi)$. So, for each $\varrho \in [0,1]$, we have
\begin{equation}
\label{needed1}
\Delta(H^0, H^{\varrho}) < T_{\min}(\lambda_Q) = 2\pi.
\end{equation}
As is easily checked, $\p_s(H^{\varrho,s}) \leq s$, hence for each $\varrho \in [0,1]$ we also have
\begin{equation}
\label{needed2}
\mathrm{cost}(H^{\varrho,s}) =0.
\end{equation}
for the corresponding homotopy.

Choose a smooth two-parameter family $J^{\varrho,s}$ of almost complex structures such that $(H^{0,s},J^{0,s})$ and $(H^{1,s}, J^{1,s})$ are regular. Let $\Mm_{1+\varrho}$ be the set of smooth maps
$u \colon \R \times \R/\Z \to \R \times M
$
such that
\begin{equation*}
\label{eq}
\p_su+J^{\varrho,s}(u)\left(\p_tu -V_{H^{\varrho,s}}(u)\right) = 0
\end{equation*}
\begin{equation*}\label{lim-}
\lim_{s \to - \infty} u(s,t)  \in X(H^0),
\end{equation*}
\begin{equation*}\label{lim+}
\lim_{s \to + \infty} u(s,t)  \in X(H^{\varrho}),
\end{equation*}
and
\begin{equation*}
\label{initial-rho}
p_{M}(u(0,0)) \in [\Mm_0].
\end{equation*}

By \eqref{needed1} and \eqref{needed2} each $\Mm_{1+\varrho}$ is compact.
So too is the collection
\begin{equation*}
\label{ }
\Mm_{[1,2]}=\{(\varrho, u) \mid \rho \in [0,1],\, u \in \Mm_{1+\varrho}\}.
\end{equation*}
In a standard way, $\Mm_{[1,2]}$ can also be described as the intersection of the zero section of an appropriate Banach space bundle with another Fredholm section.  As described by Albers and Hein in \cite{ah} (page 21), one can  then use (compact) abstract perturbations in this setting to perturb $\Mm_{[1,2]}$, away from the values $\rho =0,1$,  to obtain a compact cobordism  between the (zero-dimensional) spaces $\Mm_1$ and  $\Mm_2$. 

By definition, the space $\Mm_2$ consists of maps 
$u \colon \R \times \R/\Z \to \R \times M
$
satisfying 
\begin{equation*}
\label{eq2}
\p_su+J^{1,s}(u)\left(\p_tu -V_{H^{1,s}}(u)\right) = 0,
\end{equation*} 
\begin{equation*}\label{lim-2}
\lim_{s \to - \infty} u(s,t)  \in X(H^0),
\end{equation*}
\begin{equation*}\label{lim+2}
\lim_{s \to + \infty} u(s,t)  \in X(H^1),
\end{equation*}
and \begin{equation*}
\label{initial-rho}
p_{M}(u(0,0)) \in [\Mm_0].
\end{equation*}
It follows from the discussion above that $\Mm_2$ also has an odd number of elements.

\medskip

\noindent{\bf{Step 4.}} Now we deform $\Mm_2$ by deforming the monotone homotopy which defines it, 
\begin{equation*}
\label{}
H^{1,s} = \left[(1-\mathrm{\mathbf{step}}(s))H^0 + \mathrm{\mathbf{step}}(s)\frac{1}{2}(H^0+H^{1})\right](1-\mathrm{\mathbf{step}}(s)) + \mathrm{\mathbf{step}}(s)H^1, 
\end{equation*}
to another  monotone homotopy  from $H^0$  to $H^1$ that  lingers on the Hamiltonian $G$.

Let 
\begin{equation*}
\label{ }
G^r =(1- \mathrm{\mathbf{step}}(r-1))\frac{1}{2}(H^0+H^{1}) +\mathrm{\mathbf{step}}(r-1)G
\end{equation*}
and set
\begin{equation*}
\label{ }
\mathrm{\mathbf{STEP}}(r,s) = \mathrm{\mathbf{step}}(s-(n+1)r + \mathrm{\mathbf{step}}(r-1))).
\end{equation*}
With  these pieces define 
\begin{equation*}
\label{}
G^{r,s} = \left[(1-\mathrm{\mathbf{step}}(s))H^0 + \mathrm{\mathbf{step}}(s)G^r\right](1-\mathrm{\mathbf{STEP}}(r,s)) + \mathrm{\mathbf{STEP}}(r,s)H^1
\end{equation*}
This is an admissible homotopy of homotopies from $H^0$ to $H^1$ and 
the following addition properties of $G^{r,s}$ are easily verified. 
\begin{enumerate}
  \item[($G^{r,s}1$)] $G^{0,s} =H^{1,s}$.
  \item[($G^{r,s}2$)] $\p_s(G^{r,s}) \leq 0$.
  \item[($G^{r,s}3$)] $G^{r,s} =G$ whenever $r\geq 1$ and $s \in [0, (n+1)r]$ . 
\end{enumerate}

Choose a smooth family  of almost complex structures $\bar{J}^{r,s}$ on $\R \times M$ such that for all $r\leq 0$ we have $\bar{J}^{r,s} = J^{1,s}$ where $J^{1,s}$ is the path of almost complex structures used in the definition of $\Mm_2$, and for $r \in \N$ the continuation data set $(G^{r,s},  \bar{J}^{r,s})$ is regular. For each $r\geq 0$, define $\Mm_3^r$ to be the space  of maps
$
u \colon \R \times \R/\Z \to \R \times M
$
such that 
\begin{equation*}
\label{ }
\p_su+\bar{J}^{r,s}(u)\left(\p_tu -V_{G^{r,s}}(u)\right) = 0,
\end{equation*}
\begin{equation*}
\label{ }
\lim_{s \to - \infty} u(s,t) \in X(H^0),
\end{equation*}
\begin{equation*}
\label{ }
\lim_{s \to + \infty} u(s,t)  \in X(H^1),
\end{equation*}
and
\begin{equation}
\label{evr}
p_{M}(u(jr,0)) \in [W^s(q_j)]_{m_{k_0}} \sqcup  \cdots \sqcup [W^s(q_j)]_{m_{k_L}} \text{ for all $j=1, \dots n.$}
\end{equation}

\begin{Lemma} The space $\Mm_3^0$ is identical to $\Mm_2$.\end{Lemma} 

\begin{proof}
Since $G^{0,s} =H^{1,s}$ and $\bar{J}^{0,s} = J^{1,s}$, it suffices to show that 
when $r=0$ condition \eqref{evr} is equivalent to $p_{M}(u(0,0)) \in [\Mm_0]$.
Condition \eqref{evr} can be rewritten as 
\begin{equation*}
\label{ }
p_{M}(u(0,0)) \in \bigcap_{j=1}^n \Big( [W^s(q_j)]_{m_{k_0}} \sqcup   \cdots \sqcup [W^s(q_j)]_{m_{k_L}}\Big).
\end{equation*}
Distributing the intersections for the set appearing on the right it becomes the union of sets of the form
\begin{equation}
\label{set}
[W^s(q_1)]_{m_{k_{d_0}}} \cap  \cdots \cap [W^s(q_n)]_{m_{k_{d_L}}}.
\end{equation}
By definition, the projection of any set of this form to $Q$ is $\Mm_0$.  Thus, each such set is contained in $p_Q^{-1}(\Mm_0)$. Condition \eqref{disjoint} then implies that that the set \eqref{set} is empty unless 
$$d_0=d_1=\dots=d_L.$$
So, for $r=0$
condition \eqref{evr} becomes
\begin{equation*}
\label{ }
p_{M}(u(0,0)) \in \bigcup_{j=0}^L \Big( [W^s(q_1)]_{m_{k_j}} \cap [W^s(q_2)]_{m_{k_j}} \cap  \cdots \cap [W^s(q_n)]_{m_{k_j}}\Big).
\end{equation*}
and the set on the right equals $[\Mm_0]$ by \eqref{union}.
\end{proof}

\begin{Lemma}\
\label{seq}
For every $\ell \in \N$ the space $\Mm_3^{\ell}$ is a compact zero dimensional manifold which is cobordant to $\Mm_2$ and hence is nonempty.
\end{Lemma}

\begin{proof} Property ($G^{r,s}2$) impiles that $\mathrm{cost}(G^{\ell,s}2)=0$. Compactness then follows from this and the fact that the common endpoints of the homotopies, $H^0$ and $H^1$, are finely tuned. For a fixed $\ell \in \N$ the desired cobordism is again obtained using abstract perturbations of the Fredholm section of the appropriate Banach space bundle which cuts out the compact set 
\begin{equation*}
\label{ }
\Mm_3^{[0,\ell]} = \{(r,u) \mid r \in [0, \ell],\, u \in \Mm_3^r\}.
\end{equation*}

\end{proof}

\medskip

\noindent{\bf{End Game.}} By  Lemma  \ref{seq} we can consider a sequence of maps $u_{\ell}$ in $\Mm_3^{\ell}$ for $\ell \in \N$. The $L^2$-energy of each $u_{\ell}$ is bounded by $\Delta(H^0, H^1)< T_{\min}(\lambda_Q) =2\pi$. For $j=1, \dots, n$ consider the sequence of maps $$v^j_{\ell}(s,t)=u_{\ell}(s+j\ell, t),\,\, \ell \in \N.$$
By the uniform energy bound above, each of these $n$ sequences converges in $C^{\infty}_{loc}(\R \times \R/\Z, \R \times M)$  after passing to subsequences.  Denote the  $n^{th}$ limit obtained in this process by 
$v^j $. It follows from condition ($G^{r,s}3$) that $v_j$ satisfies the equation
\begin{equation}
\label{ }
\p_sv^j + J(v^j)(\p_tv^j -V_{G}(v^j)) =0
\end{equation}
and again has $L^2$-energy less than $2\pi$. 
Recall that we have assumed that $\lambda=f\lambda_Q$ has finitely many families of closed Reeb orbits in class $\alpha_{\mathbf{f}}$, $\Xi_1, \dots, \Xi_N$. Recall  also, from Section \ref{proof1}, that the Hamiltonian $\Psi_{\lambda}^*G$ is dividing for $\lambda$ where $\Psi_{\lambda} \colon \R \times M \to \R \times M$  is the diffeomorphism defined  by 
$$
(\tau,p) \mapsto (\tau +f(p), p).
$$
In particular, we have the following.
\begin{Lemma}
\label{Gper}
The nonconstant $1$-periodic orbits of $G$ in class $\alpha_{\mathbf{f}}$ with negative action  are of the form 
\begin{equation}
\label{Gorbit}
x(t)= (\tau^i-f(\xi_i(T_i t)), \xi_i(T_i t))
\end{equation}
for $i=1, \dots N$ where $\tau^i$ is the unique solution of $h'(e^{\tau^i}) =T_i$ in the interval $(0, \ln(1+a))$,  and $\xi_i(t)$ belongs to  the family $\Xi_i $ whose common period is $T_i$.
\end{Lemma}

Hence, for each $j=1, \dots, n$ the limits  $$\lim_{s\to \pm\infty} v^j(s,t) = x^j_{\pm}(t)$$ exist and each $x^j_{\pm}(t)$ is a one periodic orbit of $G$ of the form \eqref{Gorbit}. If all the limits $v^j$ depend nontrivially on $s$, then 
$$\Aa_G(x^1_-)< \Aa_G(x^1_+) < \Aa_G(v^2_+)< \dots< \Aa_G(v^n_+)$$ and the  orbits $$x^1_-, x^1_+, x^2_+, \dots , x^n_+$$  are all distinct. By Lemma \ref{Gper}, the corresponding closed Reeb orbits  $$\xi^1_-, \xi^1_+, \xi^2_+, \dots , \xi^n_+$$ of $\lambda$ are also distinct and so we will be done.

Assume then that one of the limits $v_j$ does not depend on $s$. By Lemma \ref{Gper} we then have
\begin{equation}
\label{ }
v^j(s,t)= v^j(0,t) =  x^j(t) =(\tau^k - f(\xi_k(t+ \theta)), \xi_k(t+ \theta)) 
\end{equation}
for some $1\leq k  \leq N$  and $\theta \in [0,T_k)$. On the other hand we have 
\begin{equation}
\label{ }
v^j(0,0) = \lim_{\ell \to \infty} u_{\ell}(j\ell, 0)
\end{equation}
and $$p_{M}(u_{\ell}(j\ell,0)) \in [W^s(q_j)]_{m_1} \sqcup [W^s(q_j)]_{m_2} \sqcup  \cdots \sqcup [W^s(q_j)]_{m_L}.$$
Together, these conditions imply that 
\begin{equation*}
\label{ }
 \xi_k(\theta) = \lim_{\ell \to \infty} p_{M}(u_{\ell}(j\ell, 0)) \in \overline{\Big([W^s(q_j)]_{m_1} \sqcup [W^s(q_j)]_{m_2} \sqcup  \cdots \sqcup [W^s(q_j)]_{m_L}\Big)}
\end{equation*}
and so 
\begin{equation*}
\label{ }
 p_Q(\xi_k(\theta))  \in \overline{W^s(q_j)}.
\end{equation*}
Since the Morse index of each $q_j$ is two, this  contradicts Corollary \ref{away}. Thus the limit $v^j$ above all depend nontrivially on $s$ and the proof of Theorem \ref{gutt+} is complete.

\section{The Proof of Theorem \ref{fast}}\label{proof2}

Recall the statement of Theorem \ref{fast}.
\begin{Theorem}
Let $(M, \lambda_0)$ be a  contact manifold. For any free homotopy class $\alpha \in [\mathbb{S}^1 ,M]$, and any positive constants $c_1, \,c_2 >0$, there is a contact  form $\lambda=f \lambda_0$ on $M$ such that $\min(f)=1$, 
$
\max(f)<1+c_1
$
and $\lambda$ has a closed Reeb orbit in class $\alpha$ of period less than $c_2.$
\end{Theorem}

We will first give the proof for the case when $M$ is three dimensionsal. Here all the essential ideas are present and unobscured.  We then give the proof for higher dimensional contact manifolds. In each case there are two steps. The first step involves the construction of a new and elementary Wilson semi-plug for Reeb flows which inserts fast closed Reeb orbits. In the second step,  the contact form resulting from the insertion of the new plug is deformed into the desired contact form (in the correct conformal class).

\bigskip

\noindent \textbf{Dimension three.} Suppose that  $(M, \lambda_0)$ is three dimensional. 
Recall that a Legendrian knot in $M$ is an embedded closed curve which is everywhere tangent to $\xi = \mathrm{ker}\lambda_0$. Recall also that there is a Legendrian knot arbitrarily  $C^0$-close to any closed loop in $M$. 
This allows us to start by fixing  a Legendrian knot $L$ in $M$ which represents the class $\alpha$.

We now consider a flow box for the Reeb flow of $\lambda_0$ around $L$ using the following normal neighborhood theorem from \cite{we1, we2}. \footnote{A self-contained and detailed account of the proof is also presented in \cite{ci2} where the relevant result appears as Lemma A5.}

\begin{Theorem}\label{normal}
For every  small enough $\epsilon>0$,  there is a neighborhood $P_{\epsilon}$ of $L$ in $M$ of the form $$\{ (t,x,\theta) \in [-2\epsilon, 2\epsilon] \times [-2 \epsilon, 2\epsilon] \times \R /\Z\}$$ in which  
$$\lambda_0 =dt + x d \theta.$$
\end{Theorem}

The Reeb vector field of $\lambda_0$ in $P_{\epsilon}$ is just $ \p_t$, and so $P_{\epsilon}$ is our flow box. Following \cite{gi} and \cite{ci2}, we now consider a deformation of $\lambda_0$ within $P_{\epsilon}$ of the form 
\begin{equation*}
\label{ }
\lambda_{\delta, \epsilon} = (1- \delta \mathscr{A})dt +\mathscr{B}d\theta.
\end{equation*}
Here, $\mathscr{A}$ and $\mathscr{B}$  are smooth functions of $t$ and $x$ described below,  and  $\delta$ is a suitably small positive constant to be chosen later. Depending on  the context, $\mathscr{A}$ and $\mathscr{B}$ will be considered as functions on either $P_{\epsilon}$ or the square 
$
Q_{\epsilon}= \left[-2\epsilon, 2\epsilon \right] \times \left[-2\epsilon, 2\epsilon \right] .
$

We choose $\mathscr{A}(t,x)$ so that it has the following simple properties.
\begin{itemize}
  \item[($\mathscr{A}$1)] $\mathscr{A}$ is supported in $Q_{\epsilon}$,
  \item[($\mathscr{A}$2)] $-1 < \mathscr{A} \leq 0$,
  \item[($\mathscr{A}$3)] $\mathscr{A}_x(0,\epsilon)=1$ on the rectangle $\left[-\epsilon, \epsilon\right]\times \left[\frac{\epsilon}{2}, \frac{3\epsilon}{2}\right].$
\end{itemize}
The function $\mathscr{B}(t,x)$ is constructed as a perturbation of the function $(t,x) \mapsto x$ of the form
\begin{equation}
\label{Bform}
\mathscr{B}(t,x)=(1-\mathscr{T}(t))x  + \mathscr{T}(t)\mathscr{X}(x),
\end{equation}
where $\mathscr{T}\colon [-2\epsilon, 2\epsilon] \to [0,1]$ is  a smooth  function such that
\begin{itemize}
 \item[($\mathscr{T}$1)] the support of $\mathscr{T}$ is $\left[-\epsilon, \epsilon\right]$,
 \item[($\mathscr{T}$2)]  $\mathscr{T}$ is an even function,
 \item[($\mathscr{T}$3)]  $\mathscr{T}^{-1}(1)=\{0\},$
\end{itemize}
and $\mathscr{X} \colon [-2\epsilon, 2\epsilon] \to [-2\epsilon, 2\epsilon]$ is chosen so that:
\begin{itemize}
  \item[($\mathscr{X}$1)] $\mathscr{X}(x)\geq x$ with equality only outside $\left(\frac{\epsilon}{2}, \frac{3\epsilon}{2}\right)$,
  \item[($\mathscr{X}$2)] $\mathscr{X}'(x) \geq 0$ with equality only at $\epsilon$,
  \item[($\mathscr{X}$3)] $\mathscr{X}(\epsilon)=\epsilon + \epsilon^2$,
  \item[($\mathscr{X}$4)] $\mathscr{X}(x) -x< 2\epsilon^2$.
 \end{itemize}
For these choices the function $\mathscr{B}$  inherits the following properties.
\begin{itemize}
  \item[($\mathscr{B}$1)]  $(0,\epsilon)$ is the only critical point  of $\mathscr{B}$ and the only point where $\mathscr{B}_x$ is not positive.
  \item[($\mathscr{B}$2)]  $0 \leq \mathscr{B}(t,x)-x <2\epsilon^2$ for all $(t,x) \in Q_{\epsilon}$.
 \end{itemize}

By extending $\lambda_{\delta, \epsilon}$ as $\lambda_0$ outside of $P_{\epsilon}$ we may view it as a 1-form on $M$. 

\begin{Lemma}
\label{contact4}
For all sufficiently small $\delta>0$, the form $\lambda_{\delta, \epsilon}$ is contact. 
\end{Lemma}

\begin{proof}
It suffices to check this in $P_{\epsilon}$ where we have
\begin{equation*}
\label{contact}
\lambda_{\delta, \epsilon} \wedge d(\lambda_{\delta, \epsilon})= \left(\mathscr{B}_x(1 -\delta \mathscr{A})+  \delta\mathscr{A}_x \mathscr{B} \right) dt\wedge dx\wedge d\theta.
\end{equation*}
By property ($\mathscr{A}$2)
\begin{equation*}
\label{ }
\mathscr{B}_x(1 -\delta \mathscr{A})+  \delta\mathscr{A}_x \mathscr{B} \geq \mathscr{B}_x+  \delta\mathscr{A}_x \mathscr{B},
\end{equation*}
and so it suffices to show that for all sufficiently small $\delta>0$ the function
$\mathscr{B}_x+  \delta\mathscr{A}_x \mathscr{B}$ is strictly positive on $Q_{\epsilon}$. In the subrectangle   $\left[-\epsilon,\epsilon\right]\times \left[\frac{\epsilon}{2}, \frac{3\epsilon}{2}\right] \subset Q_{\epsilon}$ this follows easily from  ($\mathscr{A}$3) and ($\mathscr{B}$2) which imply
\begin{equation*}
\label{ }
\mathscr{B}_x+  \delta\mathscr{A}_x \mathscr{B} = \mathscr{B}_x+  \delta \mathscr{B} > \delta \frac{\epsilon}{2}.
\end{equation*}
Outside of this subrectangle we have 
\begin{equation*}
\label{ }
\mathscr{B}_x+  \delta\mathscr{A}_x \mathscr{B} = 1+  \delta x\mathscr{A}_x. 
\end{equation*}
Choosing $\delta<|\min(x\mathscr{A}_x )|^{-1}$ we are done.
\end{proof}

The crucial feature of the new contact form $\lambda_{\delta, \epsilon}$ is the following.

\begin{Lemma}\label{new orbit}
There is exactly one $\R /\Z$-family of simple periodic Reeb orbits of $\lambda_{\delta, \epsilon}$, $\Gamma_{\delta, \epsilon}$, which is contained in $P_{\epsilon}$. The orbits  in the family $\Gamma_{\delta, \epsilon}$ represent  the class $\alpha$, and their (common) period is $2\pi(\epsilon+\epsilon^2)$.  
\end{Lemma}

\begin{proof}

The kernel of $d\lambda_{\delta, \epsilon}$ in $P_{\epsilon}$ is spanned by the vector field
\begin{equation*}
\label{ }
K=\mathscr{B}_x \p_t -\mathscr{B}_t\p_x +\delta \mathscr{A}_x\p_\theta.
\end{equation*}
Since $(0,\epsilon)$ is the only critical point of $\mathscr{B}$, is follows easily from the formula for $K$ that the embedded circle 
$
S_{\epsilon} =\{0\}\times \{\epsilon\} \times \R /\Z
$
 corresponds to a unique $\R /\Z$-family, $\Gamma_{\delta, \epsilon}$, of simple periodic Reeb orbits of $\lambda_{\delta, \epsilon}$.  Since the $\p_t$-component of $K$ is  positive away from $
S_{\epsilon}
$, there are no other closed Reeb orbits in $P_{\epsilon}$.

Since $S_{\epsilon}$ is $C^\infty$-close to $\{0\}\times \{0\} \times \R /\Z$ in $P_{\epsilon}$, it  is also  $C^\infty$-close to our original Legendgrian knot $L$. Hence orbits in $\Gamma_{\delta, \epsilon}$ represent $\alpha$. Finally, the common periods of these orbits is 
\begin{equation}
\label{action}
\left|\int_{S_{\epsilon}} \lambda_{\delta, \epsilon}\right| =2\pi \mathscr{B}(0,\epsilon) =2\pi(\epsilon+\epsilon^2).
\end{equation}
\end{proof}

\begin{Remark}\label{plug}
The deformation of  $\lambda_0$ to $\lambda_{\delta, \epsilon}$ within the flow box $P_{\epsilon}$ corresponds to the insertion of a Wilson semi-plug.  Here, the {\it semi} refers to the fact that our plug fails to have the {\it matching endpoint property}.
This means that a trajectory which enters $P_{\epsilon}$ at  $(-2\epsilon, x, \theta)$, then follows our new Reeb flow and exits $P_{\epsilon}$ at $t=2\epsilon$, does not need to do so at the point $(2\epsilon, x, \theta)$ (as the trajectories of the the original Reeb vector field $R_{\lambda_0}$ did). Thus it is possible, indeed inevitable in some cases, that we have created new closed  Reeb orbits which pass through $P_{\epsilon}$ but are not contained therein. This can not be remedied.  As shown by  Ana Rechtman already in her thesis, \cite{re}, Sullivan's characterization of geodesible vector fields from \cite{su} implies that no Wilson plug is geodesible  and hence no Wilson plug is Reeb. In dimension three,  a similar conclusion was also reached by Cieliebak in \cite{ci2}. There it is observed that if one could construct a Wilson plug for Reeb flows (centered about an arbitrary Legendrian knot)  then one could immediately construct a counterexample to the following result.

\begin{Theorem}(Hofer-Wysocki-Zehnder, \cite{hwz})\label{hwz} Every Reeb vector field  on $\mathbb{S}^3$, 
has a periodic orbit which is unknotted and has self-linking number $-1$.
\end{Theorem}
\end{Remark}

Having inserted the desired family of closed Reeb orbits with our semi-plug we must now reckon with the fact that the form $\lambda_{\delta, \epsilon}$ need not be of the form  $f\lambda_0$ for some positive function $f$.

\begin{Lemma}\label{diffeo}
For all sufficiently small $\delta>0$ there exists a diffeomorphism $\Psi$ of $M$ such that $\Psi$ is isotopic to the identity and $\Psi^*{\lambda_{\delta, \epsilon}} =f\lambda_0$ for some positive function $f$ such that $\min(f)=1$ and 
\begin{equation*}
\label{ }
\max(f)< e^{2\delta +4\epsilon}.
\end{equation*}
\end{Lemma}

\begin{proof}
The desired diffeomorphism will be constructed as a composition of two maps. Each of these will be obtained using Gray's Stability Theorem. We begin by refining the standard phrasing of this result to include some relevant quantitative information which is easily extracted from the standard proof. 

\begin{Theorem}[Quantitative Gray Stability]
Let $(\lambda_s)_{s\in [0,1]}$ be a smooth family of contact forms on $M$. Denote their Reeb vector fields by $R_s$ and set 
\begin{equation*}
\label{ }
r_s=i_{R_s}\left( \frac{d}{ds} \lambda_s \right).
\end{equation*}
There exists a one parameter family of diffeomorphisms $(\psi_s)_{s\in[0,1]}$ of $M$ which starts at the identity and satisfies
\begin{equation*}
\label{ }
\psi_s^*(\lambda_s) = \exp\left( \int_0^s r_{\sigma} \circ \psi_{\sigma} \, d\sigma\right) \lambda_0.
\end{equation*}
Hence,  
\begin{equation*}
\label{ }
\psi_1^*(\lambda_1) = f_1 \lambda_0,
\end{equation*}
where 
\begin{equation*}
\label{ }
\min(f_1) \geq  \exp\left( \int_0^1 \min_{M} r_{\sigma}  \, d\sigma\right)
\end{equation*}
and 
\begin{equation*}
\label{ }
\max(f_1) \leq  \exp\left( \int_0^1 \max_{M} r_{\sigma}  \, d\sigma\right).
\end{equation*}
\end{Theorem}

To obtain the first of our maps we apply this version of Gray's Theorem to the family of $1$-forms $$\bar{\lambda}_{s}=(1- s \delta \mathscr{A})dt +x d\theta. $$
Since $$\bar{\lambda}_{s} \wedge d \bar{\lambda}_{s} = (1+\delta s(x\mathscr{A}_x -\mathscr{A})) dt \wedge dx\wedge d\theta$$
this is a family of contact forms for all sufficiently small $\delta>0$. Their Reeb vector fields are given by 
\begin{equation*}
\label{ }
\bar{R}_{s} = \frac{1}{1+\delta s(x\mathscr{A}_x -\mathscr{A})} \left(\p_t + s\delta\mathscr{A}_x \p_{\theta}\right).
\end{equation*}
We then have
\begin{equation*}
\label{ }
\bar{r}_s=i_{\bar{R}_s}\left( \frac{d}{ds} \bar{\lambda}_s \right) = \frac{-\delta \mathscr{A}}{1+\delta s(x\mathscr{A}_x -\mathscr{A})}.
\end{equation*}
Property ($\mathscr{A}$2) implies that the numerator in the last expression is in $[0,\delta)$. It is also clear that for all sufficiently small $\delta>0$ the denominator is greater that $1/2$. Thus for all such small $\delta$ we have 
\begin{equation*}
\label{ }
0\leq \bar{r}_s< 2\delta 
\end{equation*} 
and Gray's Theorem (as stated above) yields a diffeomorphism $\bar{\psi}_1$ such that 
\begin{equation*}
\label{ }
\bar{\psi}_1^*(\bar{\lambda}_1) = \bar{f}_1 \lambda_0,
\end{equation*}
and  
\begin{equation}
\label{barr}
1 \leq \bar{f}_1 < e^{2\delta}.
\end{equation}

To obtain our second map we now consider the family of $1$-forms $$\widehat{\lambda}_{s}=\bar{\lambda}_1+ s(\mathscr{B}-x) d\theta =(1-\delta\mathscr{A})dt + \mathscr{B}^s d\theta$$
where $\mathscr{B}^s= x+ s(\mathscr{B}-x) $. These forms are contact  whenever the functions  $\mathscr{B}^s_x(1 -\delta \mathscr{A})+  \delta\mathscr{A}_x \mathscr{B}^s$ are strictly positive. Arguing as in Lemma \ref{contact4} one can easily verify that this holds for all sufficiently small $\delta$. Assuming this then, the Reeb vector field of  $\widehat{\lambda}_{s}$ is 
\begin{equation*}
\label{ }
\widehat{R}_{s} = \frac{1}{\mathscr{B}^s_x(1 -\delta \mathscr{A})+  \delta\mathscr{A}_x \mathscr{B}^s } \left(\mathscr{B}^s_x\p_t -\mathscr{B}^s_x \p_x+ \delta\mathscr{A}_x \p_{\theta}\right)
\end{equation*}
and the relevant  family of functions  is 
\begin{equation*}
\label{ }
\widehat{r}_s=i_{\widehat{R}_s}\left( \frac{d}{ds} \widehat{\lambda}_s \right) = \frac{(\mathscr{B}-x)\delta \mathscr{A}_x}{\mathscr{B}^s_x(1 -\delta \mathscr{A})+  \delta\mathscr{A}_x \mathscr{B}^s}.
\end{equation*}
Outside the subrectangle  $\left[-\epsilon,\epsilon\right]\times \left[\frac{\epsilon}{2}, \frac{3\epsilon}{2}\right] \subset Q_{\epsilon}$, where $B=x$, we have $\widehat{r}_s=0$.
Inside this subrectangle, we have $\mathscr{A}_x=1$ and so  
\begin{equation}
\label{inside}
\widehat{r}_s = \frac{(\mathscr{B}-x)\delta}{\mathscr{B}^s_x(1 -\delta \mathscr{A})+  \delta \mathscr{B}^s}
\end{equation}
which is nonnegative by property ($\mathscr{B}$2). Thus each function $\widehat{r}_s$ is nonnegative on all of $Q_{\epsilon}$.

To obtain an upper bound for the $\widehat{r}_s$ it suffices to do so inside $\left[-\epsilon,\epsilon\right]\times \left[\frac{\epsilon}{2}, \frac{3\epsilon}{2}\right]$. It follows from \eqref{inside} and  properties ($\mathscr{A}$2) and ($\mathscr{B}$2) that here we have 
\begin{equation*}
\label{ }
\frac{(\mathscr{B}-x)\delta}{\mathscr{B}^s_x(1 -\delta \mathscr{A})+  \delta \mathscr{B}^s}< \frac{2 \epsilon^2 \delta}{\mathscr{B}^s_x+  \delta \mathscr{B}^s}.
\end{equation*} 
We also have $\mathscr{B}^s_x\geq0$  and $\mathscr{B}^s\geq \frac{\epsilon}{2}$ in this subrectangle and so \begin{equation*}
\label{ }
0 \leq \widehat{r}_s<4 \epsilon.
\end{equation*}
Thus, for all sufficiently small $\delta>0$ there is a diffeomorphism $\widehat{\psi}_1$ such that 
\begin{equation*}
\label{ }
\widehat{\psi}_1^*(\widehat{\lambda}_1) = \widehat{f}_1 \bar{\lambda}_1,
\end{equation*}
and  
\begin{equation}
\label{hatt}
1 \leq \widehat{f}_1 < e^{4\epsilon}.
\end{equation}

To conclude we set $\Psi = \widehat{\psi} \circ \bar{\psi}$. Since both factors are isotopic to the identity so is $\Psi$. We  also have 
\begin{eqnarray*}
\Psi^*{\lambda_{\delta, \epsilon}} & = & (\widehat{\psi} \circ \bar{\psi})^*(\widehat{\lambda}_1) \\
{} & = &\bar{\psi}^*(\widehat{f}_1 \bar{\lambda}_1) \\
{} & = &(\widehat{f}_1\circ \bar{\psi}) \bar{f}_1\lambda_0.  
\end{eqnarray*}
By  \eqref{barr} and \eqref{hatt} the function $(\widehat{f}_1\circ \bar{\psi}) \bar{f}_1$ satisfies
\begin{equation*}
\label{ }
1\leq (\widehat{f}_1\circ \bar{\psi}) \bar{f}_1 \leq e^{2\delta+4\epsilon}.
\end{equation*}
\end{proof}

At this point we can finish the proof of Theorem \ref{fast} in the three dimension case. Choose $\delta$ and $\epsilon$ so that  $$2\delta +4\epsilon<c_1$$ and $$2\pi(\epsilon+\epsilon^2) < c_2.$$ Assuming also that $\delta$ is small enough for Lemma \ref{diffeo} to hold we set $$\lambda= \Psi^*{\lambda_{\delta, \epsilon}}.$$ 
Given an orbit  $\gamma_{\delta, \epsilon}$ in the family $\Gamma_{\delta, \epsilon}$ from Lemma \ref{new orbit} the closed curve $\Psi^{-1}(\gamma_{\delta, \epsilon}(t))$ is then a closed Reeb orbit of $\lambda$ with (the same) period, $2\pi(\epsilon+ \epsilon^2)$. Since $\Psi$ is isotopic to the identity, the periodic orbit $\Psi^{-1}(\gamma_{\delta, \epsilon}(t))$ also represents the class $\alpha$ and, with this, we are done.

\begin{Remark}
Rather than looking for specific types of periodic orbits that are  forced to exist by the Reeb condition, as in Theorem \ref{hwz}, one might instead ask (in the spirit of \cite{eg}); What collections of periodic orbits can be generated by  Reeb vector fields on a fixed three manifolds? In this direction the construction above yields the following result.
\begin{Theorem}
 For any cooriented contact three manifold $(M, \xi)$ and any link type $\mathscr{L}$ in $M$, there is a contact form $\lambda \in \Lambda(\xi)$ whose  Reeb vector field has a collection of closed periodic orbits which represent $\mathscr{L}$.
\end{Theorem}

\end{Remark}

\medskip

\noindent \textbf{Higher dimensions.} We now show that the proof above for dimension three extends easily to higher dimensions. In practice, this amounts to showing that the extra variables can be cut-off with no meaningful effects. 

Consider then a contact manifold $(M, \lambda_0)$ of dimension $2n-1>3$ and let $L$ be a simple closed curve in $M$ which is everywhere tangent to $\xi$ and represents the class $\alpha$. Theorem \ref{normal} generalizes in the obvious way and yields, for small enough $\epsilon>0$,   a normal neighborhood of $L$ in $M$ of the form 
$$P_{\epsilon}= \{(t,x,\theta,z) \in [-2\epsilon, 2\epsilon] \times [-2 \epsilon, 2\epsilon] \times \R /\Z  \times B^{2n-4}(\epsilon)\}$$  in which  
$$\lambda_0 =dt + x d \theta +\kappa_0.$$
Here, $B^{2n-4}(\epsilon)$ is the open unit ball of radius $\epsilon$ in $\R^{2n-4}$, we have 
 $$
 z = ((q_1,p_1), \dots, (q_{n-2}, p_{n-2})),
 $$
 and $\kappa_0$ is the standard Liouville form 
 \begin{equation*}
\label{ }
\kappa_0=\frac{1}{2}\sum_{i=1}^{n-2} q_i dp_i -p_idq_i.
\end{equation*}

In what follows we will only need to consider functions whose $z$-dependence is radial, i.e., which depend only on 
$$
 \rho= \frac{|z|}{2}^2.
 $$
Our tool to cut-off the extra variables will the be a smooth function $\mathrm{\mathbf{ cut}} \colon [0, \frac{\epsilon^2}{2}] \to [0,1]$ with the following properties:
\begin{itemize}
  \item[($\mathrm{\mathbf{c}}$1)] $\mathrm{\mathbf{ cut}} =0$ near $\frac{\epsilon^2}{2}$,
  \item[($\mathrm{\mathbf{c}}$2)] $\mathrm{\mathbf{ cut}}(\rho)=1-\rho$ near $0$,
  \item[($\mathrm{\mathbf{c}}$3)] $-\frac{4}{\epsilon^2} <\mathrm{\mathbf{ cut}}'(\rho) \leq 0$.
\end{itemize}

Given the functions $\mathscr{A}(t,x)$ and  $\mathscr{B}(t,x)$ from the previous section we set 
\begin{equation*}
\label{ }
\widehat{\mathscr{A}}(t,x,\rho)=\mathrm{\mathbf{ cut}}(\rho)\mathscr{A}(t,x),
\end{equation*}
and 
\begin{equation*}
\label{ }
\widehat{\mathscr{B}}(t,x,\rho)=(1-\mathrm{\mathbf{ cut}}(\rho))x + \mathrm{\mathbf{ cut}}(\rho)\mathscr{B}(t,x).
\end{equation*}
Now $(0,\epsilon,0)$ is the only critical point of $\widehat{\mathscr{B}}$  and the only point at which $\widehat{\mathscr{B}}_x$ fails to be positive.

As before, for $\delta>0$, we consider deformations of $\lambda_0$ of the form,
\begin{equation*}
\label{ }
\lambda_{\delta, \epsilon} = (1-\delta \widehat{\mathscr{A}})dt + \widehat{\mathscr{B}}d\theta + \kappa_0.
\end{equation*}

\begin{Lemma}
\label{contact6}
For all sufficiently small $\delta>0$, the form $\lambda_{\delta, \epsilon}$ is contact. 
\end{Lemma}

\begin{proof}
Simple computations yield 
\begin{equation*}
\label{}
d\lambda_{\delta, \epsilon}= \delta \widehat{\mathscr{A}}_xdt \wedge dx+\delta\widehat{\mathscr{A}}_{\rho} dt \wedge d\rho +\widehat{\mathscr{B}}_t dt \wedge d\theta +\widehat{ B}_x dx \wedge d \theta+ \widehat{\mathscr{B}}_{\rho} d\rho \wedge d\theta  +d\kappa_0
\end{equation*}
and 
\begin{eqnarray*}
(d\lambda_{\delta, \epsilon})^{n-1} & = & (n-1)\left(\delta \widehat{\mathscr{A}}_x dt \wedge dx + \widehat{\mathscr{B}}_t dt \wedge d\theta + \widehat{\mathscr{B}}_x dx \wedge d \theta \right) \wedge (d \kappa_0)^{n-2}\\
{} & {} & +(n-1)(n-2) \delta\left(\widehat{\mathscr{A}}_x\widehat{\mathscr{B}}_{\rho}-\widehat{\mathscr{A}}_{\rho}\widehat{\mathscr{B}}_x \right) dt \wedge dx \wedge d\rho \wedge d\theta \wedge \kappa_0 \wedge (d \kappa_0)^{n-3}.
\end{eqnarray*}
Using the identity 
\begin{equation*}
\label{ }
d\rho \wedge \kappa_0 \wedge (d\kappa_0)^{n-3} = \rho (d \kappa_0)^{n-2}
\end{equation*}
we then arrive at the following expression for $\lambda_{\delta, \epsilon} \wedge(d\lambda_{\delta, \epsilon})^{n-1}$,
\begin{equation*}
\label{ }
 (n-1)\left[ ((1-\delta\widehat{\mathscr{A}})\widehat{\mathscr{B}}_x +\delta\widehat{\mathscr{A}}_x\widehat{\mathscr{B}})+ (n-2)\rho \delta(\widehat{\mathscr{A}}_x\widehat{\mathscr{B}}_{\rho}-\widehat{\mathscr{A}}_{\rho}\widehat{\mathscr{B}}_x)\right] dt \wedge dx \wedge d\theta \wedge (d \kappa_0)^{n-2}.
\end{equation*}
The function $((1-\delta\widehat{\mathscr{A}})\widehat{\mathscr{B}}_x +\delta\widehat{\mathscr{A}}_x\widehat{\mathscr{B}})+ (n-2)\rho \delta(\widehat{\mathscr{A}}_x\widehat{\mathscr{B}}_{\rho}-\widehat{\mathscr{A}}_{\rho}\widehat{\mathscr{B}}_x)$ can be rewritten in the form 
$$
\widehat{\mathscr{B}}_x + \delta \mathscr{E}
$$
where $\mathscr{E}(0,\epsilon,0) = \widehat{\mathscr{B}}(0, \epsilon,0)= \epsilon+ \epsilon^2>0$.  Since $\widehat{\mathscr{B}}_x \geq 0$ with equality only at the point $(0, \epsilon,0)$ it follows from continuity that for all sufficiently small $\delta>0$ the form $\lambda_{\delta, \epsilon} \wedge(d\lambda_{\delta, \epsilon})^{n-1}$ is nonvanishing and hence $\lambda_{\delta, \epsilon}$ is a contact form.
\end{proof}

\begin{Lemma}\label{new orbit6}
There is exactly one $\R /\Z$-family of simple periodic Reeb orbits of $\lambda_{\delta, \epsilon}$, $\Gamma_{\delta, \epsilon}$, which is contained in $P_{\epsilon}$. The orbits  in the family $\Gamma_{\delta, \epsilon}$ represent  the class $\alpha$, and their (common) period is $2\pi(\epsilon+\epsilon^2)$.  
\end{Lemma}

\begin{proof}
For a fixed $t$ and $x$,  let $V^{t,x}$ be the Hamiltonian vector field on $(B^{2n-4}(\epsilon), d\kappa_0)$ defined by the function $z \mapsto \delta\widehat{\mathscr{A}}(t,x,|z|^2/2)$. That is $V^{t,x}(z)$ is defined by the equation
\begin{equation*}
\label{ }
d\kappa_0(z)(V^{t,x}(z), \cdot) = \mathrm{\mathbf{ cut}}'(|z|^2/2)\delta\mathscr{A}(t,x) d\rho (\cdot).
\end{equation*}
Let  $V(t,x,z,\theta)$ be the vector field whose  projection to $(B^{2n-4}(\epsilon), d\kappa_0)$ is  $V^{t,x}$ and whose other components are trivial. Define the vector field  $U$ by replacing $\delta\widehat{\mathscr{A}}$ above by $\widehat{\mathscr{B}}$. The kernel of $d\lambda_{\delta, \epsilon}$ is then spanned by the vector field
\begin{equation*}
\label{ }
K=\widehat{\mathscr{B}}_x \p_t -\widehat{\mathscr{B}}_t\p_x +\delta\widehat{\mathscr{A}}_x\p_\theta + \delta\widehat{\mathscr{A}}_xU-\widehat{\mathscr{B}}_x V.
\end{equation*}
To see this,  note first that  $d \rho(V) = d \rho(U) =0$ since the corresponding Hamiltonian vector fields are defined by functions of $\rho$. A simple computation then yields
\begin{eqnarray*}
i_{K}d\lambda_{\delta, \epsilon} & = &\delta\widehat{\mathscr{A}}_{\rho}\widehat{\mathscr{B}}_x d\rho-\delta\widehat{\mathscr{B}}_{\rho}\widehat{\mathscr{A}}_x d\rho + i_{\delta\widehat{\mathscr{A}}_xU} d\kappa_0 - i_{{\widehat{\mathscr{B}}}_x V} d\kappa_0 \\
{} & = & \delta\widehat{\mathscr{A}}_{\rho}\widehat{\mathscr{B}}_x d\rho-\delta\widehat{\mathscr{B}}_{\rho}\widehat{\mathscr{A}}_x d\rho + \delta\widehat{\mathscr{A}}_x \widehat{\mathscr{B}}_{\rho }d\rho- \delta{\widehat{\mathscr{B}}}_x \widehat{\mathscr{A}}_{\rho}d \rho\\ 
{} & = & 0.
\end{eqnarray*}

The $t$-component of $K$ vanishes only when $(t,x,z) =(0,\epsilon,0)$. At  this point both $\widehat{\mathscr{B}}_x$ and $\widehat{\mathscr{B}}_t$ vanish as do   $V$ and $U$ since $d \rho=0$ when $z=0$. 
Thus $\lambda_{\delta, \epsilon}$ has exactly one $\R /\Z$-family of simple closed Reeb orbits in $P_{\epsilon}$. These orbits all have image
$$
S_{\epsilon} = \{0\}\times \{\epsilon\} \times \R /\Z \times \{0\}
$$
and period 
\begin{equation*}
\label{}
\left|\int_{S_{\epsilon}} \lambda_{\delta, \epsilon}\right| =2\pi(\epsilon+ \epsilon^2).
\end{equation*}

\end{proof}

\begin{Lemma}\label{diffeo6}
For all sufficiently small $\delta>0$ there exists a diffeomorphism $\Psi$ of $M$ such that $\Psi$ is isotopic to the identity and $\Psi^*{\lambda_{\delta, \epsilon}} =f\lambda_0$ for some positive function $f$ such that $\min(f)=1$ and 
\begin{equation*}
\label{ }
\max(f)< e^{2\delta +4\epsilon}.
\end{equation*}
\end{Lemma}

\begin{proof}

As in the proof of Lemma \ref{diffeo} we first apply the quantitative version of Gray's Theorem to the family of $1$-forms $$\bar{\lambda}_{s}=(1- s \delta \widehat{\mathscr{A}})dt +x d\theta + \kappa_0.$$
These are contact for all sufficiently small $\delta>0$ and their Reeb vector fields are given by 
\begin{equation*}
\label{ }
\bar{R}_{s} = \frac{1}{(n-1)\left[ 1-s\delta\widehat{\mathscr{A}} +s\delta\widehat{\mathscr{A}}_x x- s(n-2)\rho \delta\widehat{\mathscr{A}}_{\rho}\right]} \left( \p_t +s\delta\widehat{\mathscr{A}}_x\p_\theta + s\delta\widehat{\mathscr{A}}_xY-X\right).
\end{equation*}
The relevant functions are then 
\begin{equation*}
\label{ }
\bar{r}_s=i_{\bar{R}_s}\left( \frac{d}{ds} \bar{\lambda}_s \right) = \frac{-\delta \widehat{\mathscr{A}}}{(n-1)\left[ 1-s\delta\widehat{\mathscr{A}} +s\delta\widehat{\mathscr{A}}_x x- s(n-2)\rho \delta\widehat{\mathscr{A}}_{\rho}\right]}.
\end{equation*}
Property ($\mathscr{A}$2) implies that the numerator in the last expression is in $[0,\delta)$. It is also clear that for all sufficiently small $\delta>0$ the denominator is greater that $1/2$. Thus for all such small $\delta$ we have 
\begin{equation*}
\label{ }
0\leq \bar{r}_s< 2\delta
\end{equation*} 
and Gray's Theorem (as stated above) yields a diffeomorphism $\bar{\psi}_1$ such that 
\begin{equation*}
\label{ }
\bar{\psi}_1^*(\bar{\lambda}_1) = \bar{f}_1 \lambda_0,
\end{equation*}
and  
\begin{equation}
\label{bar}
1 \leq \bar{f}_1 < e^{2\delta}.
\end{equation}

Next we consider the family of $1$-forms $$\widehat{\lambda}_{s}=\bar{\lambda}_1+ s(\widehat{\mathscr{B}}-x) d\theta =(1-\delta\widehat{\mathscr{A}})dt +\widehat{ \mathscr{B}}^s d\theta.$$ They are also contact for all sufficiently small $\delta$ and the Reeb vector field $\widehat{R}_{s}$ of  $\widehat{\lambda}_{s}$ is equal to 
\begin{equation*}
\label{ }
\widehat{\mathscr{B}}^s_x \p_t -\widehat{\mathscr{B}}^s_t\p_x +\delta\widehat{\mathscr{A}}_x\p_\theta + \delta\widehat{\mathscr{A}}_xY-\widehat{\mathscr{B}}^s_x X
\end{equation*}
multiplied by the function
\begin{equation*}
\label{ }
\left((n-1)\left[ ((1-\delta\widehat{\mathscr{A}})\widehat{\mathscr{B}}^s_x +\delta\widehat{\mathscr{A}}_x\widehat{\mathscr{B}}^s)+ (n-2)\rho \delta(\widehat{\mathscr{A}}_x\widehat{\mathscr{B}}^s_{\rho}-\widehat{\mathscr{A}}_{\rho}\widehat{\mathscr{B}}^s_x)\right]\right)^{-1}.
\end{equation*}
This makes the relevant  family of functions 
\begin{equation*}
\label{ }
\widehat{r}_s= \frac{(\widehat{\mathscr{B}}-x)\delta \widehat{\mathscr{A}}_x}{(n-1)\left[ ((1-\delta\widehat{\mathscr{A}})\widehat{\mathscr{B}}^s_x +\delta\widehat{\mathscr{A}}_x\widehat{\mathscr{B}}^s)+ (n-2)\rho \delta(\widehat{\mathscr{A}}_x\widehat{\mathscr{B}}^s_{\rho}-\widehat{\mathscr{A}}_{\rho}\widehat{\mathscr{B}}^s_x)\right]}.
\end{equation*}
Outside of $\left[-\epsilon,\epsilon\right]\times \left[\frac{\epsilon}{2}, \frac{3\epsilon}{2}\right] \times B^{2n-4}(\epsilon)$, where $\widehat{B}=x$, we have $\widehat{r}_s=0$.
Inside this region, we have $\widehat{\mathscr{A}}_x=\mathrm{\mathbf{ cut}}(\rho)$ and so  
\begin{equation*}
\label{inside6}
\widehat{r}_s = \frac{(\widehat{\mathscr{B}}-x)\delta \mathrm{\mathbf{ cut}}(\rho)}{(n-1)\left[ ((1-\delta\widehat{\mathscr{A}})\widehat{\mathscr{B}}^s_x +\delta\mathrm{\mathbf{ cut}}(\rho)\widehat{\mathscr{B}}^s)+ (n-2)\rho \delta(\mathrm{\mathbf{ cut}}(\rho)\widehat{\mathscr{B}}^s_{\rho}-\mathrm{\mathbf{ cut}}'(\rho)\widehat{\mathscr{A}}\widehat{\mathscr{B}}^s_x)\right]}.
\end{equation*}
Property ($\mathscr{B}$2) implies that this expression is nonnegative and so each function $\widehat{r}_s$ is nonnegative on all of $Q_{\epsilon}\times B^{2n-4}(\epsilon)$.

To obtain the desired upper bound for the $\widehat{r}_s$ it suffices to do so inside $\left[-\epsilon,\epsilon\right]\times \left[\frac{\epsilon}{2}, \frac{3\epsilon}{2}\right] \times B^{2n-4}(\epsilon)$. Here we have 
\begin{eqnarray*}
\widehat{r}_s & = &  \frac{(\widehat{\mathscr{B}}-x)\delta \mathrm{\mathbf{ cut}}(\rho)}{(n-1)\left[ (1-\delta\widehat{\mathscr{A}})\widehat{\mathscr{B}}^s_x +\delta\mathrm{\mathbf{ cut}}(\rho)\widehat{\mathscr{B}}^s+ (n-2)\rho \delta(\mathrm{\mathbf{ cut}}(\rho)\widehat{\mathscr{B}}^s_{\rho}-\mathrm{\mathbf{ cut}}'(\rho)\widehat{\mathscr{A}}\widehat{\mathscr{B}}^s_x)\right]} \\
{} & = & \frac{(\widehat{\mathscr{B}}-x)\delta \mathrm{\mathbf{ cut}}(\rho)}{(n-1)\left[ \left(1-\delta\widehat{\mathscr{A}}\left(1+(n-2) \rho\mathrm{\mathbf{ cut}}(\rho)\mathrm{\mathbf{ cut}}'(\rho)\right)\right)\widehat{\mathscr{B}}^s_x +\delta\mathrm{\mathbf{ cut}}(\rho)\left(\widehat{\mathscr{B}}^s+ (n-2)\rho\widehat{\mathscr{B}}^s_{\rho}\right)\right]}. 
\end{eqnarray*}
For all sufficiently small $\delta$ we may assume that the coefficient of $\widehat{\mathscr{B}}^s_x$ in the denominator is greater that $1/2$.
Using this and the formula defining $\widehat{\mathscr{B}}^s$ we get 
\begin{equation*}
\label{ }
\widehat{r}_s \leq  \frac{(\widehat{\mathscr{B}}-x)\delta \mathrm{\mathbf{ cut}}(\rho)}{(n-1)\left[ \frac{1}{2}\widehat{\mathscr{B}}^s_x +\delta\mathrm{\mathbf{ cut}}(\rho)\Big(x+ \left[\mathrm{\mathbf{ cut}}(\rho)+(n-2) \rho \mathrm{\mathbf{ cut}}'(\rho)\right] s(B-x)\Big)\right]}.
\end{equation*}
By  condition ($\mathscr{B}$2) the function $B-x$  (and thus $\widehat{\mathscr{B}}-x$) takes values in $[0, 2\epsilon^2).$ For sufficiently small $\epsilon>0$ we may therefore assume that 
\begin{equation*}
\label{ }
\Big|\left[\mathrm{\mathbf{ cut}}(\rho)+(n-2) \rho \mathrm{\mathbf{ cut}}'(\rho)\right] (B-x)\Big| < \frac{\epsilon}{4}.
\end{equation*}
(Here we have used the fact that  ($\mathrm{\mathbf{c}}$3) implies that $0\geq \rho \mathrm{\mathbf{ cut}}'(\rho) >-2.$) It then follows that on the subset of interest,  $\left[-\epsilon,\epsilon\right]\times \left[\frac{\epsilon}{2}, \frac{3\epsilon}{2}\right] \times B^{2n-4}(\epsilon)$, where $x \geq \frac{\epsilon}{2}$ we have 
\begin{equation*}
\label{ }
\widehat{r}_s <  \frac{2\epsilon^2\delta \mathrm{\mathbf{ cut}}(\rho)}{(n-1)\left[ \frac{1}{2}\widehat{\mathscr{B}}^s_x +\delta\mathrm{\mathbf{ cut}}(\rho)\frac{\epsilon}{4}\right]}.
\end{equation*}
Using the fact that  $\widehat{\mathscr{B}}^s_x \geq 0$ and $n \geq 3$ we arrive at the upper bound
\begin{equation*}
\label{ }
\widehat{r}_s  < 4\epsilon.
\end{equation*}
Thus for all sufficiently small $\delta>0$ and $\epsilon>0$, it follows from Gray's Theorem that  there is a diffeomorphism $\widehat{\psi}_1$ such that 
\begin{equation*}
\label{ }
\widehat{\psi}_1^*(\widehat{\lambda}_1) = \widehat{f}_1 \bar{\lambda}_1,
\end{equation*}
and  
\begin{equation*}
\label{hat}
1 \leq \widehat{f}_1 < e^{4\epsilon}.
\end{equation*}

The desired diffeomorphism is then  $\Psi = \widehat{\psi} \circ \bar{\psi}$. 
\end{proof}

The rest of the proof of Theorem \ref{fast} now follows exactly as it did for dimension three.


\begin{thebibliography}{BEHWZ}

\bibitem[AM]{am}
M. Abreu, L. Macarini, Multiplicity of periodic orbits for dynamically convex contact forms,  arXiv:1509.08441.

\bibitem[AFM]{afm}
P. Albers, U. Fuchs, W.J. Merry, Orderability and the Weinstein Conjecture, \textit{Compositio Mathematica},  
DOI: 10.1112/S0010437X15007642 (2015), 22 pages. 

\bibitem[AH]{ah}
P. Albers, D. Hein, Cuplength Estimates in Morse cohomology, \textit{J. Topol. Anal.},  DOI: 10.1142/S1793525316500102 (2015), 1--30.

\bibitem[AM]{am}
P. Albers, A. Momin, Cup-length estimates for leaf-wise intersections, \textit{Mathematical Proceedings of the Cambridge Philosophical Society}, \textbf{149} (2010), 539--551.


\bibitem[BLMR]{blmr}
H. Berestycki, J-M. Lasry, G. Mancini, B Ruf, Existence of multiple periodic orbits on star-shaped hamiltonian surfaces, \textit{Communications on Pure and Applied Mathematics}, \textbf{38} (1985), 253--289.


\bibitem[BC]{bc}
M. Betz, R. Cohen, Graph Moduli spaces and cohomology operations, \textit{Turkish J. of Math.}, \textbf{ 18} (1994) 23--41.


\bibitem[BPS]{bps}
P. Biran, L. Polterovich, D. Salamon, Propagation in Hamiltonian dynamics and relative symplectic homology, \textit{Duke Math. J.}, \textbf{119} (2003), 65--118.



\bibitem[BEHWZ]{behwz}
F. Bourgeois, Y. Eliashberg, H. Hofer, K. Wysocki, and E. Zehnder,
 Compactness results in symplectic field theory,
{\em Geom. Topol.}, \textbf{7} (2003), 799--888.

\bibitem[Bo]{bou}
F. Bourgeois, \textit{A Morse-Bott approach to contact homology}, Symplectic and contact topology: interactions and perspectives, 55?77, Fields Inst. Commum. 35, AMS, 2003.


\bibitem[BO]{bo}
F. Bourgeois, A. Oancea, 
Symplectic Homology, autonomous Hamiltonians, and Morse-Bott moduli spaces, 
\textit{Duke Mathematical Journal}, \textbf{146} (2009), 71--174.


\bibitem[Ch]{ch}
Y. V. Chekanov, Lagrangian intersections, symplectic energy, and areas of holomorphic curves, \textit{Duke Math. J.}, \textbf{95} (1998), 213--226.

\bibitem[Ci1]{ci1} 
K. Cieliebak,
Pseudo-holomorphic curves and periodic orbits on cotangent bundles , \textit{Journal de Mathematiques Pures et Appliquees},  \textbf{73} (1994), 251--278.

\bibitem[Ci2]{ci2} 
K. Cieliebak,
Symplectic boundaries: Creating and destroying closed characteristics, \emph{Geometric and Functional Analysis}, \textbf{7} (1997),  269--321.


\bibitem[CM]{cm} 
K. Cieliebak, K. Mohnke, Compactness for punctured holomorphic curves,  \textit{J. Symplectic Geom.}, 
\textbf{3}(2005), 589--654.

\bibitem[CH]{crh}
D. Cristofaro-Gardiner, M. Hutchings, From one Reeb orbit to two 
arXiv:1202.4839, 13 pages, to appear in \textit{J. Diff. Geom.}.

\bibitem[CW]{cw} C.B, Croke, A. Weinstein, Closed curves on convex hypersurfaces and periods of nonlinear oscillations, \textit{Inventiones mathematicae}
\textbf{64} (1981), 199--202.


\bibitem[EL]{el}
I. Ekeland,  J.-M. Lasry, On the Number of Periodic Trajectories for a Hamiltonian Flow on a Convex Energy Surface, \textit{Annals of Mathematics}, \textbf{112} (1980), 283--319.


\bibitem[EHS]{ehs} Y. Eliashberg, H. Hofer, D. Salamon, Lagrangian intersections in Contact geometry,  \textit{Geometric \& Functional Analysis}  \textbf{5} (1995), 244--269. 

\bibitem[EG]{eg} J. Etnyre and R. Ghrist, Contact topology and hydrodynamics III: knotted orbits, \emph{Trans. Amer. Math. Soc.}, \textbf{ 352} (2000), 5781--5794.

\bibitem[Gi]{gi} V.L. Ginzburg, An embedding $\mathbb{S}^{2n-1} \to \R^{2n}$, $2n-1\geq7$, whose Hamiltonian flow has no periodic trajectories, \emph{IMRN}, \textbf{2} (1995), 83--98.

\bibitem[GGM]{ggm} 
V.L. Ginzburg, B. G\"urel, L. Macarini, On the Conley conjecture for Reeb flows,  to appear in \textit{Internat. J. Math}.

\bibitem[Gu1]{gu1} 
J. Gutt, On the minimal number of periodic Reeb orbits on a contact manifold, PhD thesis, June 2014.

\bibitem[Gu2]{gu2} 
J. Gutt, The positive equivariant symplectic homology as an invariant for some contact manifolds, preprint arXiv:1503.01443

 \bibitem[GK]{gk}
J. Gutt, J. Kang, On the minimal number of periodic orbits on some hypersurfaces in $\mathbb{R}^{2n}$, arXiv:1508.00166.


\bibitem[HS]{hs}
H. Hofer, D. Salamon, 
Floer homology and Novikov rings, in \emph{The Floer memorial volume},
483--524, Progr.\ Math., 133, Birkh\"auser, Basel, 1995.

\bibitem[HV]{hv} 
H. Hofer, C. Viterbo, The Weinstein conjecture in cotangent bundles and related results, \textit{Annali della Scuola Normale Superiore di Pisa} - \textit{Classe di Scienze}, \textbf{15} (1988), 411--445.


\bibitem[HWZ]{hwz} 
H. Hofer, K. Wysocki, E. Zehnder, Unknotted periodic orbits for Reeb flows on the three-sphere, \emph{Topological Methods in Nonlinear Analysis}, \textbf{7} (1996), 219--244.

\bibitem[Le]{le}
E. Lerman, Contact cuts, \textit{Israel Journal of Mathematics},  \textbf{124} (2001), 77--92.

 

\bibitem[Ka]{ka}
A.B. Katok, Ergodic properties of degenerate integrable Hamiltonian systems, \textit{Izv. Akad. Nauk. SSSR}, \textbf{37} (1973), [Russian], 535--571.

\bibitem[Ke]{ke}
E. Kerman,
Hofer's Geometry and Floer Theory under the Quantum Limit
\textit{Int. Math. Res. Notices.}, (2008) Vol. 2008 doi:10.1093/imrn/rnm137.

\bibitem[LZ]{lz}
Y. Long, C. Zhu, Closed characteristics on compact convex hypersurfaces in $\R^{2n}$, \textit{Annals of Mathematics},  \textbf{155} (2002), 317--368.

\bibitem[Ra]{ra} H.-B. Rademacher, Existence of closed geodesics on positively curved Finsler manifolds, 
\textit{Erg. Th. \& Dyn. Syst.}, \textbf{27} (2007), 251--260.

\bibitem[Re]{re} A. Rechtman, Use and disuse of plugs in foliations., Ph.D. Thesis at \'{E}cole Normal Sup\'{e}rieure de Lyon (2009).

\bibitem[Su]{su} D. Sullivan, A foliation of geodesics is characterized by having no "tangent homologies", \emph{J. Pure Appl. Algebra}, \textbf{13} (1978), 101--104.


\bibitem[Wa]{wa6} W. Wang, Non-hyperbolic closed geodesics on Finsler spheres, 
\textit{J. Differential Geom.},  \textbf{99} (2015), 473--496.

\bibitem[We1]{we1} A. Weinstein, Symplectic manifolds and their Lagrangian submanifolds, \emph{Adv. in Math.}, \textbf{6} (1971), 329--346.

\bibitem[We2]{we2} A. Weinstein, Contact surgery and symplectic handlebodies, preprint 1990.


\end{thebibliography}
\end{document}